\documentclass[]{amsart}

\usepackage{amsthm, amsmath, amssymb, mathtools}
\usepackage{microtype}
\usepackage{overpic, graphicx, color}
\usepackage[dvipsnames]{xcolor}
\usepackage{import}
\usepackage{enumerate}
\usepackage{tikz-cd}
\usepackage[hidelinks,pagebackref,pdftex]{hyperref}
\usepackage{booktabs}

\AtBeginDocument{%
   \def\MR#1{}
}

\usepackage{marginnote}
\long\def\@savemarbox#1#2{\global\setbox#1\vtop{\hsize\marginparwidth 
  \@parboxrestore\tiny\raggedright #2}}
\renewcommand*{\backref}[1]{}
\renewcommand*{\backrefalt}[4]{
  \ifcase #1
  [No citations.]
  \or [#2]
  \else [#2]
  \fi }

\theoremstyle{plain}
\newtheorem{theorem}{Theorem}
\numberwithin{theorem}{section}
\newtheorem{corollary}[theorem]{Corollary}
\newtheorem{lemma}[theorem]{Lemma}

\newtheorem{proposition}[theorem]{Proposition}

\newtheorem{claim}[theorem]{Claim}

\newtheorem*{namedtheorem}{\theoremname}
\newcommand{\theoremname}{testing}

\theoremstyle{definition}
\newtheorem{definition}[theorem]{Definition}
\newtheorem{remark}[theorem]{Remark}

\numberwithin{equation}{section}

\newcommand{\refthm}[1]{Theorem~\ref{Thm:#1}}
\newcommand{\reflem}[1]{Lemma~\ref{Lem:#1}}
\newcommand{\refprop}[1]{Proposition~\ref{Prop:#1}}
\newcommand{\refcor}[1]{Corollary~\ref{Cor:#1}}
\newcommand{\refrem}[1]{Remark~\ref{Rem:#1}}
\newcommand{\refclaim}[1]{Claim~\ref{Claim:#1}}

\newcommand{\refeqn}[1]{\eqref{Eqn:#1}}
\newcommand{\refitm}[1]{\eqref{Itm:#1}}
\newcommand{\refdef}[1]{Definition~\ref{Def:#1}}
\newcommand{\refsec}[1]{Section~\ref{Sec:#1}}
\newcommand{\reffig}[1]{Figure~\ref{Fig:#1}}

\newcommand{\HH}{{\mathbb{H}}}
\newcommand{\RR}{{\mathbb{R}}}
\newcommand{\ZZ}{{\mathbb{Z}}}
\newcommand{\NN}{{\mathbb{N}}}
\newcommand{\CC}{{\mathbb{C}}}

\newcommand{\PP}{{\mathbb{P}}}

\newcommand{\calC}{\mathcal{C}}
\newcommand{\calT}{\mathcal{T}}

\newcommand{\calA}{\mathcal{A}}
\newcommand{\calL}{\mathcal{L}}
\newcommand{\calE}{\mathcal{E}}

\newcommand{\calN}{\mathcal{N}}

\newcommand{\from}{\colon} 
\newcommand{\cross}{\times}

\newcommand{\bdy}{\partial}

\newcommand{\vol}{\operatorname{vol}}
\newcommand{\PSL}{\operatorname{PSL}}

\newcommand{\area}{\operatorname{area}}
\newcommand{\len}{\operatorname{len}}

\newcommand{\dhyp}{{d_{\mathrm{hyp}}}}

\newcommand{\Rmin}{{R_{\min}}}

\newcommand{\zmin}{{z_{\min}}}

\renewcommand{\ss}{{\mathbf{s}}}

\newcommand{\injrad}{{\operatorname{injrad}}}

\newcommand{\haze}{{\operatorname{haze}}}
\newcommand{\ERROR}{{10^{-5}}}

\newcommand{\Isom}{{\operatorname{Isom}}}

\newcommand{\Mod}{{\operatorname{Mod}}}
\newcommand{\MC}{{\mathcal{MC}}}
\newcommand{\ML}{{\mathcal{ML}}}
\newcommand{\PML}{{\mathcal{PML}}}

\DeclareMathOperator{\arctanh}{arctanh}

\title
{Effective drilling and filling of tame hyperbolic 3-manifolds}

\author[D.~Futer]{David Futer}
\address[]{Department of Mathematics, Temple University,
Philadelphia, PA 19122, USA}
\email[]{dfuter@temple.edu}

\author[J.~Purcell]{Jessica S.~Purcell}
\address[]{School of Mathematics, Monash University, VIC 3800, Australia }
\email[]{jessica.purcell@monash.edu}

\author[S.~Schleimer]{Saul Schleimer}
\address[]{Department of Mathematics, 
University of Warwick, Coventry CV4 7AL, UK}
\email[]{s.schleimer@warwick.ac.uk}

\makeatletter
\@namedef{subjclassname@2020}{%
  \textup{2020} Mathematics Subject Classification}
\makeatother
\subjclass[2020]{57K32, 30F40}
\thanks{\today}

\begin{document}

\begin{abstract}
We give effective bilipschitz bounds on the change in metric between thick parts of a cusped hyperbolic 3-manifold and its long Dehn fillings. In the thin parts of the manifold, we give effective bounds on the change in complex length of a short closed geodesic. These results quantify the filling theorem of Brock and Bromberg, and extend previous results of the authors from finite volume hyperbolic 3-manifolds to any tame hyperbolic 3-manifold. To prove the main results, we assemble tools from Kleinian group theory into a template for transferring theorems about finite-volume manifolds into theorems about infinite-volume manifolds. We also prove and apply an infinite-volume version of the 6-Theorem.
\end{abstract}

\maketitle

\section{Introduction}

Thurston's celebrated hyperbolic Dehn surgery theorem says that almost all Dehn fillings of a cusped hyperbolic 3-manifold produce closed hyperbolic 3-manifolds. The other direction is also true: drilling a closed geodesic from a hyperbolic 3-manifold produces another hyperbolic 3-manifold, with a cusp \cite{agol:drilling}. These original results provide the existence of a hyperbolic metric but do not construct it. Hodgson and Kerckhoff's subsequent work \cite{hk:ConeRigidity, hk:univ, hk:shape} produces a way to continuously interpolate between the drilled and filled manifolds via a family of manifolds with cone singularities, in a process called \emph{cone deformation}. Their work provides analytic control over quantities such as volume. 

Bromberg extended the theory of cone deformations to infinite-volume hyperbolic 3-manifolds~\cite{bromberg:conemflds}. Brock and Bromberg further proved bilipschitz estimates on the change in geometry for such manifolds~\cite{brock-bromberg:density}. Their results are \emph{uniform}, in the sense that the change in geometry is controlled by constants independent of the manifold. However, they are not \emph{effective}, in the sense that the constants are not explicitly given. For instance, Brock and Bromberg's drilling theorem is as follows.

\begin{theorem}[Drilling theorem, \cite{brock-bromberg:density}]\label{Thm:BBDrillingThm}
Fix $J > 1$ and $\epsilon > 0$, where $\epsilon$ is smaller than the Margulis constant $\epsilon_3$.  Then there is a number $\ell_0 = \ell_0(\epsilon, J) > 0$ such that the following holds for every geometrically finite hyperbolic 3-manifold $Y$ without rank-1 cusps. 
Suppose that $\Sigma \subset Y$ is a geodesic link, whose total length is less than $\ell_0$. Then $Y - \Sigma$ admits a hyperbolic structure $Z$ with the same end invariants as those of $Y$.
Furthermore, the inclusion
\[ 
\iota \from Z \hookrightarrow Y
\]
restricts to a $J$--bilipschitz diffeomorphism on the complement of $\epsilon$--thin tubes about $\Sigma$. 
\end{theorem}

Uniform results such as \refthm{BBDrillingThm} are very useful for studying convergent sequences of hyperbolic manifolds. Indeed, Brock and Bromberg's application was a special case of the density conjecture (compare \refthm{Density}). On the other hand, ineffective results such as \refthm{BBDrillingThm} are hard to apply in the study of individual manifolds. 

In recent years, there has been a major push to make geometric estimates effective. For instance, in a previous paper, we prove an effective version of \refthm{BBDrillingThm} in the special case of finite-volume manifolds~\cite[Theorem 1.2]{FPS:EffectiveBilipschitz}. Among other applications, effective geometric estimates can be used to control Margulis numbers and rule out cosmetic surgeries.

Our main result in this paper is an effective version of  \refthm{BBDrillingThm} for all tame hyperbolic 3-manifolds. A 3-manifold $Y$ is called \emph{tame} if it is the interior of a compact 3-manifold with boundary. By the tameness theorem, due to Agol~\cite{Agol:Tameness} and Calegari and Gabai~\cite{CalegariGabai:Tameness}, a hyperbolic 3-manifold $Y$ is tame if and only if $\pi_1(Y)$ is finitely generated; see \refthm{Tameness}. Thus our results apply to all hyperbolic 3-manifolds with finitely generated fundamental group.

\begin{theorem}[Effective drilling, tame manifolds]\label{Thm:EffectiveDrillingTame}
Let $Y$ be a tame hyperbolic 3-manifold. Fix any $0<\epsilon \leq \log 3$ and any $J>1$. Let $\Sigma$ be a geodesic link in $Y$ whose total length $\ell$ satisfies
\[ \ell < \frac{1}{4} \min\left\{ \frac{\epsilon^5}{6771 \cosh^5(0.6 \epsilon + 0.1475)}, \, \frac{\epsilon^{5/2}\log(J)}{11.35} \right\}. \]
Then $Y - \Sigma$ admits a hyperbolic structure $Z$ with the same end invariants as $Y$. Moreover, there are $J$--bilipschitz inclusions
\[
\varphi \from Y^{\geq \epsilon} \hookrightarrow Z^{\geq \epsilon/1.2}, 
\qquad
\psi \from Z^{\geq \epsilon} \hookrightarrow Y^{\geq \epsilon/1.2}.
\]
\end{theorem}

We remark that existence of the hyperbolic metric $Z$ does not need any numerical hypotheses; see \reflem{DrillHyperbolic}. Those hypotheses are only needed for the $J$--bilipschitz conclusion.

\refthm{EffectiveDrillingTame} is stronger than Brock and Bromberg's \refthm{BBDrillingThm} in two respects and weaker in one respect.
Most notably, the hypotheses and conclusion of \refthm{EffectiveDrillingTame} are completely effective. 
In addition, \refthm{EffectiveDrillingTame} applies to manifolds with 
rank-1 cusps and geometrically infinite ends, which are excluded by the hypotheses of \refthm{BBDrillingThm}. 
On the other hand, \refthm{BBDrillingThm} extends its bilipschitz control into the thin parts of $Y$ that do not correspond to components of $\Sigma$; this extension is not present in \refthm{EffectiveDrillingTame}. In \refthm{LenBoundDownInfinite} below, we provide some geometric control in the thin parts of $Y$ by estimating the change in complex length of the core geodesics.

We also prove a version of \refthm{EffectiveDrillingTame} with hypotheses on the drilled manifold $Z$ rather than the filled manifold $Y$. That result requires the following definition. If $s$ is a slope on a rank-2 cusp $C$, the \emph{normalized length} of $s$ is
\[ L(s) = \frac{\len(s)}{\sqrt{\area(\bdy C)}}, \]
where $\len(s)$ denotes the length of a Euclidean geodesic representative of $s$ on $\bdy C$, and $\area(\bdy C)$ is the area in the induced Euclidean metric on $\bdy C$. The quantity $L(s)$ is scaling-invariant, hence does not depend on the choice of cusp neighborhood $C$.
Next, suppose we have a tuple of slopes $\ss = (s_1, \dots, s_k)$ on rank-2 cusps $C_1, \dots, C_k$, respectively. The \emph{total normalized length} $L = L(\ss)$ is defined by
\begin{equation}\label{Eqn:NormalizedLength}
\frac{1}{L(\ss)^2} = \sum_{j=1}^k \frac{1}{L(s_j)^2}.
\end{equation}

In \cite[Corollary~9.34]{FPS:EffectiveBilipschitz}, we prove effective bounds on the total normalized length $L$ that guarantee $J$--bilipschitz inclusions similar to those of  \refthm{EffectiveDrillingTame}.
We can now generalize that result to all tame hyperbolic 3-manifolds. 

\begin{theorem}[Effective filling, tame manifolds]\label{Thm:EffectiveFillingTame}
Fix any $0 < \epsilon \leq \log 3$ and any $J>1$. Let $M$ be a tame 3-manifold and $\Sigma \subset M$ a link, such that $M-\Sigma$ admits a hyperbolic structure $Z$. 
Suppose that, in the hyperbolic structure $Z$ on $M-\Sigma$, the total normalized length of the meridians of $\Sigma$ satisfies
\begin{equation*}
L^2 \geq 4 \, \max \left\{ \frac{2\pi \cdot 6771 \cosh^5(0.6 \epsilon + 0.1475)}{\epsilon^5} + 11.7, \:
\frac{2\pi \cdot 11.35}{\epsilon^{5/2}\log(J)} + 11.7 \right\}.
\end{equation*}
Then $M$ admits a hyperbolic structure $Y$ with the same end invariants as those of $Z$,  in which $\Sigma$ is a geodesic link. 
Moreover, there are $J$--bilipschitz inclusions
\[
\varphi \from Y^{\geq \epsilon} \hookrightarrow Z^{\geq \epsilon/1.2}, 
\qquad
\psi \from Z^{\geq \epsilon} \hookrightarrow Y^{\geq \epsilon/1.2}.
\]
\end{theorem}

The proofs of Theorems~\ref{Thm:EffectiveDrillingTame} and~\ref{Thm:EffectiveFillingTame} rely on three major recent results in Kleinian groups, recalled in \refsec{Kleinian}. The first of these is the ending lamination theorem, due to Minsky~\cite{Minsky:ModelsBounds} and Brock, Canary, and Minsky~\cite{BrockCanaryMinsky:ELC}, with an alternate proof by Bowditch~\cite{Bowditch:ELC} that includes the result for compressible ends. The second is the tameness theorem, 
due to Agol~\cite{Agol:Tameness} and Calegari and Gabai~\cite{CalegariGabai:Tameness}.
The third and most directly relevant result is the density theorem, which asserts that geometrically finite hyperbolic manifolds are dense in the space of all tame hyperbolic 3-manifolds.
The proof of the density theorem was concluded independently by Ohshika~\cite{ohshika:density} and Namazi and Souto~\cite{namazi-souto:density}, relying on many previous results including the tameness and ending lamination theorems.
These major results enable us to transfer bilipschitz control  from finite-volume manifolds to infinite-volume manifolds, by expressing the infinite-volume manifolds as geometric limits of finite-volume ones. See \refsec{Bootstrapping} and specifically \refthm{FiniteVolApproxCor} for a template describing this transfer process.

Because we use finite-volume manifolds to approximate infinite-volume ones, 
 the effective hypotheses on length in Theorems~\ref{Thm:EffectiveDrillingTame} and~\ref{Thm:EffectiveFillingTame} are directly derived from the corresponding hypotheses in the finite-volume setting. See Theorems~\ref{Thm:BilipFiniteVolume} and~\ref{Thm:EffectiveFillFiniteVolume}, and compare \refrem{FactorOf4}.
In turn, the constants appearing in Theorems~\ref{Thm:BilipFiniteVolume} and~\ref{Thm:EffectiveFillFiniteVolume} were obtained by bounding a number of analytic quantities appearing in a cone deformation. The given hypotheses on length ensure the existence of a cone deformation with a large embedded tube about the singular locus $\Sigma$, and provide control on the infinitesimal change in geometry. See  \cite[Section 1.3]{FPS:EffectiveBilipschitz} for a detailed and relatively non-technical summary of the argument. We  have no reason to believe the constants in our theorems are sharp. However, they are uniform and effective, so we expect the results to be useful.

If the manifolds $Y$ and $Z$ in Theorems~\ref{Thm:EffectiveDrillingTame} and~\ref{Thm:EffectiveFillingTame} are geometrically finite, those theorems can be proved using far less machinery.  We only need Brooks' work on circle packings~\cite{Brooks}, classical results from Ahlfors--Bers theory, and our Theorems~\ref{Thm:BilipFiniteVolume} and~\ref{Thm:EffectiveFillFiniteVolume} for finite-volume manifolds. See Remarks~\ref{Rem:GFDrillingApproximation} and~\ref{Rem:GFFillingApproximation} for details. 
 The geometrically finite case of \refthm{EffectiveDrillingTame}, which is an effective version of Brock and Bromberg's \refthm{BBDrillingThm}, can be used to prove a case of the density theorem; this was its main application in \cite{brock-bromberg:density}. In the converse direction, the density theorem allows us to extend \refthm{BBDrillingThm} from geometrically finite manifolds to any tame 3-manifold without rank-1 cusps; see \refthm{GeoInfiniteBB}. 

Our proof of \refthm{EffectiveFillingTame} also requires a version of the 6--theorem for tame manifolds, which is likely of some independent interest, although probably not surprising to experts. The original 6--theorem, due to Agol~\cite{agol:6theorem} and Lackenby~\cite{lackenby:surgery}, states that Dehn filling a finite volume 3-manifold along a slope of length greater than $6$ yields a manifold that admits a hyperbolic structure. Their result can be generalized as follows.

\begin{theorem}[6--Theorem for tame manifolds]\label{Thm:6ThmIntro}
Let $Z$ be a tame hyperbolic 3-manifold, with parabolic locus $P\cup(T_1\cup \dots\cup T_k)$, where $T_1, \dots, T_k$ are a subcollection of the torus ends of $Z$.
Let $H_1 \cup \dots \cup H_k$ be pairwise disjoint horocusps, with $H_i$ a neighborhood of $T_i$.
Let $\ss = (s_1, \dots, s_k)$ be a tuple of slopes, such that the length of a geodesic representative of each $s_i$ on $\bdy H_i$ is strictly greater than $6$. Then the manifold $Z(\ss) = Z(s_1, \ldots, s_k)$ obtained from $Z$ by Dehn filling along slopes $s_1, \dots, s_k$ admits a hyperbolic structure $Y$ with end invariants identical to those of $Z$.
\end{theorem}

The proof of \refthm{6ThmIntro} for infinite-volume manifolds closely parallels part of the proofs by Agol and Lackenby for finite-volume manifolds~\cite{agol:6theorem, lackenby:surgery}. In both arguments, one has to show that the filled manifold $Z(s_1, \ldots, s_k)$ does not contain any embedded surfaces that would obstruct hyperbolicity. The proof of the infinite-volume case also uses the above-mentioned major recent results in Kleinian groups, particularly the density theorem.
We give the argument in \refsec{6Theorem}.

Although we do not have full control on the change in geometry in the thin part of the manifold $Y$,
 we do have results that bound the change in length of the short geodesics that lie in at the cores of the thin part. The following result is an effective version of a theorem of Bromberg~\cite[Theorem 1.4]{bromberg:conemflds} and an extension of~\cite[Corollary~7.20]{FPS:EffectiveBilipschitz} to the infinite-volume case.

\begin{theorem}\label{Thm:LenBoundDownInfinite}
Let $Y$ be a tame hyperbolic 3-manifold.
Let $\Sigma$ be a geodesic link in $Y$, and $\gamma$ a closed geodesic disjoint from $\Sigma$ with complex length $\len_Y(\gamma)+i\tau_{Y}(\gamma)$, where $\len_Y(\cdot)$ denotes the length in the complete metric on $Y$.
Suppose that $\max(4\len_Y(\Sigma), \len_Y(\gamma)) < 0.0735$. 

Then 
$Y - \Sigma$ admits a hyperbolic structure $Z$ with with the same end invariants as those of $Y$. Furthermore,
$\gamma$ is isotopic to a geodesic in $Z$, whose complex length $\len_{Z}(\gamma)+i\tau_{Z}(\gamma)$ satisfies
\[
1.9793^{-1} \leq \frac{\len_{Z}(\gamma)}{\len_{Y}(\gamma)} \leq 1.9793
\qquad \text{and} \qquad
|\tau_{Z}(\gamma) - \tau_{Y}(\gamma) | \leq 0.05417.
\]
\end{theorem}

We will prove \refthm{LenBoundDownInfinite} in \refsec{ShortGeodesic} as a corollary of~\refthm{ShortDrillingTame}, which provides explicit bounds on the change in complex length of $\gamma$ between $Y$ and $Z$, as a function of the real lengths $\len_Y(\Sigma)$ and $ \len_Y(\gamma)$.
If we hold $\len_Y(\gamma)$ fixed, we find that the change in complex length of $\gamma$
tends to $0$ as $\len_Y(\gamma) \to 0$. Thus, when the geodesic link $\Sigma$ is very short, the geometry of $\gamma$ barely changes at all under drilling.

In an analogous fashion, we prove a result that bounds the change in complex length of a short geodesic $\gamma$ under filling, with hypotheses that use the geometry of the drilled manifold $Z$ rather than the filled manifold $Y$. See \refthm{ShortFillingTame} for a general result that provides a bound as a function of $\len_{Z}(\gamma)$ and the normalized length $L(\ss)$ in $Z$, with the change in complex length 
 tending to $0$ as $L(\ss) \to \infty$. See also \refcor{LenBoundUpInfinite} for a simple statement akin to \refthm{LenBoundDownInfinite}.

\subsection{Bootstrapping from finite-volume manifolds to tame manifolds}\label{Sec:Bootstrapping}
The proofs of Theorems~\ref{Thm:EffectiveDrillingTame} and~\ref{Thm:EffectiveFillingTame} both start with the analogous result for finite-volume manifolds in \cite{FPS:EffectiveBilipschitz}.
Using a strong version of density theorem due to Namazi and Souto~\cite{namazi-souto:density} (compare \refthm{Density}), we approximate any tame hyperbolic 3-manifold by a sequence of geometrically finite hyperbolic 3-manifolds. These geometrically finite manifolds, in turn, can be perturbed slightly to obtain manifolds that admit circle packings on their ends, by work of Brooks~\cite{Brooks}. Finally, manifolds that admit circle packings on their ends have convex cores embedding isometrically in finite volume manifolds, by a process of ``scooping'' and ``doubling'' (see \refdef{DoubleDouble}). At this point, the results for finite-volume manifolds can be applied. By taking better and better finite-volume approximates, we obtain the desired results for any tame manifold. 

See \reffig{DrillingApproxOutline} for a diagram summarizing the above process. In that figure, $DD(V_n^\circ)$ and $DD(W_n^\circ)$ are finite-volume hyperbolic manifolds obtained by the doubling process, to which we can apply the results of \cite{FPS:EffectiveBilipschitz}. The construction depicted in \reffig{DrillingApproxOutline} can also be summarized as follows:

\begin{theorem}\label{Thm:FiniteVolApproxCor}
Let $Y$ be a tame, infinite-volume hyperbolic 3-manifold. Let $\Sigma \subset Y$ be a geodesic link, such that each component $\sigma \subset \Sigma$ is shorter than $\log 3$. Then $Y - \Sigma$ admits a hyperbolic metric $Z$ with the same end invariants as those of $Y$. Furthermore, there is a sequence of finite-volume approximating manifolds $DD(V_n^\circ)$ and $DD(W_n^\circ)$ with the following properties:
\begin{enumerate}[\: \: $(1)$]
\item The manifold $DD(V_n^\circ)$ contains a geodesic link $DD(\Sigma_n)$, consisting of four isometric copies of a link $\Sigma_n$, such that $DD(W_n^\circ) = DD(V_n^\circ) - DD(\Sigma_n)$.
\item For any choice of basepoints $y \in Y$ and $z \in Z$, there are basepoints in the approximating manifolds such that $(DD(V_n^\circ), v_n) \to (Y, y)$ and $(DD(W_n^\circ), w_n) \to (Z, z)$ are geometric limits.
\item In the geometric limit $(DD(V_n^\circ), v_n) \to (Y, y)$, we have $\Sigma_n \to \Sigma$.
\end{enumerate}
\end{theorem}

See \refthm{DrillingApproximation} for a more detailed statement, of which \refthm{FiniteVolApproxCor} is a corollary. See also \reffig{FillingApproxOutline} and Theorems~\ref{Thm:FillingFteVolApproxCor} and \ref{Thm:FillingApproximation} for a parallel statement about filling rather than drilling.

Our hope is that  \refthm{FiniteVolApproxCor}, and the more detailed Theorems~\ref{Thm:DrillingApproximation} and~\ref{Thm:FillingApproximation}, can serve as user-friendly templates for transferring results about finite-volume manifolds to the infinite-volume setting. 
While the proofs of those results rely on the full machinery of Kleinian groups that will be described in \refsec{Kleinian}, knowledge of this machinery is not needed to \emph{apply} those theorems. We hope that this feature will make these templates useful to other researchers.

\subsection{Organization}
The paper is organized as follows.
In \refsec{Kleinian}, we review the tools from Kleinian groups that we need in this paper, particularly results on tame manifolds, their hyperbolic structures, and their limits. In \refsec{6Theorem}, we extend the 6-theorem to tame hyperbolic manifolds. Sections~\ref{Sec:ConvergenceOfLinks} and~\ref{Sec:Extracting} contain technical results that make the proofs of the main theorems work smoothly. In \refsec{ConvergenceOfLinks}, we show that under appropriate hypotheses, geodesics in approximating manifolds converge to geodesics in the limiting manifold. In \refsec{Extracting}, we prove that if there are two sequences of manifolds converging geometrically, and bilipschitz maps between the approximating manifolds, then one has bilipschitz maps between their limits as well. Then, we combine these technical results with the finite-volume bilipschitz theorems from \cite{FPS:EffectiveBilipschitz} to establish the effective drilling theorem in \refsec{Drilling} and the effective filling theorem in \refsec{Filling}. Finally, the results on short geodesics are proved in \refsec{ShortGeodesic}.

\subsection{Acknowledgements}
We thank Dick Canary for pointing us in the correct direction for results on strong limits, particularly \refthm{SameTopologyQH}. We thank Craig Hodgson and Steve Kerckhoff for helping us sort out some analytic details about maps between cone manifolds. We thank Juan Souto for insight into circle packings. 

During the course of this project, the first author was was partially supported by NSF grant DMS--1907708. 
The second author was partially supported by a grant from the Australian Research Council. 

\section{Tools from Kleinian groups}\label{Sec:Kleinian}

This section reviews a number of definitions and results from Kleinian groups that will be needed for our applications. As mentioned above, the proofs of our main theorems
use the full trifecta of major results in Kleinian groups from the early 2000s:  the tameness theorem, the ending lamination theorem, and the density theorem. We also review the (older) work of Brooks on circle packings, which will similarly prove crucial to our constructions.

Much of our exposition and notation is modeled on that of Namazi and Souto~\cite{namazi-souto:density}. Another excellent source that surveys these recent results is Canary \cite{Canary:TamenessSurvey}.

The tameness, ending lamination, and density theorems can be seen as results that relate the geometry of a hyperbolic 3-manifold to its topology. While each theorem has a succinct statement, we find it most useful to frame each result in the 
context, notation, and terminology that will be used for the applications. Setting up this notation and terminology requires a number of definitions.
We have endeavored to keep notation to a minimum, and to use consistent letters for parallel notions throughout the paper.

\subsection{Topology and Geometry}\label{Sec:TopologyGeometry}
Throughout the paper, the symbol $M$ denotes a compact 3-manifold with nonempty boundary, which is oriented, irreducible, and atoroidal. 

\begin{definition}[Pared manifolds]\label{Def:ParedManifold}
Let $M$ be a 3-manifold as above: compact, oriented, irreducible,   atoroidal, 
with $\bdy M \neq \emptyset$. We further assume that $M$ is neither a 3-ball nor a solid torus. Let $P\subset \bdy M$ be a compact subsurface consisting of incompressible tori and annuli. The pair $(M,P)$ is called a \emph{pared manifold} if the following additional conditions hold:
\begin{itemize}
\item Every $\pi_1$--injective map of a torus $T^2 \to M$ is homotopic to a map whose image is contained in $P$.
\item Every $\pi_1$--injective map of an annulus $(S^1\times I, S^1\times\bdy I)\to (M,P)$ is homotopic as a map of pairs to a map whose image is contained in $P$.
\end{itemize}
We call $P$ the \emph{parabolic locus} of $(M,P)$. The non-parabolic portion of $\bdy M$ is denoted $\bdy_0 M = \bdy M - P$.
\end{definition}

Throughout the paper, variants of the letter $M$ ($M$, $M'$, etc) always denote a 3-manifold defined only up to topological type. Similarly, variants of the letter $P$ denote the parabolic locus in a pared manifold. We will use variants of the letter $N$ ($N'$, $N_n$, etc) to denote a generic a 3-manifold endowed with a hyperbolic metric. In the context of drilling and filling, we will use variants of $V,W,Y, Z$ to denote 3-manifolds with hyperbolic metrics.

A \emph{Kleinian group} $\Gamma$ is a discrete group of isometries of $\HH^3$. For this paper, all Kleinian groups are presumed to be torsion-free and orientation-preserving, ensuring that the quotient $N = \HH^3 / \Gamma$ is an oriented hyperbolic manifold. 
All Kleinian groups are also assumed \emph{non-elementary}: this means that $\Gamma$ has no global fixed points on $\bdy \HH^3$ and implies that the topological type of $N$ is neither a solid torus nor the product of a torus and an interval.
If  $N$ is homeomorphic to the interior of $M$, we say that the Kleinian group $\Gamma$, abstractly isomorphic to $\pi_1(M)$,  endows $M$ with a \emph{hyperbolic structure}.

A \emph{horocusp} is the quotient $C = H/G$, where $H \subset \HH^3$ is an open horoball and $G$ is a discrete group of parabolic isometries of $H$, isomorphic to $\ZZ$ or $\ZZ^2$. In the first case, $C$ is homeomorphic to $A \times (0,\infty)$ where $A$ is a noncompact annulus, and is called \emph{rank 1}. In the second case, $C$ is homeomorphic to $T \times (0,\infty)$ where $T$ is a torus, and is called \emph{rank 2}. A \emph{horocusp in $N$} is an isometrically embedded (rank 1 or 2) horocusp in a hyperbolic 3-manifold $N$. A \emph{tube in $N$} is a regular neighborhood of a simple closed geodesic, of fixed radius. 

Given a constant $\epsilon > 0$, the \emph{thin part of $N$} is set of points in $N$ with injectivity radius less than $\epsilon/2$, denoted $N^{<\epsilon}$. A \emph{Margulis number} for a hyperbolic 3-manifold $N$ is any number $\epsilon>0$ such that $N^{< \epsilon}$ is a disjoint union of tubes and horocusps. The optimal Margulis number of $N$, denoted $\mu(N)$, is the supremum of its Margulis numbers. The \emph{Margulis constant} $\epsilon_3$ is the infimum of optimal Margulis numbers over all hyperbolic 3-manifolds. While it is known that $\epsilon_3 > 0$, the precise value is currently unknown. Since the Weeks manifold $W$ has $\mu(W) \leq 0.776$, it follows that $\epsilon_3 \leq 0.776$. Meyerhoff~\cite{meyerhoff} showed that $\epsilon_3 \geq 0.104$. 

In the setting of infinite-volume manifolds, we have a stronger estimate. The following result is due to Culler and Shalen~\cite[Theorem~9.1]{CullerShalen:Paradoxical}, combined with the tameness and density theorems. See Shalen~\cite[Proposition~3.12]{Shalen:SmallOptimalMargulis} for the derivation.

\begin{theorem}[Margulis numbers]\label{Thm:MargulisEstimate}
Let $N$ be a hyperbolic 3-manifold of infinite volume. Then the optimal Margulis number of $N$ satisfies $\mu(N) \geq \log 3$.
\end{theorem}

Several results in the paper assume a bound of the form $\epsilon \leq \log 3$. This can be viewed as ensuring $\epsilon$ is a Margulis number for both $N$ and the manifolds that will be used to approximate $N$.

\subsection{Tameness and compact cores}

Let $N$ be a hyperbolic 3-manifold, and suppose that $\epsilon$ is a Margulis number for $N$. Following Namazi and Souto, we let $N^\epsilon$ denote the complement in $N$ of the cusp components of $N^{<\epsilon}$. 
Then $\bdy N^{\epsilon}$ is a disjoint  union of tori and open annuli that satisfies the incompressibility requirements for the parabolic locus of a pared manifold. However, $(N^{\epsilon}, \bdy N^{\epsilon})$ is not a pared manifold because $N^{\epsilon}$ is typically not compact. The powerful tool that gives us a pared manifold from this data is the tameness theorem, proved  independently by Agol~\cite{Agol:Tameness} and by Calegari and Gabai~\cite{CalegariGabai:Tameness}.

\begin{theorem}[Tameness]\label{Thm:Tameness}
  Suppose $N$ is a hyperbolic 3-manifold with finitely generated fundamental group. Then $N$ is
 homeomorphic to the interior of a compact 3-manifold $M$. That is, $N$ is
   tame.
\end{theorem}

As a corollary of \refthm{Tameness}, we obtain:

\begin{corollary}[Standard compact cores]\label{Cor:StandardCore}
  Suppose $N$ is a hyperbolic 3-manifold with finitely generated fundamental group and let $0<\epsilon \leq \mu(N)$. Then there is a compact 3-manifold $M$ whose boundary $\bdy M$ contains a subsurface $P$, consisting of all toroidal components of $\bdy M$ and a possibly empty collection of annuli, such that $N^{\epsilon}$ is homeomorphic to $M - \bdy_0 M$. 
\end{corollary}

The pair $(M,P)$ produced by \refcor{StandardCore} is a pared manifold, unique up to pared homeomorphism, and independent of the choice of $\epsilon$. We say that $(M,P)$ is the \emph{pared manifold associated with $N$}. Note that $M - \bdy_0 M$ is homeomorphic to $N^{\epsilon}$, but is not viewed as a submanifold of $N^{\epsilon}$. A consequence of the tameness theorem is that there is a compact submanifold $(M',P')\subset (N^{\epsilon}, \bdy N^{\epsilon})$ that is homeomorphic to $(M,P)$. We call such a submanifold a \emph{standard compact core} of $(N^{\epsilon}, \bdy N^{\epsilon})$. 

The components $F_1, \dots, F_s$ of $\bdy_0 M = \bdy M - P$ are called the \emph{free sides} of $(M,P)$. Then
the submanifold $N^{\epsilon}-M'$ consists of $s$ different components called \emph{geometric ends}, each homeomorphic to $F_i\times (0,\infty)$ for some $i$. We sometimes refer to $F_i$ as an \emph{end} of $(M,P)$.
In the main case of interest, where $N$ has infinite volume, $\bdy M$ must have some non-torus boundary components, hence $(M,P)$ must contain at least one free side.

\subsection{End invariants and the ending lamination theorem}

Let $N=\HH^3/\Gamma$ be a tame hyperbolic 3-manifold. The \emph{limit set} of $\Gamma$, denoted $\Lambda_{\Gamma}$, is the set of accumulation points of an orbit $\{\Gamma x\}$ in $\bdy_{\infty} \HH^3 \cong S^2$. The \emph{convex core} $CC(N)=CH(\Lambda_{\Gamma})/\Gamma$ is the quotient by $\Gamma$ of the convex hull of the limit set.

The \emph{domain of discontinuity} of $\Gamma$, denoted $\Omega_{\Gamma}$, is $\bdy_{\infty}\HH^3-\Lambda_{\Gamma}$. The group $\Gamma$ acts properly discontinuously on $\HH^3 \cup \Omega_\Gamma$. The quotient of the action of $\Gamma$ on $\Omega_{\Gamma}$ gives a (possibly disconnected) surface with a conformal structure. This surface is the \emph{conformal boundary} of $N$.

Following \refcor{StandardCore}, let  $(M,P)$ be a pared manifold associated to $N$, and let $\calE \cong F \times (0,\infty)$ be a geometric end of $N^\epsilon$ associated to a free side $F \subset \bdy_0 M$. Then the geometric structure on $\calE$ endows $\calE$ and $F$ with  an \emph{end invariant}, as follows.

A geometric end $\calE \subset N^\epsilon$ is called \emph{geometrically finite} if it has a neighborhood whose intersection with the convex core $CC(N)$ is compact. The \emph{end invariant} of a geometrically finite end is the point in the Teichm\"uller space $\calT(F)$ determined by the component of the conformal boundary corresponding to $F$. If every end of $N$ is geometrically finite, we say $N$ is \emph{geometrically finite}.

If the end $\calE$ is not geometrically finite, then it is said to be \emph{degenerate}. In this case, the end comes equipped with a filling geodesic lamination $\lambda$ on the free side $F$. This lamination, called the \emph{ending lamination} of $\calE$, is the end invariant of $\calE$. 

The following theorem is due to Minsky~\cite{Minsky:ModelsBounds} and Brock--Canary--Minsky~\cite{BrockCanaryMinsky:ELC}. See also Bowditch~\cite{Bowditch:ELC} for an alternate proof that covers the case of compressible ends. 

\begin{theorem}[Ending Lamination]\label{Thm:ELC}
  Let $N$, $N'$ be tame hyperbolic 3-manifolds. Let $(M,P)$ and $(M',P')$ be standard compact cores of $N$ and $N'$, respectively.
Suppose there is a homeomorphism $\phi\from (M,P)\to (M',P')$ satisfying the following:
  \begin{itemize}
  \item If $F\subset \bdy M-P$ is a geometrically finite end of $N$, then $\phi(F)$ is a geometrically finite end of $N'$, and the induced map between conformal boundaries is homotopic to a bi-holomorphic map.
  \item If $F\subset \bdy M-P$ is a degenerate end with ending lamination $\lambda$, then $\phi(F)$ is a degenerate end of $N'$ with ending lamination $\phi(\lambda)$.
  \end{itemize}
  Then there is an isometry $\Phi\from N\to N'$, homotopic to $\phi$.
\end{theorem}

\subsection{Hyperbolization theorems}

We now review several results that guarantee that the topological pared manifold $(M,P)$ admits a hyperbolic structure with specified end invariants.

A hyperbolic structure on a 3-manifold $M$ defines a representation $\rho\from \pi_1(M) \to \PSL(2,\CC)$, and conjugate representations define isometric hyperbolic 3-manifolds. We let $AH(M,P)$ denote the set of conjugacy classes of discrete and faithful representations $\rho\from\pi_1(M)\to\PSL(2,\CC)$ such that those elements whose conjugacy classes are represented by loops on $P$ are mapped to parabolic elements.  The space $AH(M,P)$ is endowed with the topology of \emph{algebraic convergence}: a sequence of representations $\rho_n$ converges algebraically to $\rho$ if for all $\gamma \in \pi_1(M)$, the sequence $\{\rho_n(\gamma)\}$ converges to $\rho(\gamma)$ in $\PSL(2,\CC)$.  Similarly, a sequence $\{[\rho_n] \}$ converges algebraically to $[\rho] \in AH(M)$ if there are representatives $\rho_n \in [\rho_n]$ and $\rho\in[\rho]$ such that $\rho_n \to \rho$.
In a slight abuse of notation, we will write ``$\rho \in AH(M,P)$'' as a shorthand for the correct statement $\rho \in [\rho] \in AH(M,P)$.

A representation $\rho\in AH(M,P)$ is called \emph{minimally parabolic} if it satisfies the following property: $\rho(\gamma)\in \PSL(2,\CC)$ is parabolic if and only if $\gamma$ is conjugate into the fundamental group of some component of $P$. The following result, due to Thurston~\cite{Thurston:survey}, establishes the existence of at least one such representation.

\begin{theorem}[Hyperbolization]\label{Thm:Hyperbolization}
Let $(M,P)$ be a pared manifold. Then there is a geometrically finite hyperbolic 3-manifold $N$ such that the pared manifold associated to $N$ is $(M,P)$. Equivalently, there is a geometrically finite, minimally parabolic representation $\rho \in AH(M,P)$. 
\end{theorem}

The next classical result on hyperbolicity is a parametrization of the set of \emph{all} geometrically finite, minimally parabolic representations in $AH(M,P)$. 
Two representations $\rho \in AH(M,P)$ and $\rho' \in AH(M,P)$ are called \emph{quasi-conformally conjugated} if there is a quasi-conformal homeomorphism $f\from \bdy_{\infty}\HH^3\to \bdy_{\infty}\HH^3$ with $\rho(\gamma)\circ f = f\circ\rho'(\gamma)$ for all $\gamma\in \pi_1(M)$. Representations that are quasi-conformally conjugated to $\rho$ form an open subset $QH(\rho)$ of $AH(M,P)$. Moreover, if $\rho$ is minimally parabolic and geometrically finite, then so is $\rho'$, and the conformal boundary of $N'=\HH^3/\rho'(\pi_1(M))$ gives a point in Teichm\"uller space.

The statement below combines the work of
Ahlfors, Bers, Kra, Marden, Maskit, Mostow, Prasad, Sullivan, and Thurston. Our formulation is drawn from Namazi and Souto~\cite[Theorem~4.3]{namazi-souto:density}. See also Canary~\cite[Theorem 11.1]{Canary:TamenessSurvey}  and Canary and McCullough~\cite[Chapter~7]{CanaryMcCullough}.

\begin{theorem}[Ahlfors--Bers Uniformization]\label{Thm:ABUniformization}
  Let $(M,P)$ be a pared manifold. Let $\rho\in AH(M,P)$ be a minimally parabolic, geometrically finite representation, which exists by \refthm{Hyperbolization}. Then there is a covering map $\pi_{AB}$ from the Teichm\"uller space $\calT(\bdy_0 M)$ to $QH(\rho)\subset AH(M,P)$ with covering group $\Mod_0^+(M,P)$. Furthermore, for all $X\in \calT(\bdy_0 M)$, the hyperbolic manifold obtained from $\pi_{AB}(X)$ has associated pared manifold $(M,P)$ and conformal boundary bi-holomorphic to $X$.
\end{theorem}

If $\bdy_0 M = \emptyset$, then 
\refthm{ABUniformization} restates the Mostow--Prasad rigidity theorem that $QH(\rho)$ contains a single point.
In the main case of interest, if $F_1, \ldots, F_s$ are the free sides of $(M,P)$, then a tuple  
\[
X = (X_1, \dots, X_s) \: \in \: \calT(\bdy_0 M) = \calT(F_1) \times \cdots \times \calT(F_s)\]
 in Techm\"uller space  determines a conjugacy class $\pi_{AB}(X) \in AH(M,P)$. Any representative of this conjugacy class is called the \emph{Ahlfors--Bers representation} corresponding to $X$. The corresponding hyperbolic manifold has the chosen  points $(X_1, \dots, X_s)$ as the full collection of end invariants. 

The analogue of \refthm{ABUniformization} in the presence of degenerate ends is \refthm{Realization} below. Stating this result is harder, as it requires definitions involving projectively measured laminations. The following construction follows Thurston~\cite[pages~421--422]{thurston:geometry-dynamics}.

Let $F$ be a connected, oriented surface of finite type (for instance, the free side of a pared manifold).  Let $\calC = \calC(F)$ be the set of essential non-peripheral simple closed curves in $F$, considered up to isotopy. Given $\alpha, \beta \in \calC$, we define the \emph{geometric intersection number}  $\iota(\alpha, \beta)$ to be the minimal intersection number between isotopy representatives of $\alpha$ and $\beta$.  For any $\alpha$ we use $\iota_\alpha \from \calC \to \RR$ to denote the resulting function $\beta \mapsto \iota(\alpha, \beta)$. 

Recall that $\RR^\calC$ is a topological vector space over $\RR$, equipped with the product topology. 
In a minor abuse of notation, we write $\iota \from \calC \to \RR^\calC$ for the resulting injection.  We define $\MC(F) = \RR_{\geq 0} \cdot \iota(\calC)$ to be the subset of \emph{measured curves}; that is, functions of the form 
\[
r \iota_\alpha \quad \mbox{where} \quad r \in \RR_{\geq 0}, \quad \alpha \in \calC.
\]
We define $\ML(F)$ to be the space of \emph{measured laminations}; this is the closure of $\MC(F)$ inside of $\RR^\calC$.  We define $\PML(F)$ to be the image of $\ML(F)$ in the projectivization $\PP\RR^\calC$.  We end this review by noting the important fact, recorded by Bonahon~\cite{Bonahon:Currents}, that $\iota$ extends to give a homogenous, continuous function from $\ML(F) \cross \ML(F) \to \RR$. 

Now, let $F$ be a free side of a pared manifold $(M,P)$. A \emph{meridian} on $F$ is a simple closed curve $\alpha \in \calC(F)$ that bounds a disk in $M$. 
The \emph{Masur domain} of $F$ consists of all $\lambda \in \PML(F)$ such that $\iota(\lambda, \mu) \neq 0$ for every measured lamination $\mu \in \ML(F)$  that arises as a limit of measured meridians. See \cite[Section 6.1]{namazi-souto:density}.

\begin{definition}[Filling end invariants]\label{Def:FillingEndInvts}
Let $(M,P)$ be a pared manifold with free sides $F_1, \ldots, F_s$. Let $0 \leq r \leq s$. Consider a collection of end invariants $(X_1, \ldots, X_r, \lambda_{r+1}, \ldots, \lambda_s)$, where $X_i \in \calT(F_i)$ for $i \leq r$ and $\lambda_i$ is an ending lamination on $F_i$ for $i \geq r+1$. This collection of end invariants is called
\emph{filling} if it satisfies the following conditions:
\begin{enumerate}
\item[(*)] If $M$ is an interval bundle over a compact surface $S$ and $N$ has no geometrically finite ends, then the projection of the ending laminations to $S$ has transverse self-intersection.
\item[(**)] If a compressible component $F_i  \subset \bdy_0 M$ corresponds to a degenerate end, then the end invariant $\lambda_i$, equipped with some transverse measure, is a Masur domain lamination. Equivalently, $\lambda_i$ is not contained in the Hausdorff limit of any sequence of meridians.
\end{enumerate}
\end{definition}

Canary~\cite{Canary:Ends} proved that conditions (*) and (**) are necessary for the end invariants to be realized by a hyperbolic structure on $(M,P)$. Namazi and Souto proved that these conditions are also sufficient for being realized by a hyperbolic structure~\cite[Theorem~1.3]{namazi-souto:density}:

\begin{theorem}[Realization]\label{Thm:Realization}
  Let $(M,P)$ be a pared 3-manifold, with a collection of end invariants on the free sides of $(M,P)$. Then there exists a minimally parabolic representation $\rho\in AH(M,P)$ yielding a hyperbolic manifold $N_{\rho}=\HH^3/\rho(\pi_1(M))$ with the given end invariants if and only if the collection of end invariants is filling.
\end{theorem}

\subsection{Geometric and strong limits}

In addition to the algebraic topology on the space of hyperbolic 3-manifolds, we need to use another, finer topology. 
Let $\Gamma_n$ be a sequence of Kleinian groups. We say that $\Gamma_n$ \emph{converges geometrically} to $\Gamma$ if the groups converge in the Chabauty topology on closed subsets of $\PSL(2,\CC)$. 
Convergence in this topology can be characterized as follows:
\begin{itemize}
\item every $\gamma \in \Gamma$ is the limit of some sequence $\{\gamma_n\}$ with $\gamma_n \in \Gamma_n$; 
\item if $\gamma_n \to \gamma$ is a convergent sequence with $\gamma_n\in\Gamma_n$, then $\gamma \in \Gamma$. 
\end{itemize}
The Chabauty topology on Kleinian groups is metrizable \cite[Proposition 3.1.2]{ceg:notes-on-notes}. We use the notation $d_{\rm Chaub}$ to denote a conjugation-invariant metric inducing this topology.

We endow $\HH^3$ with an origin (denoted $0$) and an orthonormal frame at $0$. Then each quotient manifold $N_n = \HH^3 / \Gamma_n$ is endowed with a \emph{baseframe} $\omega_n$, namely the quotient of the fixed orthonormal frame at $0 \in \HH^3$.  Then $\Gamma_n$ and the pair $(N_n, \omega_n)$ determine one another. Changing $\Gamma_n$ by conjugation in $\PSL(2,\CC)$ keeps the quotient manifold the same up to isometry, but changes the baseframe. We emphasize that the Chabauty topology is a topology on Kleinian groups (not conjugacy classes), or equivalently a topology on the set of hyperbolic manifolds endowed with baseframes. 

Geometric convergence has the following intrinsic characterization. Let $(N_n, \omega_n)$ be a sequence of framed hyperbolic 3-manifolds. Let $(N, \omega)$ be another framed hyperbolic 3-manifold, where $\omega$ is a baseframe at $x \in N$. For $R >0$, let $B_R(x) \subset N$ be the metric $R$--ball in $N$ centered at $x$, meaning the
set of points in $N$ of distance less than $R$ from $x$.
Then $(N_n, \omega_n)$ \emph{converges geometrically} to $(N, \omega)$ if and only if, for every $R$, there are embeddings
\begin{equation}\label{Eqn:GromovHausdorff}
f_{n,R} \from (B_R(x), \omega) \hookrightarrow (N_n, \omega_n),
\end{equation}
for all $n$ sufficiently large, which converge to isometries in the $C^\infty$ topology as $n \to \infty$.
See~\cite[Theorem~3.2.9]{ceg:notes-on-notes} for the equivalence between geometric convergence $(N_n, \omega_n) \to (N, \omega)$ and convergence $\Gamma_n \to \Gamma$ in the Chabauty topology.

In practice, it is often sufficient to keep track of the points where frames are based. If $(N_n, x_n)$ is a sequence of hyperbolic 3-manifolds endowed with basepoints, and $(N, x)$ is another hyperbolic 3-manifold with a basepoint $x$, we say that $(N_n, x_n) \to (N, x)$ if, for every $R$, there are embeddings 
\begin{equation}\label{Eqn:PointedGromovHausdorff}
g_{n,R} \from (B_R(x), x) \hookrightarrow (N_n, x_n),
\end{equation}
for all $n$ sufficiently large, which converge to isometries in the $C^\infty$ topology as $n \to \infty$. Now, suppose we have chosen orthonormal frames $\omega_n$ at $x_n$ and $\omega$ at $x$. If there is a geometric limit $(N_n, \omega_n) \to (N, \omega)$, then of course $(N_n, x_n) \to (N, x)$ as well. Conversely, since the set of orthonormal frames at a given point is compact, a pointed limit $(N_n, x_n) \to (N, x)$ implies that there is a subsequence $n_i$ and a frame $\nu$ at $x$ such that $(N_{n_i}, \omega_{n_i}) \to (N, \nu)$. See~\cite[Lemma 3.2.8]{ceg:notes-on-notes}.

By a mild abuse of notation, we will say that $(N_n, x_n)$ \emph{converges geometrically} to $(N, x)$, meaning that there exists a choice of frames at $x_n$ and $x$ such that $(N_n, \omega_n) \to (N, \omega)$. In most of our limit arguments, basepoints will be important while frames will remain implicit.

Now, consider a sequence of discrete, faithful representations $\rho_n \from \pi_1(M) \to \PSL(2,\CC)$ with image groups $\Gamma_n$, and a representation $\rho$ with image group $\Gamma$. We say that $\rho_n$ \emph{converges strongly} to $\rho$ if $\rho_n \to \rho$ algebraically and $\Gamma_n \to \Gamma$ in the Chabauty topology. 

The question of whether algebraic and geometric limits agree is subtle, and has received extensive attention in the literature. 
See Marden~\cite[Chapter~4]{Marden:HyperbolicManifolds} for a survey. 
For our purposes, we will need only the following foundational statement.

\begin{theorem}[Same topology on $QH(\rho)$]\label{Thm:SameTopologyQH}
Let $(M,P)$ be a pared 3-manifold, and let $\rho \in AH(M,P)$ be a geometrically finite, minimally parabolic representation. Then the algebraic and geometric topologies on the open set $QH(\rho) \subset AH(M,P)$ agree. More precisely:
\begin{itemize}
\item\label{Itm:AlgToStrong} If $\rho_n$ is a sequence in $QH(\rho)$ such that $\rho_n \to \rho$ algebraically, then $\rho_n \to \rho$ strongly.
\item\label{Itm:GeomToStrong} If $\rho_n$ is a sequence in $QH(\rho)$ such that $\rho_n(\pi_1 (M)) \to \rho(\pi_1(M))$ geometrically, then $\rho_n \to \rho$ strongly.
\end{itemize}
\end{theorem}

The first bullet is due to Anderson and Canary~\cite[Theorem~3.1]{AndersonCanary}. The second bullet is due to J{\o}rgensen and Marden~\cite[Theorem~4.9]{JorgensenMarden}. See also Marden~\cite[Theorems~4.6.1 and~4.6.2]{Marden:HyperbolicManifolds}.

\subsection{Strong density and approximation theorems}

The density theorem, whose proof was concluded independently by Ohshika~\cite{ohshika:density} and Namazi and Souto~\cite{namazi-souto:density}, states that for every $\rho\in AH(M,P)$, there exists a geometrically finite sequence $\rho_n\in AH(M,P)$ converging algebraically to $\rho$.
The results we need, stated below as Theorems~\ref{Thm:Density} and~\ref{Thm:NamaziSouto}, are stronger. Stating these results requires the notion of a filling sequence.

\begin{definition}[Filling sequence in $\calT(\bdy M-P)$]\label{Def:FillingCocmpt}
Let $(M, P)$ be a pared manifold with free sides $F_1, \ldots, F_s$. As in \refdef{FillingEndInvts}, let $(X_1, \dots, X_r, \lambda_{r+1}, \dots, \lambda_s)$ be a filling tuple of end invariants,  where $X_1, \dots, X_r$ are points in Teichm\"uller space and $\lambda_{r+1}, \dots, \lambda_s$ are ending laminations.
 Suppose a sequence $(X_1^n, \dots, X_s^n)\in \calT(F_1)\times \dots \times \calT(F_s) = \calT(\bdy_0 M)$ satisfies:
\begin{enumerate}
\item for all $n$ and all $i\leq r$, we have $X_i^n = X_i$, and
\item for all $i\geq r+1$, there is a sequence of simple closed curves $\gamma_i^n \in \calC(F_i)$ which converge to $\lambda_i$ in $\PML(F_i)$, such that the ratio of lengths $\ell_{X_i^n}(\gamma_i^n) / \ell_{X_i^1}(\gamma_i^n)$ approaches zero.
\end{enumerate}
Then the sequence $(X_1^n, \dots, X_s^n)$ is said to be \emph{filling}.
\end{definition}

Let $\rho \in AH(M,P)$ be a discrete, faithful representation corresponding to $(X_1, \dots, X_r, \lambda_{r+1}, \dots, \lambda_s)$, with quotient manifold $N_\rho = \HH^3/\rho(\pi_1 M)$. Then, for $i > r$, the curves $\gamma_i^n$ appearing in item~(2)  define closed geodesics in the end of $N_{\rho}$ associated with $F_i$. These curves are said to \emph{exit the end} associated with $F_i$. This means that all but finitely many lie in the geometric end homeomorphic to $F_i\times (0,\infty)$, and for any compact set $K \subset N_{\rho}$, only finitely many $\gamma_i^n$ intersect $K$.

We can now state a strong form of the density theorem \cite[Corollary~12.3]{namazi-souto:density}.

\begin{theorem}[Strong density theorem]\label{Thm:Density}
Let $\Gamma$ be a finitely generated Kleinian group. Let $(M,P)$ be the pared manifold associated with $\HH^3 /\Gamma$, and let $\rho \from \Gamma \hookrightarrow PSL(2,\CC)$ be the inclusion map. Then there is a sequence of geometrically finite, minimally parabolic representations $\rho_n \in AH(M,P)$ converging strongly to $\rho$. Furthermore, the sequence of end invariants corresponding to $\rho_n$ is filling. 
\end{theorem}

Note that the ``furthermore'' sentence in our statement of \refthm{Density} is not stated directly in \cite[Corollary~12.3]{namazi-souto:density}. However, this assertion is central to the proof of \cite[Corollary~12.3]{namazi-souto:density}: the approximating manifolds appearing in Namazi and Souto's construction are taken to have ends forming a filling sequence. We also note that the proof of \refthm{Density} uses both the tameness and the ending lamination theorems.

The following related statement is \cite[Corollary~12.5]{namazi-souto:density}.

\begin{theorem}[Approximation theorem]\label{Thm:NamaziSouto}
  Let $(M,P)$ be a pared 3-manifold with free sides $F_1,\dots, F_s$. Suppose $(X_1^n, \dots, X_s^n) \in \calT(\bdy M-P)$ is a filling sequence converging to the filling end invariants $(X_1, \dots, X_r, \lambda_{r+1}, \dots, \lambda_s)$.
Let $\rho_n \in \pi_{AB}(X_1^n, \dots, X_s^n) \in AH(M,P)$ be the Ahlfors--Bers representation, giving an associated geometrically finite hyperbolic 3-manifold $N_n=\HH^3/\rho_n(\pi_1(M))$.

  Then, up to passing to a subsequence, $\rho_n$ converges strongly to a discrete and faithful representation $\rho$. If $(M',P')$ is the pared manifold associated with the hyperbolic manifold $N_{\rho}=\HH^3/\rho(\pi_1(M))$, then there is a homeomorphism $\phi\from (M,P)\to (M',P')$ in the homotopy class determined by $\rho$ which maps the filling tuple $(X_1, \dots, X_r, \lambda_{r+1}, \dots, \lambda_s)$ to the end invariants of $N_{\rho}$. 
\end{theorem}

\subsection{Circle packings}\label{Sec:CirclePacking}

Let $R$ be a Riemann surface of hyperbolic type, meaning that every component $R_i \subset R$ has $\chi(R_i)<0$. Let $\Omega$ be an open subset of the Riemann sphere $S^2$ that uniformizes $R$; that is, $R$ is a quotient of $\Omega$ by M\"obius transformations. A configuration of circles on $R$ (relative to $\Omega$) is a collection of simple closed curves on $R$ that bound disks, such that the interiors of the disks are disjoint, and the lifts of the curves to $\Omega$ are round circles on $S^2$. A configuration of circles is a \emph{circle packing} if the interstitial regions, complementary to the interiors of the disks, consist only of curvilinear triangles.

The following theorem is from Beardon and Stephenson~\cite[Theorem~6]{BeardonStephenson}.

\begin{theorem}[Uniformization theorem for circle packings]\label{Thm:CirclePackingRigidity}
Let $R$ be a Riemann surface that admits a circle packing. Then the circle packing uniquely determines a conformal structure on $R$. 
\end{theorem}

Now, let $\Gamma$ be a finitely generated Kleinian group, with $\Lambda(\Gamma)$ its limit set and $\Omega(\Gamma)$ its domain of discontinuity. Then $\Omega(\Gamma)/\Gamma$ is a (possibly disconnected) Riemann surface of hyperbolic type.

The following theorem follows from work of Brooks~\cite{Brooks}. 

\begin{theorem}[Circle packings approximate]\label{Thm:Brooks}
Let $N=\HH^3/\Gamma$ be a geometrically finite hyperbolic 3-manifold with associated pared manifold $(M,P)$, and let $\rho \from \pi_1(M) \to \Gamma$ be the associated representation. Then, for every $\delta >0$, there is a geometrically finite representation $\rho_\delta \in QH(\rho)$, representing an $e^\delta$--quasiconformal deformation of $\rho$, such that the conformal boundary $\Omega(\Gamma_\delta)/\Gamma_\delta$ of the image group $\Gamma_\delta$ admits a circle packing.
\end{theorem}

\begin{proof}[Proof sketch]
The ideas behind this statement are all contained in Brooks' proof of~\cite[Theorem~2]{Brooks}. If a component of $R=\Omega(\Gamma)/\Gamma$ is a closed surface $S$, then we may uniformize $S$ by a component of $\Omega(\Gamma)$. Pack circles into this component $S$, obtaining interstices that are triangles and quads. Brooks shows that for every $\delta > 0$, there is an $e^\delta$--quasiconformal deformation $\Gamma_\delta$ that eliminates the quads of a sufficiently fine packing, so the deformed conformal structure on $S$ admits a circle packing. 

When a component $S \subset R$ has cusps, the packing procedure requires some additional care. The argument is given, for example, in Hoffman and Purcell~\cite[Lemma~2.3]{Hoffman-Purcell}. Since $\Gamma$ is geometrically finite, there is a fundamental domain $\mathcal{F}$ for $\Gamma$ whose sides consist of finitely many geodesic hyperplanes. Where there is a rank--1 cusp, there will be two circular arcs $C_1, C_2$ bounding a polygonal region of the boundary of $\mathcal{F}$ such that $C_1$ and $C_2$ meet tangentially. Begin by adding two circles meeting orthogonally at that point of tangency, and ensure those circles are small enough that they meet no other sides of $\mathcal F$. Then fill in the rest of $\Omega(\Gamma)$ by circles, and use Brooks to perform a quasi-conformal deformation as above to obtain a circle packing.
\end{proof}

We close this section with a definition that will be useful for our constructions.

\begin{definition}[Scooped manifold, double--double of a circle packed manifold]\label{Def:DoubleDouble}
Let $N = \HH^3/\Gamma$ be a tame, geometrically finite hyperbolic manifold with associated pared manifold $(M,P)$, and suppose the conformal boundary associated with each free side of $\bdy M-P$ admits a circle packing. Let $\widetilde{C}$ denote the collection of inverse images in $\bdy_{\infty}\HH^3$ of the circles in the packing. Scoop out the half spaces in $\HH^3$ bounded by Euclidean hemispheres with boundary circles $\widetilde{C}$, and color their boundaries blue. The triangular interstices between circles of $\widetilde{C}$ uniquely determine additional circles dual to the blue ones. Scoop out the half spaces in $\HH^3$ bounded by these dual circles as well, and color their boundaries red. The resulting space, denoted $\widetilde{N}^{\circ}$, is $\HH^3$ with the interiors of red and blue hemispheres removed.

The group $\Gamma$ stabilizes $\widetilde{N}^{\circ}$. The quotient space $N^{\circ} = \widetilde{N}^{\circ}/\Gamma$ is a manifold with corners whose interior is homeomorphic to $N$, and whose boundary consists of geodesic blue ideal polygons and geodesic red ideal triangles. We call $N^{\circ}$  the \emph{scooped manifold} associated with $N$. 

Finally, starting with the scooped manifold $N^{\circ}$ associated with $N$, double first across the blue polygons, then double again across the red triangles. The result is a finite volume hyperbolic manifold with rank 2 cusps, which we will call the \emph{double--double of $N$}, and denote by $DD(N^{\circ})$. 
\end{definition}

\section{A 6--theorem for tame manifolds}\label{Sec:6Theorem}

The 6--theorem for finite volume 3-manifolds, proved by Agol~\cite{agol:6theorem} and Lackenby~\cite{lackenby:surgery}, states that Dehn filling along a slope of length greater than $6$ yields a manifold that admits a hyperbolic structure. In this section, we use many of the ideas behind their proof to extend their result to the infinite volume setting.

\begin{theorem}[6--theorem for pared manifolds]\label{Thm:6TheoremInf}
Let $Z$ be a tame hyperbolic 3-manifold, with associated pared manifold $(M,P\cup(T_1\cup \dots\cup T_k))$, where $T_1, \dots, T_k$ are torus boundary components of $M$. Assume that $\bdy M - (T_1\cup \dots\cup T_k) \neq \emptyset$.
Let $H_1 \cup \dots \cup H_k$ be pairwise disjoint horocusps, with $H_i$ a neighborhood of $T_i$.
Let $\ss = (s_1, \dots, s_k)$ be a tuple of slopes, such that the length of a geodesic representative of each $s_i$ on $\bdy H_i$ is strictly greater than $6$ for each $i$. 

Let $M(\ss) = M(s_1, \dots, s_k)$ denote the 3-manifold obtained by Dehn filling $M$ along the slopes $s_i$ on $T_i$. Then $(M(\ss), P)$ is a pared manifold, such that the free sides of $\bdy_0 M(\ss)$ are identical to those of $\bdy_0 M$. Furthermore, $(M(\ss), P)$ admits a hyperbolic structure $Y = Z(\ss)$ with end invariants identical to those of $Z$.
\end{theorem}

Observe that \refthm{6ThmIntro} follows immediately from \refthm{6TheoremInf} and prior work. If $\bdy M(\ss) = \emptyset$, then Agol and Lackenby's 6--theorem~\cite{agol:6theorem, lackenby:surgery}, combined with Perelman's geometrization theorem, says that $M(\ss) = M(s_1, \dots, s_k)$ admits a hyperbolic structure (with empty end invariants). Otherwise, if $\bdy M(\ss) \neq \emptyset$, \refthm{6TheoremInf} gives the desired conclusion. 
 
There is one case where we re-prove a portion of the original 6--theorem, by following the same line of argument: the case where $\vol(Z) < \infty$ and $P \neq \emptyset$. In this case, \reflem{6ThmParedMfld} combined with \refthm{Hyperbolization} shows that $M(\ss)$ admits a hyperbolic structure. 

Before beginning the proof of \refthm{6TheoremInf}, we record the following slight generalization  B\"or\"oczky's theorem on cusp density \cite{boroczky}.

\begin{lemma}\label{Lem:CuspDensitySurface}
Let $S$ be a hyperbolic surface with finite area, with a positive number of cusps, and with (possibly empty) boundary consisting of geodesics. Let $H \subset S$ be an embedded neighborhood of the cusps, such $H \cap \bdy S = \emptyset$ and $\bdy H$ is a disjoint union of horocycles. Then
\[
\area(S) \geq \frac{\pi}{3}\area(H).
\]
\end{lemma}

\begin{proof}
If $\bdy S = \emptyset$, this result is due to B\"or\"oczky \cite[Theorem 4]{boroczky}.

If $\bdy S \neq \emptyset$, let $DS$ be the complete hyperbolic surface obtained by doubling $S$ along its geodesic boundary. Since $H \cap \bdy S = \emptyset$, the double of $H$ is an embedded cusp neighborhood $DH \subset DS$, with $\bdy (DH)$ a disjoint union of horocycles. B\"or\"oczky's theorem says that $\area(DS) \geq \frac{\pi}{3}\area(DH)$, hence $\area(S) \geq \frac{\pi}{3}\area(H)$.
\end{proof}

The first step of the proof of \refthm{6TheoremInf} is to show that $(M(\ss), P)$ is a pared manifold. The main idea of the proof is drawn directly from the arguments of Agol~\cite{agol:6theorem} and Lackenby~\cite{lackenby:surgery}.

\begin{lemma}\label{Lem:6ThmParedMfld}
With notation and hypotheses as in \refthm{6TheoremInf}, the pair $(M(\ss), P)$ is a pared manifold.
\end{lemma}

\begin{proof}
Since $M$ is compact and oriented, the filling $M(\ss)$ is compact and oriented as well. Since $\bdy M - (T_1\cup \dots\cup T_k) \neq \emptyset$, it follows that $\bdy M(\ss) \neq \emptyset$.
In addition, $M(\ss)$ cannot be a 3-ball: otherwise, $\bdy M$ has a component $\Sigma \cong S^2$ and a separate component $T_1 \cong T^2$, which is impossible because $M$ is irreducible. The remaining obstructions to $(M(\ss), P)$ being a pared manifold are as follows.

\begin{claim}\label{Claim:SurfaceObstruction}
If $(M(\ss), P)$ is not a pared manifold, then there is an essential, embedded surface $(S, \bdy S) \to (M(\ss), P)$, where $S$ is a sphere, disk, torus, or annulus.
\end{claim}

To prove the claim, we check the remaining conditions of \refdef{ParedManifold}. If $M(\ss)$ is reducible or toroidal, then by definition it contains an essential sphere or torus.
If $M(\ss)$ is a solid torus, then $P \subset \bdy M$ must be a torus that becomes compressible in $M(\ss)$, hence $(M(\ss), P)$ contains an essential compression disk. If there is a $\pi_1$--injective map of an annulus $(S^1\times I, S^1\times\bdy I) \to (M(\ss),P)$, which is not boundary-parallel, then the annulus theorem says that there is also a $\pi_1$--injective embedding of an annulus with the same property. (See Jaco~\cite[Theorem VIII.13]{Jaco:Lectures} or Scott~\cite{Scott:AnnulusTorus}.) Finally, if there is a $\pi_1$--injective map of a torus $T^2 \to M(\ss)$, then observe that $M(\ss)$ is either reducible (hence the claim holds) or Haken. In the latter scenario, the torus theorem~\cite{Scott:AnnulusTorus} says that $M(s_1, \ldots, s_k)$ either contains an essential torus or annulus (hence the claim holds), or is Seifert fibered. But every Seifert fibered 3-manifold with nonempty boundary contains an essential disk or annulus, proving the claim.

\smallskip

Now, let $(S, \bdy S) \to (M(\ss), P)$ be an essential, embedded surface as in the claim. Assume that $S$ has been moved by isotopy to intersect the filling solid tori a minimal number of times. 
Then every component of intersection must be a meridian disk of some filling solid torus. Note that the intersection is nonempty, because $(M,P\cup T_1 \cup \dots \cup T_k)$ is a pared manifold by hypothesis. After removing the meridian disks from $S$, we obtain an essential surface $S' \subset M$ whose boundary consists of $\bdy S$ and a collection of curves on the tori $T_1, \dots, T_k$, with each boundary on $T_i$ having slope $s_i$.

Because $M$ is homeomorphic to a standard compact core of $(Z^{\epsilon}, \bdy Z^{\epsilon})$ for small $\epsilon>0$, we obtain an embedding of $S'$ into $Z^{\epsilon}$, which we extend by a product into the cusps, obtaining a punctured surface that we continue to call $S'$.

Using the hyperbolic metric $Z$, we homotope $S'$ to be a \emph{pleated surface}. This means that after a homotopy, $S'$ becomes an immersed surface consisting of totally geodesic ideal triangles, with bending allowed along the edges of the triangles. A homotopy that moves $S'$ into pleated form exists by \cite[Lemma~2.2]{lackenby:surgery}. As in Agol's and Lackenby's proof of the original 6--theorem (see \cite[Theorem~5.1]{agol:6theorem}, and compare \cite[Lemma~2.5]{futer-schleimer:cusp-geometry}), the horocusp $H = H_1 \cup \dots \cup H_k$ induces a disjoint union of horocycles in the pleated surface $S'$, such that the length of each horocycle is at least the length of the corresponding slope $s_i$, hence each has length strictly larger than $6$.

Let $m$ be the number of boundary components of $S'$ on components $T_1, \dots, T_k$. This is the number of horocycles in $S'$ constructed in the previous paragraph. The area of the cusp ends of $S'$ cut off by these horocycles is the sum of the lengths of the horocycles; hence, the total area is strictly larger than $6m$. By \reflem{CuspDensitySurface}, the area of $S'$ is at least $\pi/3$ times the area of the cusp neighborhood, hence $\area(S') > 2\pi m$.

On the other hand, by the Gauss--Bonnet theorem, $\area(S') = -2\pi\chi(S')$. Thus the area of $S'$ satisfies:
\[ \area(S') = \begin{cases}
  2\pi(m-2) & \mbox{ if $S$ is a sphere,} \\
  2\pi(m-1) & \mbox{ if $S$ is a disk,} \\
  2\pi(m) & \mbox{ if $S$ is a torus or annulus.} 
\end{cases}
\]
In all cases, the area is at most $2\pi m$. This contradiction shows that $S$ cannot exist, hence \refclaim{SurfaceObstruction} shows that $(M(\ss), P)$ is a pared manifold.
\end{proof}

The second step in the proof of \refthm{6TheoremInf} is to show that the same end invariants that are realizable in $(M, P\cup(T_1\cup \dots T_k))$ are also realizable in $(M(\ss), P)$.

\begin{proof}[Proof of \refthm{6TheoremInf}]
By \reflem{6ThmParedMfld}, $(M(\ss), P)$ is a pared manifold. Notice that $\bdy_0 M = \bdy_0 M(\ss)$, hence the free sides of $(M(\ss), P)$ are identical to the free sides of $(M, P\cup(T_1\cup \dots T_k))$. Let $F_1, \ldots, F_s$ be these free sides. We further assume that the free sides have been ordered so that the end invariants of $Z$ are $(X_1, \dots, X_r, \lambda_{r+1}, \dots, \lambda_s)$, where $X_i\in\calT(F_i)$ for $i \leq r$ and $\lambda_j$ is an ending lamination for $F_j$ for $j \geq r+1$.

By    the realization theorem, \refthm{Realization}, 
$(M(\ss), P)$ admits a hyperbolic structure with these end invariants if and only if the tuple of end invariants $(X_1, \dots, X_r, \lambda_{r+1}, \dots, \lambda_s)$ is filling. That is, we must check that $(X_1, \dots, X_r, \lambda_{r+1}, \dots, \lambda_s)$ satisfy properties (*) and (**) of \refdef{FillingEndInvts}, when viewed as invariants of $(M(\ss),P)$.

\medskip

If (*) is false, then $M(\ss)$ is an interval bundle over a surface $S$. Without loss of generality (replacing $S$ by its orientable double cover if needed), we may assume that $S$ is orientable and $M(\ss) \cong S \times I$. Since (*) has failed, $M(\ss)$ has no geometrically finite ends. Furthermore, there is an ending lamination $\lambda_i$ on a free side $F_i \subset S \times \{0\}$ and an ending lamination $\lambda_j$ on a free side $F_j \subset S \times \{1\}$ containing parallel, non-isolated leaves. But these leaves are dense in $\lambda_i$ and $\lambda_j$ respectively, which means $(F_i, \lambda_i)$ and $(F_j, \lambda_j)$ have the same projection to $S$. Since $\bdy F_i$ and $\bdy F_j$ are contained in $P$, and the parabolic locus $P$ cannot admit essential annuli, it follows that $F_i = S \times \{0\}$ and $F_j = S \times \{1\}$ are the only free sides of $(M(\ss),P)$. We write $\lambda$ for the ending lamination on $S \times \{0\}$ and $\overline{\lambda}$ for the ending lamination on $S \times \{1\}$. By the above argument (dense leaves), $\lambda$ and $\overline{\lambda}$ are parallel.

Let $X^n \in \calT(S)$ be a filling sequence converging to $\lambda$. Writing $\overline{X^n} \in \calT(\overline S)$ for the same hyperbolic structure with opposite orientation, the sequence $\overline{X^n}$ converges to $\overline \lambda$.
Then $(X^n, \overline{X^n}) \in \calT(\bdy M-P)$ is a filling sequence converging to $(\lambda, \overline \lambda)$.
Theorems~\ref{Thm:Hyperbolization} and~\ref{Thm:ABUniformization} imply there exists a hyperbolic 3-manifold $Z_n$ with standard compact core homeomorphic to $(M, P\cup(T_1\cup \dots \cup T_k))$, and with end invariants $(X^n, \overline{X^n})$. Because the sequence $(X^n, \overline{X^n})$ is filling, there exists a sequence of simple closed curves $\gamma^n \subset S$ converging to $\lambda$, with 
$\ell_{X^n}(\gamma^n) / \ell_{X^1}(\gamma^n) \to 0$. 
In the hyperbolic manifold $Z_n$, the curve $\gamma^n$ is homotopic to a geodesic in the lower end whose conformal structure is $X^n$, and $\overline{\gamma^n}$ is homotopic to a geodesic in the upper end whose conformal structure is $\overline{X^n}$.

The approximation \refthm{NamaziSouto} implies that after passing to a subsequence, the manifolds $Z_n$ converge strongly to a manifold homeomorphic to $M$, with the same end invariants as $Z$. The ending lamination \refthm{ELC} 
implies that the limiting manifold is isometric to $Z$. For each $n$, let $H^n = H_1^n \cup \dots \cup H_k^n$ be a disjoint union of horocusps for $T_1\cup\dots\cup T_k$ in $Z_n$, which converge to the horocusp neighborhood $H = H_1 \cup \dots \cup H_k \subset Z$. The strong limit $Z_n \to Z$ implies that for sufficiently large $n$, we have $\len(s_i) > 6$ in the Euclidean metric on $\bdy H_i^n$.

Consider an annulus $A_n = \gamma^n \times I \subset S\times I = M(\ss)$. Let $A_n'$ be the remnant of $A_n$ in $M \subset M(\ss)$, moved by isotopy to minimize the intersection number with the cores of filling solid tori. Then $A_n'$ has a boundary component along $\gamma^n$ in the lower end, a boundary component along $\overline{\gamma^n}$ in the upper end, as well as some number of punctures along meridians $s_1, \dots, s_k$. In the hyperbolic metric $Z_n$, we may homotope $A_n'$ to be a pleated surface with geodesic boundary along $\gamma^n \cup \overline{\gamma^n}$. For sufficiently large $n$, the geodesics $\gamma^n \cup \overline{\gamma^n} \subset Z_n$ are disjoint from $H^n$, because the geodesic realizations of the same curves in $Z$ are exiting the ends of $Z$ as $n \to \infty$. Furthermore, the horocusp neighborhood
$H^n$
 induces a disjoint union of horocycles in $A'_n$, where each horocycle has length greater than $6$.

Let $m_n$ be the number of punctures in $A'_n$. 
As in the proof of \reflem{6ThmParedMfld}, the Gauss--Bonnet theorem implies that  $\area(A'_n) = 2\pi m_n$. On the other hand, for $n$ large, the cusp area of $A'_n$ is strictly larger than $6 m_n$, hence \reflem{CuspDensitySurface}
implies that $\area(A_n')  > 2\pi m_n$. This contradiction proves that the end invariants of $(M(\ss),P)$ must satisfy (*).

\smallskip

Next, suppose (**) is false. Then $(M(\ss), P)$ has a compressible free side $F$, and the ending lamination $\lambda$ is contained in a limit of meridians. It follows that there is a sequence of compression disks $D_n$ for $(M(\ss), P)$ such that the boundary curves $\gamma^n = \bdy D_n$ converge to a lamination $\mu$ containing $\lambda$.
Since $\lambda$ is \emph{not} contained in a limit of meridians in $(M,P)$, it follows that the curve $\gamma^n$ cannot be a meridian in $(M,P)$ for sufficiently large $n$. In other words, viewing $M$ as a submanifold of $M(s_1, \dots, s_k)$, it follows that the meridian disks $D_n$ cannot be contained in $M$ for large $n$. We isotope each $D_n$ to meet the filling solid tori as few times as possible, and set $D_n' = D_n \cap M$.

By strong density, \refthm{Density}, there is a sequence of geometrically finite and minimally parabolic representations $\rho_n \in AH(M,P)$ converging strongly to the representation $\rho$ corresponding to $Z$. Let $Z_n$ be the associated manifolds. As above, each $Z_n$ is equipped with a disjoint union of horocusps $H^n = H_1^n \cup \dots \cup H_k^n$, where for sufficiently large $n$ we have $\len(s_i) > 6$ in the Euclidean metric on $\bdy H_i^n$. Since the curves $\gamma^n$ limit to a lamination containing $\lambda$, they must exit the end of $Z$ corresponding to $F$, hence $\gamma^n$ is disjoint from the horocusp neighborhood $H^n$ for large $n$. As above, we may pleat the punctured disk $D_n'$ in $Z_n$, so that it has geodesic boundary along $\gamma^n$ and some number of punctures corresponding to meridians of $s_1, \ldots, s_k$. These punctures are cut off by disjoint horocycles in $D_n'$, where each horocycle has length greater than $6$. 

We can now obtain a contradiction as above. Let $m_n$ be the number of punctures in $D'_n$. The Gauss--Bonnet theorem implies that  $\area(D'_n) = 2\pi (m_n -1)$. On the other hand, for $n$ large, the cusp area of $D'_n$ is strictly larger than $6 m_n$, hence \reflem{CuspDensitySurface}
implies that $\area(D_n')  > 2\pi m_n$. This contradiction proves that the end invariants $(X_1, \dots, X_r, \lambda_{r+1}, \dots, \lambda_s)$ must satisfy (**) as well.

Since the tuple of end invariants $(X_1, \dots, X_r, \lambda_{r+1}, \dots, \lambda_s)$ satisfy both (*) and (**), \refthm{Realization} says that these invariants are realized by a hyperbolic metric $Y$ on $(M(\ss),P)$.
\end{proof}

\section{Convergence of geodesics}\label{Sec:ConvergenceOfLinks}
In several theorems in the subsequent sections, we will need to control geodesic links in a convergent sequence of manifolds. The following proposition says that geodesic links behave exactly as expected.

\begin{proposition}\label{Prop:GeodesicsConverge}
Let $N_\infty$ be a hyperbolic 3-manifold with standard compact core $(M,P)$ and associated representation $\rho_\infty \in AH(M,P)$. Let $\rho_n \in AH(M,P)$ be a sequence of geometrically finite, minimally parabolic representations converging strongly to $\rho_\infty$. Let $N_n$ be the hyperbolic 3-manifolds associated to $\rho_n$. Let $\Sigma = \sigma_1 \cup \ldots \cup \sigma_k \subset M$ be a smooth link. 

Then the following are equivalent:
\begin{enumerate}[\: \: $(1)$]
\item\label{Itm:GeoInLimit} $\Sigma$ is isotopic to a geodesic link $\Sigma_\infty \subset N_\infty$, where each component has length less than $\log 3$.
\item\label{Itm:GeoInApprox} For all $n \gg 0$, $\Sigma$ is isotopic to a geodesic link $\Sigma_n$ in the hyperbolic structure $N_n$, where each component has length less than $\log 3 - \delta$ for some uniform $\delta > 0$.
\end{enumerate}
Furthermore, assuming either \refitm{GeoInLimit} or \refitm{GeoInApprox} hold, we have $\Sigma_n \to \Sigma_\infty$ in the geometric limit. 
\end{proposition}

\begin{proof}
If $\bdy_0 M = \emptyset$, then any convergent sequence in $AH(M,P)$ is eventually constant, hence the result is vacuous.
We may now suppose that $\bdy_0 M \neq \emptyset$, or equvivalently $\vol(N_n) = \vol(N_\infty) = \infty$.

We begin by setting some notation. For each component $\sigma_i \subset \Sigma$, orient $\sigma_i$ and choose an arc $\alpha_i$ that runs from the basepoint $x \in M$ to $\sigma_i$. Then the based loop $\alpha_i \cdot \sigma_i \cdot \overline \alpha_i$ represents a homotopy class $\gamma_i \in \pi_1(M,x)$. 
We also choose basepoints $x_n \in N_n$ and $x_\infty \in N_\infty$ so that $(N_n, x_n) \to (N_\infty, x_\infty)$ is a geometric limit.

\smallskip

We first prove that \refitm{GeoInLimit} implies \refitm{GeoInApprox}. Suppose $\Sigma_\infty \subset N_\infty$ is a geodesic realization of $\Sigma$. Since the components $\sigma_1, \ldots, \sigma_k \subset \Sigma$ are isotopic to disjoint, simple geodesics in $N_\infty$, the group elements $\gamma_1, \ldots, \gamma_k$ corresponding to $\sigma_1, \ldots, \sigma_k$ are primitive and pairwise non-conjugate.

Choose a radius $R$ large enough so that $\Sigma_\infty \subset B_R(x_\infty)$. According to the characterization of geometric convergence in \refeqn{PointedGromovHausdorff}, for large $n$ we have embeddings $g_{n,R} \from (B_R(x_\infty),x_\infty) \hookrightarrow (N_n, x_n)$ that converge to isometries in the $C^\infty$ topology. These embeddings map the geodesic link $\Sigma_\infty$ to a link $\Sigma'_n \subset N_n$. 

Since each of the $k$ components of $\Sigma_\infty$ is shorter than $\log 3$, there is a uniform $\delta > 0$ such that each component is shorter than $\log 3 - \delta$.

Let $\sigma_{i,\infty}$ be a component of $\Sigma_\infty$. Then, for all sufficiently large $n$, the image $\sigma'_{i,n} = g_{n,R}(\sigma_{i,\infty}) \subset \Sigma'_n$ is shorter than $\log 3 - \delta$. By \refthm{MargulisEstimate}, $\log 3$ is a Margulis number for each $N_n$. Since the group elements $\gamma_1, \ldots, \gamma_k$ represent distinct, primitive conjugacy classes, the curves $\sigma'_{1,n}, \ldots, \sigma'_{k,n}$
lie in disjoint tube components of $N_n^{< \log 3}$, with each component homotopic to the core of its tube. Furthermore, each $\sigma'_{i,n}$ can be taken to have arbitrarily small geodesic curvature (by choosing $n$ large and applying the definition of a geometric limit), hence $\sigma_n$ cannot have any local knotting. Thus $\sigma'_{i,n}$ is \emph{isotopic} to the core of the tube, denoted $\sigma_{i,n}$, where $\len(\sigma_{i,n}) < \log 3- \delta$. The isotopies in distinct tubes do not interact, hence $\Sigma'_n$ is isotopic to a geodesic link $\Sigma_n$, proving \refitm{GeoInApprox}.

Next, we check that $\Sigma_n \to \Sigma_\infty$. For each $n$, let $\rho_n(\gamma_i)$ be the holonomy of $\gamma_i$ in the hyperbolic structure $N_n$. For each $n \gg 0$, the group element $\rho_n(\gamma_i)$ stabilizes a geodesic axis $\widetilde \sigma_{i,n}$ that covers $\sigma_{i,n}$.
The algebraic limit $\rho_n \to \rho_\infty$ implies that $\rho_n(\gamma_i) \to \rho_\infty(\gamma_i) \in \Isom(\HH^3)$. Thus the fixed points of $\rho_n(\gamma_i)$ converge to the fixed points of $\rho_\infty(\gamma_i)$, and the axes $\widetilde \sigma_{i,n}$ converge to the axis $\widetilde \sigma_{i,\infty}$ that covers $\sigma_{i,\infty}$. Projecting down to the quotient manifold $M$, we learn that the closed geodesics $\sigma_{i,n} \subset N_n$ converge to $\sigma_{i,\infty}$, as desired.

\smallskip

Now, we prove that \refitm{GeoInApprox} implies \refitm{GeoInLimit}. Suppose that for $n \gg 0$, the hyperbolic structure $N_n$ contains a geodesic link $\Sigma_n$ isotopic to $\Sigma$. Then the component $\sigma_i \subset \Sigma$ is isotopic to a closed geodesic $\sigma_{i,n} \subset N_n$, of length $\len(\sigma_{i,n}) < \log 3 - \delta$ for some uniform $\delta$. By the same argument as in the above paragraph, the algebraic limit $\rho_n \to \rho_\infty$ implies that the geodesic axes $\widetilde \sigma_{i,n}$ for  $\rho_n(\gamma_i)$ converge to the geodesic axis $\widetilde \sigma_{i,\infty}$ for $\rho_\infty(\gamma_i)$. Thus, for each $i$, the closed geodesics $\sigma_{i,n} \subset N_n$ converge to a closed geodesic $\sigma_{i,\infty} \subset N_\infty$. Since translation lengths converge in the limit, we have
$\len(\sigma_{i,\infty}) \leq \log 3 - \delta < \log 3$ for each $i$.
 Consequently, the closed geodesics $\sigma_{1,\infty}, \ldots, \sigma_{k,\infty}$ lie at the cores of disjoint tube components of $N_\infty^{< \log 3}$, and $\Sigma_\infty = \sigma_{1,\infty} \cup \ldots \cup \sigma_{k,\infty}$ is a geodesic link.

By construction, the loop $\gamma_i = \alpha_i \cdot \sigma_i \cdot \overline \alpha_i$ is freely homotopic to $\sigma_i$, hence the closed geodesic $\sigma_{i,\infty}$ is freely homotopic to $\sigma_i$. It remains to show that this homotopy can be achieved by isotopy, and that the isotopies for different components of $\Sigma$ do not interact.

As above, we may choose a radius $R$ large enough that $\Sigma_\infty \subset B_{R/2}(x_\infty)$. Then, for $n \gg 0$, we have embeddings $g_{n,R} \from (B_R(x_\infty),x_\infty) \hookrightarrow (N_n, x_n)$ that converge to isometries in the $C^\infty$ topology. Since $\Sigma_n \to \Sigma_\infty$, it follows that for $n \gg 0$, we have $\Sigma_n \subset g_{n,R} (B_R(x_\infty))$. Then, for each component $\sigma_i$, the preimage $g_{n,R}^{-1}(\sigma_{i,n})$ is an almost-geodesic closed curve in the Margulis tube containing $\sigma_{i, \infty}$. As above, the control on geodesic curvature implies that $g_{n,R}^{-1}(\sigma_{i,n})$ cannot have local knotting, hence is isotopic to $\sigma_{i, \infty}$. Thus $\sigma_i \subset \Sigma$ is isotopic to $\sigma_{i, \infty}$, with isotopies of different components supported in disjoint Margulis tubes. Thus $\Sigma \subset M$ is isotopic to $\Sigma_\infty$, as desired.
\end{proof}

\section{Extracting bilipschitz limits}\label{Sec:Extracting}

In this section, we describe the construction of a bilipschitz function from a pair of geometric limits. The main result, \refthm{ConstructBilip}, is a technical statement that will be used in the proofs of Theorems~\ref{Thm:EffectiveDrillingTame} and~\ref{Thm:EffectiveFillingTame}. The idea is that if hyperbolic 3-manifolds $Y$ and $Z$ each have a sequence of manifolds limiting to them geometrically, and there are $J$--bilipschitz maps between the thick parts of the approximating manifolds, then there is also a $J$--bilipschitz map between the thick parts of $Y$ and $Z$.

The precise statement that we need involves \emph{convex submanifolds} of a hyperbolic manifold. If $Y$ is a hyperbolic 3-manifold with universal covering map $\pi \from \HH^3 \to Y$, a submanifold $Q \subset Y$ is called \emph{convex} if the full preimage $\pi^{-1}(Q)$ is a convex subset of $\HH^3$. The inclusion $Q \hookrightarrow Y$ is necessarily a homotopy equivalence.
The convex core $CC(Y)$ is always a convex submanifold; in fact, it is the intersection of all convex submanifolds of $Y$. Another important example of a convex submanifold is the scooped manifold $N^\circ \subset N$ in \refdef{DoubleDouble}.

One important property of a convex submanifold $Q \subset Y$ is that intrinsic and extrinsic notions of injectivity radius agree. Given a point $x \in Q$, the injectivity radius  $\injrad(x) = \epsilon/2$ is realized by a geodesic loop $\gamma$ based at $x$, of length exactly $2\epsilon$; compare \cite[Lemma 2.11]{FPS:EffectiveBilipschitz}. By convexity, this loop must be contained in $Q$. Consequently, $Q^{\geq \epsilon} = Q \cap Y^{\geq \epsilon}$.

We have the following theorem.

\begin{theorem}[Bilipschitz limit]\label{Thm:ConstructBilip}
Fix $\delta > 0$, $\epsilon >0$ and $J > 1$.  Let $(Y_m, y_m) \to (Y, y)$ and $(Z_m, z_m) \to (Z, z)$ be geometrically converging sequences of based hyperbolic manifolds.
For each $m$, let $Y_m^\circ$ and $Z_m^\circ$ be convex submanifolds of $Y_m$ and $Z_m$, respectively.  Suppose that $y \in Y^{> \epsilon}$ and  $y_m \in (Y_m^\circ)^{\geq \epsilon}$,  while  $z_m \in Z_m^\circ$. Suppose $(Y_m^\circ, y_m) \to (Y, y)$ and $(Z_m^\circ, z_m) \to (Z, z)$.

Suppose that, for each $m$, there is a $J$--bilipschitz inclusion $\varphi_m \from (Y_m^\circ)^{\geq \epsilon} \hookrightarrow (Z_m^\circ)^{\geq \delta}$, such that $d(\varphi_m(y_m), z_m)$ is uniformly bounded. Then there is also a $J$--bilipschitz inclusion $\varphi \from Y^{\geq \epsilon} \hookrightarrow Z^{\geq \delta}$.
\end{theorem}

The proof of \refthm{ConstructBilip} proceeds in two steps. In the first step, carried out in \reflem{ConstructPartialBilip}, we construct compact sets $K_n \subset Y$ and bilipschitz functions $h_n \from K_n \to Z$, such that each $h_n$ is \emph{almost} $J$--bilipschitz and has image \emph{almost} contained in $Z^{\geq \delta}$. These compact sets are nested, with $K_1 \subset K_2 \subset \ldots$, and form an exhaustion of $Y^{> \epsilon}$. In the second step, carried out in \reflem{ArzelaFTW}, we extract subsequential limits of the locally defined functions $h_n$ to obtain the desired bilipschitz inclusion $\varphi$.

\begin{lemma}\label{Lem:ConstructPartialBilip}
Let the notation and hypotheses be as in \refthm{ConstructBilip}. For $n \in \NN$, define sequences of numbers as follows: 
\[
\delta_n = (1- \tfrac{1}{n})\delta, \quad \epsilon_n = (1+ \tfrac{1}{n}) \epsilon, \quad J_n = 2^{1/n} J.
\]
Then, for all sufficiently large $n$, there is a $J_n$--bilipschitz map
\[
h_n \from \overline B_{n}(y) \cap Y^{\geq \epsilon_n} \hookrightarrow Z^{\geq \delta_n}.
\]
Furthermore, $\{h_n(y)\}$ is a bounded sequence in $Z$.
\end{lemma}

\begin{proof}
We begin by characterizing what it means for $n$ to be ``sufficiently large''. Since $y \in Y^{> \epsilon}$ by hypothesis, and $\epsilon_n \to \epsilon$, we choose $n$ large enough to ensure $y \in Y^{\geq \epsilon_n}$. 
Next, let $D$ be an upper bound on the distances $d(\varphi(y_m) ,z_m)$, where $y_m \in (Y_m^\circ)^{\geq \epsilon}$ is the basepoint of $Y_m^\circ$ and $z_m \in Z_m^\circ$ is the basepoint of $Z_m^\circ$. We also choose $n$ large enough so that $4Jn > 2^{1/n} J \cdot 2n + D$.

For each sufficiently large $n$, we will construct a map $h_n$. To that end, consider the set of points in $Y$ of distance less than $2n$ from $y$, denoted $B_{2n}(y)$.
Since $(Y_m^\circ, y_m) \to (Y, y)$, equation~\refeqn{PointedGromovHausdorff} says that for large $m \in \NN$ there exist embeddings
 \[
 f_{m, 2n} \from (B_{2n}(y), y) \hookrightarrow (Y_m^\circ, y_m)
\]
that converge to isometries in the $C^\infty$ topology as $m \to \infty$. 

We will collect several desirable properties that hold for large $m$. First, observe that the closed set $\overline B_n(y)$ is compact and the derivatives of $ f_{m, 2n}$ are converging to the identity. Thus, for all sufficiently large $m$, we have:
\begin{equation}\label{Eqn:fLipschitz}
\text{The restriction of $ f_{m, 2n}$ to $\overline B_n(y)$ is $2^{1/2n}$--bilipschitz.}
\end{equation}
Next, a lemma of Canary, Epstein, and Green \cite[Lemma 3.2.6]{ceg:notes-on-notes} shows that injectivity radii converge in a geometric limit. Since $\epsilon_n > \epsilon$, it follows that for all sufficiently large $m$, we have
\begin{equation}\label{Eqn:fInjectivity}
f_{m, 2n} \big( \overline B_{n}(y) \cap Y^{\geq \epsilon_n} \big) \subset (Y_m^\circ)^{\geq \epsilon}.
 \end{equation}
 In particular, $f_{m, 2n} (y) = y_m \in (Y_m^\circ)^{\geq \epsilon}$.
Combining \refeqn{fLipschitz}, \refeqn{fInjectivity}, and the $J$--bilipschitz property of the map
 $\varphi_m \from (Y_m^\circ)^{\geq \epsilon} \to (Z_m^\circ)^{\geq \delta}$ gives
\[
\mathrm{diam} \left( \varphi_m \circ f_{m, 2n} \big( \overline B_{n}(y) \cap Y^{\geq \epsilon_n} \big) \right) \leq J \cdot 2^{1/2n} \cdot 2n.
\]
The image set $ \varphi_m \circ f_{m, 2n} \big( \overline B_{n}(y) \cap Y^{\geq \epsilon_n} \big) $ contains $\varphi_m(y_m)$, hence this set is contained in the $ \big( 2^{1/2n} J \cdot 2n +D \big)$--neighborhood of $z_m$, by the definition of $D$.

Now, consider the geometric limit $(Z_m^\circ, z_m)  \to (Z, z)$. For large $m$, there exist embeddings
 \[
g_{m, 4Jn} \from (B_{4Jn}(z), z) \hookrightarrow (Z_m^\circ, z_m)
\]
that converge to isometries as $m \to \infty$,  with injectivity radii converging as before. Since the image of $B_{4Jn}(z)$ converges to a set of point of distance less than $4Jn$ to $z_m$, and we have already chosen $n$ so that $4Jn > 2^{1/n} J \cdot 2n + D$, it follows that for large $m$ we have:
\begin{equation}\label{Eqn:gContainment}
 \varphi_m \circ f_{m, 2n} \big( \overline B_{n}(y) \cap Y^{\geq \epsilon_n} \big) \subset g_{m, 4Jn} (B_{4Jn}(z)) \subset (Z_m^\circ)^{\geq \delta}.
\end{equation}
Since the derivatives of $g_{m, 4Jn}$ and its inverse converge to the identity as $m \to \infty$, choosing $m$ large ensures that the lipschitz constants on any compact subset are close to $1$. Thus, for $m$ large:
\begin{equation}\label{Eqn:gLipschitz}
\text{The restriction of $ g_{m, 4Jn}^{-1}$ to $\varphi_m \circ f_{m, 2n} \big( \overline B_{n}(y) \cap Y^{\geq \epsilon_n} \big)$ is $2^{1/2n}$--bilipschitz.}
\end{equation}
Finally, since $\varphi_m \circ f_{m, 2n} \big( \overline B_{n}(y) \cap Y^{\geq \epsilon_n} \big) \subset (Z_m^\circ)^{\geq \delta}$, and injectivity radii converge in the limit $(Z_m^\circ, z_m)  \to (Z, z)$, it follows that choosing $m$ large ensures:
\begin{equation} \label{Eqn:gInjectivity}
  g_{m, 4Jn}^{-1} \circ \varphi_m \circ f_{m, 2n} \big( \overline B_{n}(y) \cap Y^{\geq \epsilon_n} \big) \subset Z^{\geq 
\delta_n}.
\end{equation}

We are now ready to define the function $h_n$ and check that it has all the desired properties. For each $n$, choose a single number $m = m_n$ so that conditions \refeqn{fLipschitz}--\refeqn{gInjectivity} hold simultaneously. Then, set
\[
h_n = g_{m, 4Jn}^{-1} \circ \varphi_{m} \circ f_{m, 2n} \: :  \: \overline B_{n}(y) \cap Y^{\geq \epsilon_n} \longrightarrow Z.
\]
This function is well-defined by \refeqn{fInjectivity} and \refeqn{gContainment}. It is $J_n$--bilipschitz for $J_n = 2^{1/n}J$, because $f_{m, 2n}$ and $g_{m, 4Jn}^{-1}$ are both $2^{1/2n}$--bilipschitz on the relevant domain by \refeqn{fLipschitz} and \refeqn{gLipschitz}, while $\varphi_{m}$ is $J$--bilipschitz. The image $h_n \big( \overline B_{n}(y) \cap Y^{\geq \epsilon_n} \big)$ is contained in $Z^{\geq 
\delta_n}$ by \refeqn{gInjectivity}. Finally, the points
\[
h_n(y) =  g_{m, 4Jn}^{-1} \big( \varphi_{m} (y_{m})  \big) \quad \text{and} \quad z = g_{m, 4Jn}^{-1} (z_{m})
\]
are within distance $2 D$, by \refeqn{gLipschitz} and the definition of $D$, hence $\{h_n(y)\}$ is a bounded sequence.
\end{proof}

The second step of the proof of \refthm{ConstructBilip} is to extract a subsequential limit of the functions $h_n$ that were built in \reflem{ConstructPartialBilip}. This step does not need any hyperbolic geometry or smoothness, and only needs $Y$ and $Z$ to be metric spaces. So we write down the proof in that generality.

\begin{lemma}\label{Lem:ArzelaFTW}
Let $Y$ and $Z$ be metric spaces. Let $K_1 \subset K_2 \subset \ldots $ be an exhaustion of $Y$ by compact sets. For each $n$, let $h_n \from K_n \to Z$ be a continuous function that is a $J_n$--bilipschitz homeomorphism to its image. Suppose that the sequence of images $\{h_n(y)\}$ is bounded for some basepoint $y \in K_1$, and that $\lim_{n \to \infty} J_n = J$ for some $J$.

Then there is a $J$--bilipschitz inclusion $h \from Y \hookrightarrow Z$. Furthermore, for every $x \in Y$, we have $h(x) = \lim h_{n_i}(x)$ for some subsequence $n_i$.
\end{lemma}

\begin{proof}
After replacing $\{h_n\}$ by a subsequence, we may assume without loss of generality that $\{ J_n \}$ is a monotonic sequence. 
We will construct $h$ by repeated application of the Arzela--Ascoli theorem.

Focusing attention on a single compact set $K_m$, the functions $h_n$ are defined on $K_m$ for all $n \geq m$. The family $\{h_n : n \geq m\}$ is equicontinuous on $K_m$ because each $h_n$ is $\overline J$--lipschitz, where $\overline J = \sup J_n$. Furthermore, the set of images $\{h_n(y)\}$ is bounded by hypothesis, and $d(h_n(y), h_n(x))$ is also uniformly bounded for $x \in K_m$ because the functions $h_n$ are uniformly lipschitz. Thus $\{h_n : n \geq m\}$ is uniformly bounded on $K_m$. By Arzela--Ascoli, some subsequence converges uniformly on $K_m$.

Now, consider $K_1$. By the above paragraph, there is a subsequence $\{ h_n^1 \}$ of $\{ h_n \}$ that converges uniformly on $K_1$. Define $h(x) = \lim_{n \to \infty} h_n^1 (x)$ for $x \in K_1$.

Next, consider $K_2 \supset K_1$. As above, there is a subsequence $\{ h_n^2 \}$ of $\{ h_n^1 \}$ that converges uniformly on $K_2$. Define $h(x) = \lim_{n \to \infty} h_n^2 (x)$ for $x \in K_2$. This agrees with the previous definition of $h$ on $K_1 \subset K_2$ because we have taken a subsequence of a sequence that already converges on $K_1$.

Continuing inductively in this manner, we have a subsequence $\{ h_n^m \}$ of $\{ h_n^{m-1} \}$ that converges uniformly on $K_m$. We then have $h(x) = \lim_{n \to \infty} h_n^m (x)$ for $x \in K_m$. Since the $K_m$ provide an exhaustion of $Y$, and the definition is consistent as $m$ grows, this defines $h$ on all of $Y$.

It remains to show that $h$ is $J$--bilipschitz, where $J = \lim J_n$. Consider a pair of points $x, x' \in Y$, and let $K_m$ be such that $x, x' \in K_m$. Recall our assumption at the beginning of the proof that $\{ J_n \}$ is monotonic. If $J_n$ is monotonically increasing with $n$, then every $h_n^m$ is already $J$--bilipschitz. In particular, we have:
\[
J^{-1} \cdot d(x, x') \leq d\big(h_n^m(x), h_n^m(x') \big) \leq J \cdot d(x, x').
\]
As $n \to \infty$, the middle term converges to $d(h(x), h(x'))$,
hence $h$ is $J$--bilipschitz.

If $J_n$ is monotonically decreasing with $n$, then $J_n^m$ is also monotonically decreasing. Fix an integer $k \gg 0$. Then, for $n \geq k$, every $h_n^m$ is  $J_k^m$--bilipschitz. Thus, for all $n \geq k$, we have:
\[
(J_k^m)^{-1} \cdot d(x, x') \leq d\big(h_n^m(x), h_n^m(x') \big) \leq J_k^m \cdot d(x, x').
\]
As $n \to \infty$ (holding $k$ fixed), the middle term converges to $d(h(x), h(x') )$. Then we take a limit as $k \to \infty$, and $J_k^m$ converges to $J$. We obtain
\[
J^{-1} \cdot d(x, x') \leq d\big(h(x), h(x') \big) \leq J \cdot d(x, x'),
\]
hence $h$ is $J$--bilipschitz.
\end{proof}

\begin{proof}[Proof of \refthm{ConstructBilip}]
As in the statement of the theorem, we have hyperbolic manifolds $Y$ and $Z$, with a basepoint $y \in Y^{> \epsilon}$. Let $\delta_n \to \delta$, $\epsilon_n \to \epsilon$, and $J_n \to J$ be the convergent sequences of \reflem{ConstructPartialBilip}, and let $K_n = \overline B_n(y) \cap Y^{\geq \epsilon_n}$ be a compact set. \reflem{ConstructPartialBilip} says that for all sufficiently large $n$, there is a $J_n$--bilipschitz map $h_n \from K_n \hookrightarrow Z^{\geq \delta_n}$,
such that the images of the basepoint $\{h_n(y)\}$ are bounded in $Z$. We reindex the sequence so that it starts at $n=1$.

Observe that $K_1 \subset K_2 \subset \ldots$ is an exhaustion of $Y^{> \epsilon}$ by compact sets. Now, \reflem{ArzelaFTW} constructs a $J$--bilipschitz function $h \from Y^{> \epsilon} \hookrightarrow Z$. For every $x \in Y$, there is a subsequence $h_{n_i}$ (depending on the compact set containing $x$) such that $h(x) = \lim h_{n_i}(x)$. Since $h_{n_i}(x) \in Z^{\geq \delta_{n_i}}$, and $\delta_{n_i} \to \delta$, it follows that in fact $h(x) \in Z^{\geq \delta}$. Thus we have a $J$--bilipschitz map $h \from Y^{> \epsilon} \hookrightarrow Z^{\geq \delta}$.

Since $h$ is $J$--bilipschitz, it has a continuous and $J$--bilipschitz extension to $\bdy Y^{> \epsilon}$. This achieves our goal: a $J$--bilipschitz inclusion $\varphi = h \from Y^{\geq \epsilon} \hookrightarrow Z^{\geq \delta}$.
\end{proof}

\begin{remark}\label{Rem:ConstructBilipThin}
In the proof of \refthm{GeoInfiniteBB}, we will rely on the following straightforward extension of \refthm{ConstructBilip}. Suppose that  $\epsilon$ and $\delta$ are Margulis numbers for $Y$ and $Z$, respectively. Let $Y^{<\epsilon}_{\Sigma}$ be the disjoint union of certain components of $Y^{<\epsilon}$, and let $Z^{<\delta}_{\Sigma}$ be the disjoint union of certain components of $Z^{<\delta}$. Let $(Y_m)^{<\epsilon}_{\Sigma}$ and $(Z_m)^{<\delta}_{\Sigma}$ be the corresponding subsets of the approximating manifolds in the geometric limits $(Y_m, y_m) \to (Y, y)$ and $(Z_m, z_m) \to (Z, z)$.
Suppose that, for each $m$, there is a $J$--bilipschitz inclusion $\varphi_m \from (Y_m - (Y_m)^{<\epsilon}_{\Sigma}) \hookrightarrow (Z_m - (Z_m)^{<\delta}_{\Sigma})$, such that $d(\varphi_m(y_m), z_m)$ is uniformly bounded. Then there is also a $J$--bilipschitz inclusion $\varphi \from (Y - Y^{<\epsilon}_{\Sigma}) \hookrightarrow (Z - Z^{<\delta}_{\Sigma})$.

The main change from the statement of \refthm{ConstructBilip} is that $Y^{\geq \epsilon}$ has been replaced by $(Y - Y^{<\epsilon}_{\Sigma})$, and similarly for $Z$, so that the $J$--bilipschitz inclusion $\varphi$ now extends into some of the thin parts of $Y$. The only change needed in the proof is that in \reflem{ConstructPartialBilip} the compact set $\overline B_n(y) \cap Y^{\geq \epsilon}$ should be replaced by the compact set $\overline B_n(y) \cap (Y - Y^{<\epsilon}_{\Sigma})$. This again gives a compact exhaustion of the desired submanifold of $Y$; the rest of the argument works verbatim.
\end{remark}

\section{The effective drilling theorem}\label{Sec:Drilling}

In this section, we prove \refthm{EffectiveDrillingTame}. The proof proceeds in two main steps. In the first step, accomplished in \refthm{DrillingApproximation}, we use a number of results from Kleinian groups (\refsec{Kleinian}) to approximate both $Y$ and the drilled manifold $Z \cong Y-\Sigma$ with geometrically finite manifolds that admit a circle packing. We also double-double the scooped versions of the geometrically finite manifolds to obtain finite-volume hyperbolic manifolds converging to $Y$ and $Z$, respectively.
In the second step, we use the geometric limits constructed in \refthm{DrillingApproximation}, combined with  the finite-volume \refthm{BilipFiniteVolume} and the bilipschitz limit \refthm{ConstructBilip}, to build bilipschitz maps between the thick parts of $Y$ and $Z$. See \reffig{DrillingApproxOutline} for a visual outline. 

We also employ the same outline to prove \refthm{GeoInfiniteBB}, which extends Brock and Bromberg's \refthm{BBDrillingThm} to all tame hyperbolic 3-manifolds without rank-1 cusps.

Before beginning the proofs of those results, we need to verify that for any geodesic link $\Sigma$, the complement $Y - \Sigma$ admits a hyperbolic structure $Z$ with end invariants identical to those of $Y$.

\begin{lemma}\label{Lem:DrillHyperbolic}
Let $Y$ be a tame hyperbolic 3-manifold with standard compact core $(M, P)$. Let $\Sigma \subset Y$ be a geodesic link with a regular neighborhood $\calN(\Sigma) \subset M$. Then $(M- \calN(\Sigma), P\cup \bdy \calN(\Sigma))$ is a pared manifold that admits a hyperbolic metric $Z$ with the same end invariants as those of $Y$.
\end{lemma}

\begin{proof}
We begin by checking that $(M- \calN(\Sigma), P\cup \bdy \calN(\Sigma))$ is a pared manifold. Since $M$ is compact, oriented, and not a 3-ball or solid torus, the same is true of $M - \calN(\Sigma)$. 
For the other properties, recall a theorem of Kerckhoff that  $M - \Sigma$ admits a complete metric of variable negative curvature. (See Agol~\cite[pages 908--909]{agol:drilling} for a proof.) It follows that $M - \calN(\Sigma)$ is irreducible and algebraically atoroidal, meaning that every $\pi_1$--injective map of a torus $T^2 \to  M - \calN(\Sigma)$ is homotopic into some boundary torus belonging to either $P$ or to $\bdy \calN(\Sigma)$.
    
Now, consider a $\pi_1$--injective map of an annulus $f \from (A, \bdy A) \to (M- \calN(\Sigma), P\cup \bdy \calN(\Sigma))$.
Since $\Sigma$ is a disjoint union of geodesics in $Y$, an essential curve on $\bdy \calN(\Sigma)$ cannot be homotopic to $P$ through $A$, and two distinct essential curves on $\bdy \calN(\Sigma)$ cannot be homotopic to each other through $A$. Thus both components of $\bdy A$ must be mapped to $P$, which means that we in fact have $f \from (A, \bdy A) \to (M - \calN(\Sigma), P)$. Since $(M,P)$ is a pared manifold, it follows that $f$ is homotopic into $P$ through $M$. Since we have already checked that $M- \calN(\Sigma)$ is irreducible and atoroidal, and the geodesic components of $\Sigma$ cannot be homotopic into $P$, the homotopy of $f$ can be taken to avoid $\calN(\Sigma)$. This proves that $(M- \calN(\Sigma), P\cup \bdy \calN(\Sigma))$ is a pared manifold. 

To prove that the end invariants of $Y$ are realized by a hyperbolic structure on $M - \calN(\Sigma)$,
we need to check that the end invariants of $Y$ are still filling when viewed as end invariants for $(M- \calN(\Sigma), P\cup\bdy \calN(\Sigma))$. That is, we need to check conditions (*) and (**) of \refdef{FillingEndInvts}. Condition (*) holds automatically: the only way $M- \calN(\Sigma)$ can be an interval bundle over a surface is if $M$ is a solid torus, which we have already ruled out. For condition (**), let $F_i$ be a free side of $\bdy_0 M$ and $\lambda_i$ be the ending lamination on $F_i$. If $\lambda_i$ is contained in the Hausdorff limit of a sequence of meridians in $M- \calN(\Sigma)$, then the same meridians are also meridians in $M$, a contradiction. Thus the end invariants of $M- \calN(\Sigma)$ are filling, and \refthm{Realization} says that these end invariants are realized by a hyperbolic structure $Z$.
\end{proof}

\subsection{Scooped manifolds approaching $Y$ and $Z$}
The following theorem encapsulates the limiting construction that will be used in the proof of \refthm{EffectiveDrillingTame}. 
We will also use this result in the proof of \refthm{ShortDrillingTame}.

\begin{theorem}\label{Thm:DrillingApproximation}
Let $Y$ be a tame, infinite-volume hyperbolic 3-manifold with associated pared manifold $(M,P)$ and associated representation $\rho \in AH(M,P)$. Let $\Sigma \subset Y$ be a geodesic link, such that each component of $\Sigma$ is shorter than $\log 3$. Then $(M - \calN(\Sigma), P \cup \bdy \calN(\Sigma))$ is a pared manifold admitting a hyperbolic metric $Z$ with the same end invariants as those of $Y$. Furthermore,  there exist approximating sequences such that the following properties hold for all $n \gg 0$:
\begin{enumerate}[\: \: $(1)$]
\item\label{Itm:CirclePack} There is a geometrically finite, minimally parabolic representation $\rho_n \in AH(M,P)$.
The conformal boundary of $V_n = \HH^3 / \rho_n(\pi_1 M)$ admits a circle packing $C_n$. Furthermore, $\rho_n \to \rho$ is a strong limit.
\smallskip
\item\label{Itm:ScoopConverge} For every $y \in Y$, there is a choice of basepoints $v_n \in V_n^\circ$, such that $(V_n^\circ, v_n)$ converges geometrically to $(Y, y)$.  There is a geodesic link $\Sigma_n \subset V_n^\circ$, carried to $\Sigma$ by a homeomorphism $V_n \to Y$, such that $\Sigma_n \to \Sigma$ in the geometric limit.
\smallskip
\item\label{Itm:ScoopedDrill} There is a geometrically finite, minimally parabolic representation $\xi_n \in AH(M - \calN(\Sigma), \allowbreak P \cup \bdy \calN(\Sigma))$ such that the associated hyperbolic manifold
 $W_n$ is homeomorphic to $Y- \Sigma$ and $V_n - \Sigma_n$ and has end invariants that are identical to those of $V_n$. In particular, the conformal boundary of $W_n$ admits the same circle packing $C_n$. 
\smallskip
\item\label{Itm:DrilledLimit} For every $z \in Z$, there is a choice of basepoints $w_n \in W_n^\circ$, such that $(W_n^\circ, w_n)$ converges geometrically to $(Z,z)$. Furthermore, there is a strong limit $\xi_n \to \xi$, where $\xi \in AH(M - \calN(\Sigma), \allowbreak P \cup \bdy \calN(\Sigma))$ is a representation associated to $Z$.
\smallskip
\item\label{Itm:DrillDoubleCommute} The operations of drilling and double-doubling commute. That is:
\[
DD(V_n^\circ) - DD(\Sigma_n) \cong DD(W_n^\circ).
\]

\smallskip
\item\label{Itm:FiniteVolumeLimit} The manifolds $Y$ and $Z$ are geometric limits of finite-volume hyperbolic manifolds, as follows: $(DD(V_n^\circ), v_n) \to (Y, y)$ and $(DD(W_n^\circ), w_n) \to (Z, z)$.
\end{enumerate}
\end{theorem}

\begin{figure}
\begin{tikzcd}[column sep=4.0em, row sep=3.0em]
& & & Y \arrow[from=dlll, dashed, "\text{geom limit}"] \arrow[from=dl, dashed] \arrow[from=dr] \arrow[from=dr] \arrow[from=drrr, "\text{strong density}"'] \arrow[dd, "\text{drill } \Sigma"'] & & & \\
  DD(V_n^{\circ}) \arrow[rr, hookleftarrow, "\text{isometric inclusion}"'] \arrow[dd, "\text{drill } DD(\Sigma_n)", "\text{finite vol}"'] & & V_n^{\circ} \arrow[rr, hookrightarrow, "\text{isometric \: inclusion}"', crossing over]  & & V_n  \arrow[rr,-,dashed, "< \delta_n", "\text{circle pack}"']  & & Y_n \arrow[dd,"\text{drill}"] \\
  & & & Z  \arrow[from=dlll, dashed, "\text{geom limit}"] \arrow[from=dl, dashed] \arrow[from=dr] \arrow[from=dr] \arrow[from=drrr, "\text{approximation thm}"'] & & & \\
  DD(W_n^{\circ}) \arrow[rr, hookleftarrow, "\text{isometric inclusion}"'] & & W_n^{\circ} \arrow[rr, hookrightarrow, "\text{isometric inclusion}"'] \arrow[from=uu,"\text{drill $\Sigma_n$}", crossing over] & & W_n \arrow[rr,-,dashed, "< \delta_n", "\text{circle pack}"'] \arrow[from=uu,"\text{drill } \Sigma_n", crossing over] & & Z_n
\end{tikzcd}

\caption{The manifolds appearing in the statement and proof of \refthm{DrillingApproximation}. Hooked horizontal arrows represent isometric inclusions. Dashed horizontal lines represent a small quasiconformal deformation that produces a circle-packed manifold. Vertical arrows represent drilling out an embedded link. Solid diagonal arrows represent strong limits. Dashed diagonal arrows represent geometric limits only.}
\label{Fig:DrillingApproxOutline}
\end{figure}

Observe  that \refthm{FiniteVolApproxCor}, stated in the Introduction, is an immediate corollary of the above theorem.

See \reffig{DrillingApproxOutline} for a commutative diagram encapsulating the main objects in the statement of \refthm{DrillingApproximation}, as well as in its proof. We will begin with the top-right of the diagram, with the strong limit $Y_n \to Y$, and then construct the approximating manifolds $V_n$ and $V_n^\circ$ by proceeding right to left. We will then drill out an appropriate copy of $\Sigma$ from each of these manifolds and construct the limiting sequences in the bottom row of the diagram.

\begin{proof}[Proof of \refthm{DrillingApproximation}]
Let $\Gamma = \rho(\pi_1(M))$ be the Kleinian group associated to $Y$. Let $\{ O_n : n \in \NN \}$ be an open neighborhood system about $[\rho] \in AH(M,P)$.
By strong density, \refthm{Density}, there exists a strongly convergent sequence $\sigma_n \to \rho$, such that $[\sigma_n] \in O_n$. Let $\Gamma_n = \sigma_n(\pi_1(M))$, and let $Y_n = \HH^3 / \Gamma_n$ be the associated
geometrically finite hyperbolic 3-manifolds.
After passing to a subsequence, we can ensure that $d_{\rm Chaub}(\Gamma, \Gamma_n) < 2^{-n}$.
 By \refthm{Density}, the end invariants of $Y_n$ form a filling sequence, converging to the end invariants of $Y$.  
Fix an arbitrary basepoint $y \in Y$.

For each $n$, pick a constant $\delta_n > 0$ such that $\lim \delta_n = 0$. (In subsequent paragraphs, we will impose additional constraints, all of which hold when $\delta_n$ is sufficiently small.)
For each $n$, Brooks' \refthm{Brooks} says that there is a geometrically finite Kleinian group $\Gamma_{n, \delta_n}$ representing an $e^{\delta_n}$--quasiconformal deformation  of $\Gamma_n$, such that the conformal boundary of $\Gamma_{n, \delta_n}$ admits a circle packing. 
Let $\rho_n \from \pi_1(M) \to \PSL(2,\CC)$ be the associated representation. 
The bound of $e^{\delta_n}$ on the quasiconformal deformation means that the distance in  $\calT(\bdy_0 M)$ between the conformal end invariants of $\Gamma_n$ and $\Gamma_{n, \delta_n}$ is at most $\delta_n$. Here, we are measuring distances in the Teichm\"uller metric on the Teichm\"uller space of the (possibly disconnected) surface $\bdy_0 M = \bdy M - P$.
We choose $\delta_n$ small enough that 
\begin{equation}\label{Eqn:GammaChaub}
d_{\rm Chaub}(\Gamma_{n, \delta_n}, \Gamma_n) < 2^{-n}. 
\end{equation}
We also choose $\delta_n$ small enough to ensure $[\rho_n] \in O_n$; this is possible by \refthm{SameTopologyQH}. The restrictions on $\delta_n$ ensure that $\rho_n \to \rho$ is a strong limit.
Let $V_n = \HH^3 / \Gamma_{n, \delta_n}$ be the quotient manifold.

Next, we construct three sets of hyperbolic manifolds with pared manifold $(M - \calN(\Sigma), P \cup \bdy\calN(\Sigma))$:
\begin{itemize}
\item A hyperbolic manifold $Z$ whose end invariants agree with those of $Y$,
\item A hyperbolic manifold $Z_n$ whose end invariants agree with those of $Y_n$,
\item A hyperbolic manifold $W_n$ whose end invariants agree with those of $V_n$.
\end{itemize}
The pared manifold $(M - \calN(\Sigma), P \cup \bdy\calN(\Sigma))$ and the hyperbolic structure $Z$ are both 
 guaranteed to exist by \reflem{DrillHyperbolic}. Let $\xi \from \pi_1(M - \calN(\Sigma)) \to PSL(2,\CC)$ be the discrete faithful representation corresponding to $Z$, with image $\Delta = \xi ( \pi_1(M - \calN(\Sigma)))$. Let $\xi' \in AH(M - \calN(\Sigma), P \cup \bdy\calN(\Sigma))$ be a geometrically finite, minimally parabolic representation; this exists by \refthm{Hyperbolization}. 
 Thus, by \refthm{ABUniformization}, there exist representations $\tau_n \in QH(\xi')$ with image group $\Delta_n$ and quotient manifold $Z_n$, as well as representations $\xi_n \in QH(\xi')$ with image group $\Delta_{n, \delta_n}$ and quotient manifold $W_n$.

Let $Q_0 = AH(M - \calN(\Sigma), P \cup \bdy\calN(\Sigma))$. Let $\{ Q_n : n \in \NN \}$ be a nested system of open neighborhoods about  $[\xi] \in AH(M - \calN(\Sigma), P \cup \bdy\calN(\Sigma))$. Now, define a sequence $m(n)$ as follows. If $[\tau_n] = [\xi]$, let $m(n)=n$; otherwise, let $m(n)$ be the largest integer $m$ such that $[\tau_n] \in Q_{m}$. In either case, we have $[\tau_n] \in Q_{m(n)}$. We will check below that $[\tau_n] \to [\xi]$, which implies $m(n) \to \infty$.

Observe that $\bdy_0(M - \calN(\Sigma)) = \bdy_0 M$, because $\bdy \calN(\Sigma)$ is part of the parabolic locus. Since the (geometrically finite) end invariants of $Z_n$ and $W_n$ agree with those of $Y_n$ and $V_n$, respectively,  the distance in $\calT(\bdy_0(M - \calN(\Sigma)))$ between the conformal end invariants of $Z_n$ and $W_n$ is at most $\delta_n$. 
Recall that by the Ahfors--Bers Uniformization \refthm{ABUniformization},  $\calT(\bdy_0(M - \calN(\Sigma)))$ provides a local parametrization of 
 the interior $QH(\xi') \subset AH(M - \calN(\Sigma), P \cup \bdy\calN(\Sigma))$, and by \refthm{SameTopologyQH} the algebraic topology on $QH(\xi')$ agrees with the Chabauty topology. Thus a sufficiently small choice of $\delta_n$ ensures that $\xi_n$ has the following properties. We choose $\delta_n$ so that 
\begin{equation}\label{Eqn:DeltaChaub}
d_{\rm Chaub}(\Delta_{n, \delta_n}, \Delta_n) < 2^{-n}.
\end{equation}
and so that $[\xi_n] \in Q_{m(n)}$ for each $n$.
This completes the list of requirements that $\delta_n$ needs to satisfy.

Now, we proceed to verify that the sequences $V_n$ and $W_n$ have all of the required properties. Conclusion \refitm{CirclePack} holds by construction, since we have chosen $\Gamma_{n, \delta_n}$ so that
the conformal boundary of $V_n = \HH^n / \Gamma_{n, \delta_n}$ admits a circle packing. We call this circle packing $C_n$.

Equation~\refeqn{GammaChaub} implies that $\Gamma_{n, \delta_n} \to \Gamma$ in the Chabauty topology. Thus there is a choice of basepoints $v_n \in V_n$ giving a geometric limit $(V_n, v_n)  \to (Y, y)$. In particular, for any fixed $R$ and large $n$, there is an almost-isometric embedding $f_n \from (B_R(y), y) \hookrightarrow (V_n, v_n)$. Since $R$ is fixed, a sufficiently large choice of $n$ ensures that the image $f_n(B_R(y))$ will be contained in the scooped manifold $V_n^{\circ}$. Thus we also have a geometric limit $(V_n^\circ, v_n) \to (Y, y)$. Since $\rho_n \to \rho$ is a strong limit, \refprop{GeodesicsConverge} says that the homeomorphic image of $\Sigma$ is a link in $V_n$, isotopic to a geodesic link $\Sigma_n$ when $n \gg 0$. Observe that $\Sigma_n \subset CC(V_n) \subset V_n^\circ$, hence~\refitm{ScoopConverge} holds.

Conclusion~\refitm{ScoopedDrill} holds by construction, because $W_n$ has pared manifold $(M - \calN(\Sigma), P \cup \bdy\calN(\Sigma))$ and conformal boundary identical to that of $V_n$. Thus $C_n$ is also a circle packing on the conformal boundary of $W_n$. Observe that $\xi_n \in QH(\xi')$ is geometrically finite and minimally parabolic by construction. The homeomorphism $W_n \cong Y - \Sigma \cong V_n - \Sigma_n$ holds by \refitm{ScoopConverge}.

Next, we turn to conclusion~\refitm{DrilledLimit}. Recall that the end invariants of $\sigma_n$ form a filling sequence, converging to the end invariants $(X_1, \ldots, X_r, \lambda_{r+1}, \dots, \lambda_s)$ associated to $\rho$ and to $Y$. Thus the end invariants of $\tau_n$ (which are the same as those of $\sigma_n$) also form a filling sequence, converging to the end invariants $(X_1, \ldots, X_r, \lambda_{r+1}, \dots, \lambda_s)$ of $\xi$. Now, the approximation \refthm{NamaziSouto} combined with the ending lamination \refthm{ELC} says that (after passing to a subsequence) 
we have a strong limit $\tau_n \to \xi$. 
Having passed to this subsequence, we have $[\tau_n] \in Q_{m(n)}$ where $m(n) \to \infty$, as well as a Chabauty limit $\Delta_n \to \Delta$.

In \refeqn{DeltaChaub}, we chose $\delta_n$ so that $d_{\rm Chaub}(\Delta_{n, \delta_n}, \Delta_n) < 2^{-n}$. Thus 
 $\Delta_{n, \delta_n} \to \Delta$. Similarly, $\delta_n$ was chosen so that $[\xi_n] \in Q_{m(n)}$ for the same sequence $m(n) \to \infty$. Since $Q_{m(n)}$ is a nested system of open neighborhoods of $[\xi]$, it follows that $[\xi_n] \to [\xi]$ in the algebraic topology. Thus $\xi_n \to \xi$ is a strong limit.

Since the geometric topology is the Chabauty topology, for every $z \in Z$ there is a choice of basepoints $w_n$ such that $(W_n, w_n) \to (Z, z)$. Furthermore, as above, the almost-isometric image of $B_R(z)$ will be contained in the scooped manifold $W_n^\circ$ for large $n$, hence we also have a geometric limit $(W_n^\circ, w_n) \to (Z, z)$, hence \refitm{DrilledLimit} holds.

After double-doubling the scooped manifold $V_n^\circ$, as in \refdef{DoubleDouble}, we obtain a finite-volume hyperbolic manifold $DD(V_n^\circ)$. This finite-volume manifold contains a link $DD(\Sigma_n)$, consisting of four isometric copies of $\Sigma_n$.  Recall that by \refitm{ScoopedDrill}, $W_n$ is homeomorphic to $V_n - \Sigma_n$, and has identical conformal boundary admitting the same circle packing $C_n$. Applying \refdef{DoubleDouble} again, we may double $W_n^\circ$ twice, first in the blue faces and then the red, to obtain a finite-volume hyperbolic manifold homeomorphic to $DD(V_n^{\circ} - \Sigma_n)$. Thus, by Mostow--Prasad rigidity (or by the rigidity of circle packings, \refthm{CirclePackingRigidity}), we have isometries 
\[
DD(V_n^{\circ}) - DD(\Sigma_n) =  DD(V_n^{\circ} - \Sigma_n) = DD(W_n^\circ),
\]
establishing conclusion \refitm{DrillDoubleCommute}.

Finally, conclusion \refitm{FiniteVolumeLimit} is a corollary of \refitm{ScoopConverge} and \refitm{DrilledLimit}, because for any $R> 0$, a metric $R$--ball about $v_n \in DD(V_n^\circ)$ will in fact be contained in the original copy of $V_n^\circ$ for $n \gg 0$. A similar statement holds in $DD(W_n^\circ)$.
\end{proof}

\begin{remark}\label{Rem:GFDrillingApproximation}
If the hyperbolic manifold $Y$ is geometrically finite, the preceding proof becomes considerably more lightweight. In this special case, one can take constant sequences $Y_n = Y$ and $Z_n = Z$. Consequently, the strong density \refthm{Density} and the approximation \refthm{NamaziSouto} become unnecessary, as does the ending lamination theorem. Thurston's hyperbolization, \refthm{Hyperbolization}, becomes unnecessary because the representation $\xi$ corresponding to $Z$ is already geometrically finite and minimally parabolic. Finally, the realization \refthm{Realization}, which is used inside \reflem{DrillHyperbolic} to establish a hyperbolic structure on $(M - \calN(\Sigma), P \cup \bdy\calN(\Sigma))$ with the correct end invariants, can be replaced with the Ahlfors--Bers \refthm{ABUniformization}. Thus, in the geometrically finite case, the only tools required in the proof are \refthm{ABUniformization} and Brooks' \refthm{Brooks}.
\end{remark}

The final tool that we need to prove \refthm{EffectiveDrillingTame} is a finite-volume analogue of the same result. The following is a restatement of~\cite[Theorem~1.2]{FPS:EffectiveBilipschitz}.

\begin{theorem}[Effective drilling in finite volume, \cite{FPS:EffectiveBilipschitz}]\label{Thm:BilipFiniteVolume}
Fix  $0 < \epsilon \leq \log 3$ and $J>1$. Let $V$ be a finite-volume hyperbolic 3-manifold and $\Sigma$ a geodesic link in $V$ whose total length $\ell$ satisfies
\begin{equation*}
\ell \leq \min\left\{ \frac{\epsilon^5}{6771 \cosh^5(0.6 \epsilon + 0.1475)}, \, \frac{\epsilon^{5/2}\log(J)}{11.35} \right\}.
\end{equation*}
Then $V-\Sigma$ admits a complete hyperbolic metric $W$.
There are canonical $J$--bilipschitz inclusions
\[
\varphi \from V^{\geq \epsilon} \hookrightarrow W^{\geq \epsilon/1.2}, 
\qquad
\psi \from W^{\geq \epsilon} \hookrightarrow V^{\geq \epsilon/1.2}.
\]
The maps $\varphi$ and $\psi$ are equivariant with respect to the symmetry group of the pair $(V, \Sigma)$. 
\end{theorem}

We now have all the necessary tools to bootstrap from \refthm{BilipFiniteVolume} to \refthm{EffectiveDrillingTame}. The proof involves chasing the left half of the diagram in \reffig{DrillingApproxOutline}. Starting from $Y$, we will consider a circle-packed approximating manifold $V_n$, the finite-volume manifold $DD(V_n^\circ)$, its drilling $DD(W_n^\circ)$, and the scooped submanifold $W_n^\circ$ that approximates $Z$, the hyperbolic structure on $Y - \Sigma$.

\begin{proof}[Proof of \refthm{EffectiveDrillingTame}]
If $\vol(Y) < \infty$, the desired statement already appears in \refthm{BilipFiniteVolume}, substituting $V = Y$ and $W=Z$. For the rest of the proof, we assume $\vol(Y) = \infty$. 

Let $V_n$ and $W_n$ be the sequences of geometrically finite manifolds constructed in \refthm{DrillingApproximation}. In particular, every $V_n$ is homeomorphic to $Y$ and every $W_n$ is homeomorphic to $Z$. Furthermore,
the conformal boundaries of each $V_n$ and each $W_n$ admit the same circle packing $C_n$. By \refthm{DrillingApproximation}, we have strong limits $\rho_n \to \rho$ (corresponding to $V_n \to Y$) and $\xi_n \to \xi$ (corresponding to $W_n \to Z$).

Recall that by \refthm{MargulisEstimate}, $\epsilon < \log 3$ is a Margulis number for every infinite-volume hyperbolic $3$--manifold, hence $Y^{< \epsilon}$ is a disjoint union of tubes and horocusps. We let $T^{< \epsilon}$ denote a component of $\bdy Y^{< \epsilon}$; this is either a horotorus about a cusp or an equidistant torus about a short geodesic. We can choose a number $\eta \in (\epsilon, 2\epsilon)$ such that $\eta$ is still a Margulis number for $Y$.

Fix a basepoint $y \in Y$ so that $2 \injrad(y) = \eta \in (\epsilon, 2\epsilon)$, and furthermore $y$ lies on an embedded, $\eta$--thick equidistant torus $T^\eta(\sigma)$ about the first component $\sigma \subset \Sigma$. Such a choice of $y \in Y^{> \epsilon}$ is possible because $\eta$ is a Margulis number for $Y$, hence the thick part $Y^{> \eta}$ is non-empty. Similarly, fix a basepoint $z \in Z$ so that $2 \injrad(z) = \epsilon$, and furthermore $z$ lies on an embedded $\epsilon$--thick horotorus $T^\epsilon(\sigma)$ that bounds an embedded horocusp about the same component $\sigma \subset \Sigma$. Again, such a choice is possible because $\epsilon$ is a Margulis number for $Z$.

Now, \refthm{DrillingApproximation} says that for $n \gg 0$, there exist choices of basepoints $v_n \in V_n^\circ$ and $w_n \in W_n^\circ$, such that $(V_n^\circ, v_n) \to (Y, y)$ and $(W_n^\circ, w_n) \to (Z, z)$.
Furthermore, the homeomorphic image of $\Sigma$ is a link in $V_n$, isotopic to a geodesic link $\Sigma_n$ when $n \gg 0$. Note that $\Sigma_n \subset CC(V_n) \subset V_n^\circ$.

After doubling $V_n^{\circ}$ twice to obtain the finite-volume manifold $DD(V_n^{\circ})$, as in \refdef{DoubleDouble}, we also obtain a geodesic link $DD(\Sigma_n) \subset DD(V_n^\circ)$ consisting of four isometric copies of $\Sigma_n$. 
By \refprop{GeodesicsConverge}, the strong limit $\rho_n \to \rho$ means that $\Sigma_n \to \Sigma$. Thus for large $n$, the length of $\Sigma_n$ is arbitrarily close to $\len(\Sigma) = \ell$. In particular, we have
\begin{equation}\label{Eqn:FactorOf4Length}
\len(DD(\Sigma_n)) = 4 \len(\Sigma_n) \leq  \min\left\{ \frac{\epsilon^5}{6771 \cosh^5(0.6 \epsilon + 0.1475)}, \, \frac{\epsilon^{5/2}\log(J)}{11.35} \right\}. 
\end{equation}
Thus we may apply the finite-volume effective drilling result, \refthm{BilipFiniteVolume}, to $DD(V_n^{\circ})$ and  $DD(\Sigma_n)$. 
For the unique hyperbolic metric on $DD(V_n^{\circ}) \!-\! DD(\Sigma_n)$, \refthm{BilipFiniteVolume} gives $J$--bilipschitz inclusions 
\begin{align*} 
\varphi_n &\from DD(V_n^{\circ})^{\geq \epsilon} \longrightarrow \big( DD(V_n^{\circ}) \!-\! DD(\Sigma_n) \big)^{\geq \epsilon /1.2}, \\
\psi_n &\from \big( DD(V_n^{\circ}) \!-\! DD(\Sigma_n) \big)^{\geq \epsilon } \longrightarrow DD(V_n^{\circ})^{\geq \epsilon /1.2} .
\end{align*}
Furthermore, $\varphi_n$ and $\psi_n$ respect the symmetries of the pair $(DD(V_n^{\circ}), DD(\Sigma_n))$.

The pair $(DD(V_n^{\circ}), DD(\Sigma_n))$ has a  $\ZZ_2 \times \ZZ_2$ group of symmetries, where the generator of the first $\ZZ_2$ acts by reflection in the blue faces of $V_n^{\circ}$ and the generator of the second $\ZZ_2$ acts by reflection in the red faces. This action restricts to a $\ZZ_2 \times \ZZ_2$ group of symmetries of $DD(V_n^{\circ})-DD(\Sigma_n) = DD(V_n^{\circ} - \Sigma_n)$, with a fundamental domain of the form $V_n^{\circ}-\Sigma_n$. Since $\varphi_n$ and $\psi_n$ respect these symmetries, we obtain $J$--bilipschitz inclusions
\[
\varphi_n \from (V_n^{\circ})^{\geq \epsilon} \to ( V_n^{\circ}  \!-\! \Sigma_n )^{\geq \epsilon /1.2}, \qquad
\psi_n \from ( V_n^{\circ} \!-\! \Sigma_n )^{\geq \epsilon } \to (V_n^{\circ})^{\geq \epsilon /1.2} ,
\]
isotopic to the topological drilling of $\Sigma_n$. By \refthm{DrillingApproximation}, we have $V_n^{\circ}-\Sigma_n \cong W_n^\circ$.

We can now construct the $J$--bilipschitz inclusions $\varphi \from Y^{\geq \epsilon} \hookrightarrow Z^{\geq \epsilon/1.2}$ and
$\psi \from Z^{\geq \epsilon} \hookrightarrow Y^{\geq \epsilon/1.2}$ using \refthm{ConstructBilip}. Most of the hypotheses of that theorem have already been verified. We have geometrically convergent sequences $(V_n^\circ, v_n) \to (Y,y)$ and $(W_n, w_n) \to (Z, z)$. We have $y \in Y^{> \epsilon}$ as required. Since injectivity radii converge in a geometric limit \cite[Lemma 3.2.6]{ceg:notes-on-notes}, it follows that  $v_n \in (V_n^\circ)^{\geq \epsilon}$ for large $n$. For large $n$, we have a $J$--bilipschitz inclusion $\varphi_n \from (V_n^{\circ})^{\geq \epsilon} \to ( W_n^{\circ})^{\geq \epsilon /1.2}$. To apply \refthm{ConstructBilip}, it remains to check that $d(\varphi_n(v_n), w_n)$ is uniformly bounded.

This can be checked as follows. By construction, the basepoint $y \in Y$ lies on an equidistant torus $T^\eta$ about $\sigma \subset \Sigma$, where $\epsilon < 2\injrad(y) < 2\epsilon$. By \cite[Lemma 3.2.6]{ceg:notes-on-notes}, the same two-sided bound holds for $\injrad(v_n)$ for large $n$. By \cite[Theorem 9.30]{FPS:EffectiveBilipschitz},  of which \refthm{BilipFiniteVolume} is a corollary, the injectivity radius of $v_n \in (V_n^\circ)^{\geq \epsilon}$ changes by a factor of at most $1.2$ under $\varphi_n$. Thus
\[
\epsilon/1.2 < 2 \injrad(\varphi_n(v_n)) = \mu_n < 2.4 \epsilon,
\]
and furthermore $\varphi_n(v_n)$ still lies on a $\mu_n$--thin horospherical torus about the same component $\sigma \subset \Sigma$. Meanwhile, observe that $2\injrad(w_n)$ is nearly equal to $\epsilon$, and $w_n$ also lies on a horospherical torus about $\sigma \subset \Sigma$. The diameter of each of those tori is uniformly bounded, by the geometric limit $(W_n, w_n) \to (Z, z)$, while the distance between the $\mu_n$--thin and $\epsilon$--thin tori is uniformly bounded by \cite[Proposition 1.4]{FPS:Tubes}. Thus $d(\varphi_n(v_n), w_n)$ is uniformly bounded. Hence  \refthm{ConstructBilip} gives a $J$--bilipschitz inclusion $\varphi \from Y^{\geq \epsilon} \hookrightarrow Z^{\geq \epsilon/1.2}$.

The reverse inclusion $\psi \from Z^{\geq \epsilon} \hookrightarrow Y^{\geq \epsilon/1.2}$ is constructed in exactly the same way, tracing points backwards to check the hypotheses of  \refthm{ConstructBilip}. The points $v_n$ and $\psi_n(w_n)$ lie on equidistant tori in a tube about $\sigma \subset \Sigma$, hence $d(v_n, \psi_n(w_n))$ is uniformly bounded by \cite[Proposition 5.7]{FPS:Tubes}.
\end{proof}

\subsection{Extended and strengthened results}
The above proof method can be used to extend Brock and Bromberg's \refthm{BBDrillingThm} to all tame hyperbolic manifolds without rank-1 cusps. 

\begin{theorem}\label{Thm:GeoInfiniteBB}
Fix $J > 1$ and $\epsilon > 0$, where $\epsilon$ is smaller than the Margulis constant $\epsilon_3$.  Then there is a number $\ell_0 = \ell_0(\epsilon, J) > 0$ such that the following holds for every tame hyperbolic 3-manifold $Y$ without rank-1 cusps. 
Suppose that $\Sigma \subset Y$ is a geodesic link, whose total length is less than $\ell_0$. Then $Y - \Sigma$ admits a hyperbolic structure $Z$ with the same end invariants as those of $Y$.
Furthermore, the inclusion
\[ 
\iota \from Z \hookrightarrow Y
\]
restricts to a $J$--bilipschitz diffeomorphism on the complement of $\epsilon$--thin tubes about $\Sigma$. 
\end{theorem}

In contrast with \refthm{EffectiveDrillingTame}, this result is not effective because the function $\ell_0(\epsilon, J)$ is not explicitly given. On the other hand, this result has the advantage that $J$--bilipschitz control extends into the thin parts of $Y$ and $Z$ that do not correspond to components of $\Sigma$.

For the following proof, it is helpful to consult the commutative square in \reffig{DrillingApproxOutline} with corners at $Y$, $V_n$, $W_n$, and $Z$.

\begin{proof}[Proof of \refthm{GeoInfiniteBB}]
Fix $\epsilon < \epsilon_3$, and recall from \refsec{TopologyGeometry} that $\epsilon_3 < 1$. 
One of the first constraints in Brock and Bromberg's proof of \refthm{BBDrillingThm}  is that $\ell_0 < \epsilon$. 

Let $Y$ be a tame hyperbolic 3-manifold without rank-1 cusps. We may assume that $Y$ is geometrically infinite, as otherwise \refthm{BBDrillingThm} already gives the desired conclusion. Let  $\Sigma \subset Y$  be a geodesic link with $\len(\Sigma) < \ell_0 < \epsilon$. Then \reflem{DrillHyperbolic} implies that $Y - \Sigma$ admits a hyperbolic structure $Z$ with the same invariants. 
\refthm{DrillingApproximation} constructs (type-preserving) geometric limits $(V_n, v_n) \to (Y, y)$ and $(W_n, w_n) \to (Z, z)$, where $W_n$ is obtained by drilling a geodesic link $\Sigma_n$ from $V_n$, and where $\Sigma_n \to \Sigma$ in the geometric limit.


We introduce the following notation. Let $Y_{\Sigma}^{< \epsilon}$ be the union of the $\epsilon$--thin Margulis tubes about the components of $\Sigma$. This is a disjoint union, because $\len(\Sigma) < \epsilon < \epsilon_3$. Let $Z_{\Sigma}^{< \epsilon}$ be the union of $\epsilon$--thin cusps in $Z$ corresponding to components of $\Sigma$. Define $(V_n)_{\Sigma}^{< \epsilon}$ and $(W_n)_{\Sigma}^{< \epsilon}$ similarly.

Since $V_n$ is geometrically finite, and $\len(\Sigma_n) < \ell_0 = \ell_0(\epsilon, J)$ for all sufficiently large $n$, \refthm{BBDrillingThm} says that the homeomorphism $V_n - \Sigma_n \to W_n$ induces a $J$--bilipschitz diffeomorphism 
\[
\varphi_n \from V_n - (V_n)_{\Sigma}^{< \epsilon} \longrightarrow W_n - (W_n)_{\Sigma}^{< \epsilon} 
\]
As in the  proof of \refthm{EffectiveDrillingTame}, the distance estimate \cite[Proposition 1.4]{FPS:Tubes} ensures that basepoints on equidistant tori about a component $\sigma \subset \Sigma$ do not escape under $\varphi_n$. Then \refthm{ConstructBilip} and \refrem{ConstructBilipThin}
imply that the $J$--bilipschitz maps $\varphi_n$ converge (after passing to a subsequence) to the desired $J$--bilipschitz map $ \varphi \from Y - Y_{\Sigma}^{< \epsilon} \longrightarrow Z - Z_{\Sigma}^{< \epsilon}$.
\end{proof}

We also discuss a quantitative strengthening of \refthm{EffectiveDrillingTame}.

\begin{remark}\label{Rem:FactorOf4}
An astute reader will notice that the quantitative hypothesis on $\ell = \len(\Sigma)$ in \refthm{EffectiveDrillingTame} differs from the hypothesis in \refthm{BilipFiniteVolume} by a factor of $4$. This discrepancy occurs because we need to double-double the scooped manifold $V_n^\circ$ to obtain a finite-volume manifold $DD(V_n^\circ)$. See Equation~\refeqn{FactorOf4Length}, where the transition from $\len(\Sigma)$ to $\len(DD(\Sigma_n))$ happens.

One may ask whether paying a factor of 4 is a necessary price in adapting finite-volume results in the infinite-volume setting. We suspect that that the factor of 4 can probably be eliminated, at the cost of a different price: reexamining some of the technical analytic estimates that were used to prove \refthm{BilipFiniteVolume}. Here are the two most salient points.

In the cone deformation theory of Hodgson and Kerckhoff \cite{hk:univ, hk:shape}, the length of the geodesic link $\Sigma$ is used to control the radius $R$ of an embedded tube about $\Sigma$; in turn, this tube radius is used to control almost all other quantities. (See the discussion around \refdef{Haze} for a more quantitative summary.)
Hodgson and Kerckhoff proved the original radius bound \cite[Theorem 5.6]{hk:shape}, and we adapted their proof in \cite[Theorem 4.21]{FPS:EffectiveBilipschitz}. The proof is essentially a packing argument. In the context of the symmetric manifold $DD(V_n^\circ)$, one should be able to adapt the argument to use $\len(\Sigma_n)$ rather than $\len(DD(\Sigma_n))$, provided that the tubes in different copies of $V_n^\circ$ do not meet during the cone deformation. This can probably be ensured, because the tubes stay in the core portion of $V_n^\circ$, whereas the circle-packing and scooping happen deep into the ends.

Since the length of $\Sigma$ and its cone angle $\alpha$ both change throughout the cone deformation, the change in length must itself be controlled. A key differential inequality, proved in \cite[Proposition 5.5 and page 1079]{hk:shape} and restated in \cite[Lemma 6.7]{FPS:EffectiveBilipschitz}, bounds the change in the ratio $\alpha/ \len(\Sigma)$ in terms of functions of the tube radius $R$. This estimate does not scale correctly when we pass from $V_n^\circ$ to $DD(V_n^\circ)$, because the length gets quadrupled but the radius stays the same. Thus removing the factor of 4 would also require adapting the proof of \cite[Proposition 5.5]{hk:shape} to work directly in $V_n^\circ$, thought of as a hyperbolic orbifold with mirrored boundary. This can be done, because the proof of \cite[Proposition 5.5]{hk:shape} is essentially an application of the Cauchy--Schwartz inequality.

Provided the above technical points are addressed, the rest of the estimates from \cite{FPS:EffectiveBilipschitz} should go through unchanged,
with $\len(\Sigma)$ in place of $\len(DD(\Sigma_n))$. 
\end{remark}

\section{The effective filling theorem}\label{Sec:Filling}

In this section, we prove \refthm{EffectiveFillingTame}, giving effective bilipschitz bounds on Dehn fillings of tame hyperbolic 3-manifolds. The proof follows the same two-step process as in the previous section.
The first step, \refthm{FillingApproximation}, is an analogue of \refthm{DrillingApproximation}. It uses a number of results from Kleinian groups to approximate both the drilled and filled manifolds with sequences of geometrically finite manifolds whose conformal boundaries admit a circle packing. See \reffig{FillingApproxOutline} for a visual preview. The second step uses the sequences constructed in \refthm{FillingApproximation}, together with a finite-volume bilipschitz theorem (\refthm{EffectiveFillFiniteVolume}, proved in \cite{FPS:EffectiveBilipschitz}), to complete the proof of \refthm{EffectiveFillingTame}.

The first step can be summarized by the following theorem, which is analogous to \refthm{FiniteVolApproxCor}, but in the filling rather than drilling case.

\begin{theorem}\label{Thm:FillingFteVolApproxCor}
Let $Z$ be a tame, infinite-volume hyperbolic 3-manifold with a fixed collection of rank-2 cusps, and with fixed slopes on those cusps of total normalized length at least $L^2\geq 230.1$. Then the Dehn filling of $Z$ along those slopes is a tame manifold that admits a hyperbolic metric $Y$ with the same end invariants as those of $Z$, and a geodesic link $\Sigma \subset Y$ such that $Z$ is homeomorphic to $Y-\Sigma$. Furthermore, there is a sequence of finite-volume approximating manifolds $DD(W_n^\circ)$ and $DD(V_n^\circ)$ with the following properties:
\begin{enumerate}[\: \: $(1)$]
\item The manifold $DD(V_n^\circ)$ contains a geodesic link $DD(\Sigma_n)$, consisting of four isometric copies of a link $\Sigma_n$, such that $DD(W_n^\circ) = DD(V_n^\circ) - DD(\Sigma_n)$.
\item For any choice of basepoints $y \in Y$ and $z \in Z$, there are basepoints in the approximating manifolds such that $(DD(V_n^\circ), v_n) \to (Y, y)$ and $(DD(W_n^\circ), w_n) \to (Z, z)$ are geometric limits.
\item In the geometric limit $(DD(V_n^\circ), v_n) \to (Y, y)$, we have $\Sigma_n \to \Sigma$.
\end{enumerate}
\end{theorem}

Indeed, \refthm{FillingFteVolApproxCor} is an immediate consequence of the following more detailed result.

\begin{theorem}\label{Thm:FillingApproximation}
Let $M$ be a compact 3-manifold, and $P \subset \bdy M$ a collection of annuli and tori.
Let $\Sigma \subset M$ be a smooth link with regular neighborhood $\calN(\Sigma)$. Suppose that $(M-\calN(\Sigma), P\cup \bdy \calN(\Sigma))$ is a pared manifold that admits an infinite-volume hyperbolic structure $Z$, uniformized by a representation $\xi$, where the total normalized length of the meridians of $\Sigma$ satisfies $L^2 \geq 230.1$.

Then $(M,P)$ is a pared manifold that admits a hyperbolic structure $Y$, uniformized by a representation $\rho \in AH(M,P)$, 
with the same end invariants as those of $Z$.
Furthermore, there exist approximating sequences such that the following hold for all $n \gg 0$:
\begin{enumerate}[\: \: $(1)$]
\item\label{Itm:CirclePackCusped} There is a geometrically finite, minimally parabolic representation $\xi_n \in AH(M-\calN(\Sigma), \allowbreak P\cup \bdy \calN(\Sigma))$, such that the 3-manifold $W_n = \HH^3 / \xi_n(\pi_1(M - \Sigma))$ has conformal boundary admitting a circle packing $C_n$. Furthermore, $\xi_n \to \xi$ is a strong limit.
\smallskip
\item\label{Itm:ScoopConvergeCusped} For every $z \in Z$, there is a choice of basepoints $w_n \in W_n^\circ$, such that $(W_n^\circ, w_n)$ converges geometrically to $(Z,z)$. 
\smallskip
\item\label{Itm:ScoopedFill} There is a geometrically finite, minimally parabolic representation $\rho_n \in AH(M,P)$ such that the associated hyperbolic 3-manifold $V_n$ has end invariants that are identical to those of $W_n$. In particular, the conformal boundary of $V_n$ admits the same circle packing $C_n$. Furthermore, $\Sigma \subset M$ is isotopic to a geodesic link $\Sigma_n \subset V_n^\circ$.
\smallskip
\item\label{Itm:FilledLimit} For every $y \in Y$, there is a choice of basepoints $v_n \in V_n^\circ$, such that $(V_n^\circ, v_n)$ converges geometrically to $(Y, y)$.  In the geometric limit, we have $\Sigma_n \to \Sigma_\infty$, where $\Sigma_\infty \subset Y$ is a geodesic link isotopic to $\Sigma$. Furthermore, there is a strong limit $\rho_n \to \rho$, where $\rho \in AH(M,P)$ is a representation associated to $Y$.
\smallskip
\item\label{Itm:FillDoubleCommute} The operations of filling and double-doubling commute. That is:
\[
DD(V_n^\circ) - DD(\Sigma_n) \cong DD(W_n^\circ).
\]

\smallskip
\item \label{Itm:FillFiniteVolumeLimit} The manifolds $Y$ and $Z$ are geometric limits of finite-volume hyperbolic manifolds, as follows: $(DD(V_n^\circ), v_n) \to (Y, y)$ and $(DD(W_n^\circ), w_n) \to (Z, z)$.
\end{enumerate}
\end{theorem}

See \reffig{FillingApproxOutline} for a visual summary of the theorem. 

\begin{figure}
\begin{tikzcd}[column sep=4.0em, row sep=3.0em]
& & & Y \arrow[from=dlll, dashed, "\text{geom limit}"] \arrow[from=dl, dashed] \arrow[from=dr] \arrow[from=dr] \arrow[from=drrr, "\text{approximation thm}"'] \arrow[from=dd, "\text{fill }"] & & & \\
  DD(V_n^{\circ}) \arrow[rr, hookleftarrow, "\text{isometric inclusion}"']  & & V_n^{\circ} \arrow[rr, hookrightarrow, "\text{isometric \: inclusion}"', crossing over]  & & V_n  \arrow[rr,-,dashed, "< \delta_n", "\text{circle pack}"']  & & Y_n \\
  & & & Z  \arrow[from=dlll, dashed, "\text{geom limit}"] \arrow[from=dl, dashed] \arrow[from=dr] \arrow[from=dr] \arrow[from=drrr, "\text{strong density}"'] & & & \\
  DD(W_n^{\circ}) \arrow[rr, hookleftarrow, "\text{isometric inclusion}"'] \arrow[uu, "\text{finite vol}", "\text{fill}"'] & & W_n^{\circ} \arrow[rr, hookrightarrow, "\text{isometric inclusion}"'] \arrow[uu,"\text{fill}"', crossing over] & & W_n \arrow[rr,-,dashed, "< \delta_n", "\text{circle pack}"'] \arrow[uu,"\text{fill}"', crossing over] & & Z_n  \arrow[uu,"\text{fill}"']
\end{tikzcd}

\caption{The manifolds appearing in the statement and proof of \refthm{FillingApproximation}. Hooked horizontal arrows represent isometric inclusions. Dashed horizontal lines represent a small quasiconformal deformation that produces a circle-packed manifold. Vertical arrows represent Dehn filling. Solid diagonal arrows represent strong limits. Dashed diagonal arrows represent geometric limits only.}
\label{Fig:FillingApproxOutline}
\end{figure}

\begin{proof}
The proof is very similar to the proof of \refthm{DrillingApproximation}, with the notable difference  that drilling is replaced by filling. This change of direction means that some more work is required to ensure that the filled manifolds are hyperbolic and contain geodesic links representing $\Sigma$.

We begin by verifying the existence of a hyperbolic manifold $Y$ with pared manifold $(M,P)$ and the same end invariants as those of $Z$. Let $H_1, \ldots, H_k$ be disjoint horocusps in $Z$ about the components $\sigma_1, \ldots, \sigma_k$ of $\Sigma$. By a theorem of Meyerhoff \cite[Section 5]{meyerhoff}, the $H_i$ can be chosen so that $\area(\bdy H_i) \geq \sqrt{3}/2$ for every $i$. Let $s_i$ be a Euclidean geodesic on $\bdy H_i$ representing the meridian of $\sigma_i$. Let $L_i = \len(s_i)/\sqrt{\area(\bdy H_i)}$ be the normalized length of $s_i$, and $L = L(\ss)$ the total normalized length of $\ss = (s_1, \ldots, s_k)$. Thus our hypothesis $L(\ss)^2 \geq 230.1$ combined with \refeqn{NormalizedLength} implies
\begin{equation}
\frac{1}{230.1} \geq \frac{1}{L(\ss)^2} \geq \frac{1}{ L(s_i)^2} = \frac{\area(\bdy H_i)}{\len(s_i)^2} \geq \frac{\sqrt{3}/2}{\len(s_i)^2} \, ,
\end{equation}
hence $\len(s_i) > 14$ for each $i$. Consequently, \refthm{6TheoremInf} tells us that $(M, P)$ is a pared manifold that admits a hyperbolic structure $Y$ with  end invariants matching those of $Z$. 

Starting from $Z$, we apply strong density, \refthm{Density}, to find a sequence of geometrically finite, minimally parabolic representations $\tau_n \in AH(M-\calN(\Sigma), P\cup \bdy \calN(\Sigma))$, such that $\tau_n \to \xi$ is a strong limit.
The end invariants of the associated hyperbolic manifolds $Z_n$ form a filling sequence, converging to the end invariants of $Z$. The Kleinian groups $\Delta_n$ associated to $Z_n$ converge (both geometrically and algebraically) to the Kleinian group $\Delta$ associated to $Z$. In particular, we have $\Delta_n \to \Delta$ in the Chabauty topology. We also choose an open neighborhood system $Q_n$ about $[\xi]$, such that $[\tau_n] \in Q_n$.

For each $n$, we choose a positive constant $\delta_n$ such that $\lim \delta_n = 0$ and several more constraints (specified below) are all satisfied. For each $n$, Brooks' \refthm{Brooks} says that there is a geometrically finite Kleinian group $\Delta_{n, \delta_n}$ representing an $e^{\delta_n}$--quasiconformal deformation of $\Delta_n$, and such that the conformal boundary of $\Delta_{n, \delta_n}$ admits a circle packing. By picking $\delta_n$ small enough, we ensure that $d_{\rm Chaub}(\Delta_{n, \delta_n}, \Delta_n) < 2^{-n}$. Picking $\delta_n$ small enough also ensures that the associated representation $\xi_n$  satisfies $[\xi_n] \in Q_n$. These choices imply that $\xi_n \to \xi$ is a strong limit. We let $W_n = \HH^3 / \Delta_{n, \delta_n}$.

Next, we fill in the meridians of $\Sigma$ to recover the pared manifold $(M,P)$. We obtain three sets of hyperbolic structures on $(M,P)$:
\begin{itemize}
\item A hyperbolic structure $Y$, whose end invariants agree with those of $Z$,
\item A hyperbolic structure $Y_n$, whose end invariants agree with to those of $Z_n$,
\item A hyperbolic structure $V_n$, whose end invariants agree with those of $W_n$.
\end{itemize}
Since $(M,P)$ is a pared manifold, \refthm{Hyperbolization} says there is also a geometrically finite, minimally parabolic representation $\rho' \in AH(M,P)$. The hyperbolic structures $Y_n$ and $V_n$, represented by $\sigma_n \in QH(\rho')$ and $\rho_n \in QH(\rho')$, respectively, exist by \refthm{ABUniformization}.

By construction, the conformal boundary of $V_n$ is an $e^{\delta_n}$--quasiconformal deformation of the conformal boundary of $Y_n$. Thus the distance in $\calT(\bdy_0 M)$ between the conformal boundaries of those manifolds is at most $\delta_n$. By choosing $\delta_n$ small enough, we ensure that $Y_n$ and $V_n$ are represented by Kleinian groups $\Gamma_n$ and $\Gamma_{n, \delta_n}$, respectively, such that
$d_{\rm Chaub} (\Gamma_{n, \delta_n}, \Gamma_n) < 2^{-n}$. We also ensure that, for a neighborhood system $\{ O_n \}$ about $[\rho]$, we have $\rho_n \in O_{m(n)}$ whenever $\rho \neq \sigma_n \in O_{m(n)}$, and $\rho_n \in O_n$ whenever $\sigma_n = \rho$.
This completes the list of conditions that $\delta_n$ needs to satisfy.

Now, we check the conclusions of the proposition. Conclusion \refitm{CirclePackCusped} holds by construction, because we have chosen $\Delta_{n, \delta_n}$ so that the conformal boundary of $W_n = \HH^3/\Delta_{n, \delta_n}$ admits a circle packing $C_n$, and because we have chosen $\delta_n$ small enough to ensure  $\xi_n \to \xi$ is a strong limit.

For conclusion \refitm{ScoopConvergeCusped}, recall that we have already checked that $\Delta_{n, \delta_n} \to \Delta$ in the Chabauty topology. Thus, for every basepoint $z \in Z$, there exist basepoints $w_n \in W_n$ such that $(W_n, w_n) \to (Z, z)$. For any fixed radius, the ball $B_R(w_n)$ will lie in the scooped manifold $W_n^\circ$ when $n \gg 0$, hence the sequence of scooped manifolds $(W_n^\circ, w_n)$ also converges geometrically to $(Z, z)$.

We will prove conclusions \refitm{ScoopedFill} and \refitm{FillDoubleCommute} together. Observe that since the conformal boundary of $V_n$ agrees with that of $W_n$, it also admits the same circle packing $C_n$. Thus we may double-double the scooped manifolds $V_n^\circ$ and $W_n^\circ$ to obtain finite-volume manifolds $DD(V_n^\circ)$ and $DD(W_n^\circ)$, respectively. The scooped manifold $W_n^\circ$ contains a tuple of slopes $\ss^n = (s_1^n, \ldots, s_k^n)$ corresponding to the meridians of $\Sigma$, with the property that $L(s_i^n) \to L(s_i)$ in the geometric limit $(W_n^\circ, w_n) \to (Z, z)$. Thus, for $n \gg 0$, we have $L(\ss^n)^2 > 230.08$. Thus the finite-volume manifold 
 $DD(W_n^\circ)$ contains a tuple of slopes $DD(\ss^n)$, containing four copies of each meridian of $\Sigma$. The definition of normalized length in \refeqn{NormalizedLength} implies that for all $n \gg 0$, we have
 \begin{equation}\label{Eqn:DoubleDoubledNormLength}
 \frac{1}{ L(DD(\ss^n))^2} = \frac{4}{L(\ss^n)^2} < \frac{4}{230.08} = \frac{1}{57.52},
 \quad \text{hence} \quad
 L(DD(\ss^n)) > \sqrt{57.52} > 7.584.
 \end{equation}
 Consequently, a theorem of Hodgson and Kerckhoff \cite[Theorem 1.2]{hk:shape}, implies that Dehn filling  $DD(W_n^\circ)$ along the tuple of slopes $DD(\ss^n)$ produces a hyperbolic manifold $U_n$, where the union of cores of the Dehn filling solid tori is a geodesic link $\Upsilon_n$. By  \cite[Corollary 5.13]{hk:shape}, the total length of $\Upsilon_n$ is less than $0.16$. 

Now, recall that each $V_n$ is homeomorphic to $M$, and contains a smoothly embedded copy of $\Sigma$, such that $W_n$ is homeomorphic to $V_n - \Sigma$. After an isotopy, we may assume that the embedded copy of $\Sigma$ lies in $V_n^\circ$, hence $DD(V_n^\circ)$ contains the double-doubled link $DD(\Sigma)$. Since $\bdy V_n^\circ$ and $\bdy W_n^\circ$ contain exactly the same pattern of red and blue faces, the double-doubling construction shows that $DD(W_n^\circ)$ is homeomorphic to 
$DD(V_n^\circ) - DD(\Sigma)$ by a homeomorphism that carries $DD(\ss^n)$ to the tuple of meridians of $DD(\Sigma)$. The above homeomorphism extends to a homeomorphism of pairs $(U_n, \Upsilon_n) \to(DD(V_n^\circ) , DD(\Sigma))$.
Since $DD(V_n^\circ)$ is Haken, Waldhausen's topological rigidity theorem \cite[Theorem~7.1]{Waldhausen:SufficientlyLarge} implies that
$DD(\Sigma)$ is isotopic to the geodesic link $\Upsilon_n$.
The components of $\Upsilon_n$ that lie in the original copy of $V_n^\circ$ form a geodesic link $\Sigma_n$, with the property that
\[
DD(V_n^\circ) - DD(\Sigma_n) \cong DD(V_n^\circ - \Sigma_n) \cong DD(W_n^\circ),
\] 
proving \refitm{ScoopedFill} and \refitm{FillDoubleCommute}. We note that $\len(\Sigma_n) = \len(\Upsilon_n)/4 < 0.04$.

Conclusion \refitm{FilledLimit} is established exactly as in \refthm{DrillingApproximation}.
By construction, the end invariants of $Y_n$ agree with those of $Z_n$. Thus those end invariants form a filling sequence that limits to the end invariants of $Y$. Thus, by the approximation \refthm{NamaziSouto}, the representations $\sigma_n$ associated to $Y_n$ converge strongly (after a subsequence) to a representation of a hyperbolic manifold homeomorphic to $Y$ and having the same end invariants as $Y$. By  the ending lamination \refthm{ELC}, we have a strong limit $\sigma_n \to \rho$ for a representation $\rho$ associated to $Y$.
In particular, $\Gamma_n \to \Gamma$ in the Chabauty topology and $[\sigma_n] \in O_{m(n)}$ for a sequence $m(n) \to \infty$.

By our choice of $\delta_n$, we have a Chabauty limit $\Gamma_{n, \delta_n} \to \Gamma$ and $[\rho_n] \in O_{m(n)}$ for the same sequence $m(n)$, hence $\rho_n \to \rho$ is a strong limit. As above, we can fix a basepoint $y \in Y$ and then find basepoints $v_n \in V_n$, which lie in $V_n^\circ$ for $n \gg 0$, such that $(V_n^\circ, v_n) $ converges geometrically to $(Y, y)$. Since $\rho_n \to \rho$ is a strong limit, and $\len(\Sigma_n) < 0.04$ for all $n \gg 0$, \refprop{GeodesicsConverge} says that $\Sigma_n \to \Sigma_\infty \subset Y$, a geodesic link isotopic to $\Sigma$.

Finally, conclusion \refitm{FillFiniteVolumeLimit} is a corollary of \refitm{ScoopConvergeCusped} and \refitm{FilledLimit}.
\end{proof}

\begin{remark}\label{Rem:GFFillingApproximation}
If the end invariants of $Y$ (equivalently, the end invariants of $Z$) are geometrically finite, the preceding proof becomes considerably more lightweight. In this case, one can take constant sequences $Y_n = N$ and $Z_n = Z$. Thus, as in \refrem{GFDrillingApproximation}, the hyperbolization, strong density, ending lamination, and approximation theorems become unnecessary. Thus, in the geometrically finite case, the only tools required in the proof are \refthm{ABUniformization} and Brooks' \refthm{Brooks}.
\end{remark}

Next, we record a finite-volume version of \refthm{EffectiveFillingTame}, with some additional information. The following result is 
 \cite[Corollary 9.34]{FPS:EffectiveBilipschitz}. As in \refrem{FactorOf4}, the hypotheses of \refthm{EffectiveFillFiniteVolume} differ from those of \refthm{EffectiveFillingTame} by a factor of 4.

\begin{theorem}[Effective filling in finite volume, \cite{FPS:EffectiveBilipschitz}]\label{Thm:EffectiveFillFiniteVolume}
Fix any $0 < \epsilon \leq \log 3$ and any $J>1$. Let $M$ be a 3-manifold with empty or toroidal boundary, and $\Sigma$ a link in $M$. Suppose that $M - \Sigma$ admits a complete, finite-volume hyperbolic metric $W$, in which the total normalized length of the meridians of $\Sigma$ satisfies
\begin{equation*}
L^2 \geq \max \left\{ \frac{2\pi \cdot 6771 \cosh^5(0.6 \epsilon + 0.1475)}{\epsilon^5} + 11.7, \:
\frac{2\pi \cdot 11.35}{\epsilon^{5/2}\log(J)} + 11.7 \right\}.
\end{equation*}
Then $M$ has a hyperbolic metric $V$, in which $\Sigma$ a geodesic link.
Furthermore, there are canonical $J$--bilipschitz inclusions
\[
\varphi \from V^{\geq \epsilon} \hookrightarrow W^{\geq \epsilon/1.2}, 
\qquad
\psi \from W^{\geq \epsilon} \hookrightarrow V^{\geq \epsilon/1.2}.
\]
which are equivariant with respect to the symmetry group of the pair $(M, \Sigma)$. 
\end{theorem}

We can now use Theorems~\ref{Thm:FillingApproximation} and~\ref{Thm:EffectiveFillFiniteVolume} to prove  \refthm{EffectiveFillingTame}. The proof involves chasing the left half of the diagram in \reffig{FillingApproxOutline}. Starting from the hyperbolic structure $Z$ on the complement of $\Sigma$, we will consider a circle-packed approximating manifold $W_n$, the finite-volume manifold $DD(W_n^\circ)$, its filling $DD(V_n^\circ)$, and the scooped submanifold $V_n^\circ$ that approximates $Y$.

\begin{proof}[Proof of \refthm{EffectiveFillingTame}]
Let $M$ be a tame 3-manifold and $\Sigma \subset M$ a link such that $M - \Sigma$ admits a hyperbolic structure $Z$. If $\vol(Z) < \infty$, then the desired conclusion is covered by \refthm{EffectiveFillFiniteVolume}, substituting $V = Y$ and $W=Z$. For the rest of the proof, we assume $\vol(Z) = \infty$.

Let $\ss = (s_1, \dots, s_k)$ be a tuple of slopes on the cusps of $Z$, with one slope for the meridian of each component of $\Sigma$. We may think of each $s_i$ as a slope on a torus of $\calN(\Sigma)$.  Recall our hypothesis on the total normalized length of the meridians:
\begin{equation}\label{Eqn:L2HypothFilling}
L(\ss)^2 > 4 \, \max \left\{ \frac{2\pi \cdot 6771 \cosh^5(0.6 \epsilon + 0.1475)}{\epsilon^5} + 11.7, \:
\frac{2\pi \cdot 11.35}{\epsilon^{5/2}\log(J)} + 11.7 \right\}.
\end{equation}
Since $\epsilon \leq \log 3$, and since $\cosh^5(x) \geq 1$ for any $x$,  hypothesis \refeqn{L2HypothFilling} is considerably stronger than the normalized length hypothesis of \refthm{FillingApproximation}. Thus $Y = Z(\ss)$ is a hyperbolic structure on $M$ with the same end invariants as those of $Z$.

Fix a basepoint $z \in Z$ so that $\epsilon = 2 \injrad(z)$, and furthermore $z$ lies on an embedded, $\epsilon$--thick horotorus $T^\eta(\sigma)$ about the first component $\sigma \subset \Sigma$. Such a choice of $z \in Z^{\geq \epsilon}$ is possible because $\epsilon$ is a Margulis number for $Z$ by \refthm{MargulisEstimate}; hence the thick part $Z^{> \epsilon}$ is non-empty. 

Let $V_n$ and $W_n$ be the sequences of geometrically finite manifolds constructed in \refthm{FillingApproximation}. In particular, the conformal boundaries of each $V_n$ and each $W_n$ admit the same circle packing $C_n$. 
Let $V_n^\circ \subset V_n$ and $W_n^\circ \subset W_n$ be the scooped submanifolds defined by $C_n$, as in \refdef{DoubleDouble}. We also have strong limits $\rho_n \to \rho$ and $\xi_n \to \xi$, as described in \refthm{FillingApproximation}.

By \refthm{FillingApproximation}, each approximating manifold $W_n$ has pared manifold $(M - \calN(\Sigma), P \cup \bdy \calN(\Sigma))$. Thus each $W_n$ has a rank-2 horocusp $H_i^n$ corresponding to each component $\sigma_i \subset \Sigma$,
with meridian slope $s_i^n$. We may assume $H_i^n \subset W_n^\circ$, after shrinking the horocusp as needed. For each $i$, the normalized length $L(s_i^n)$ converges to the normalized length $L(s_i)$ measured in $Z$. Write $\ss^n = (s_1^n, \ldots, s_k^n)$ for the tuple of slopes in $W_n$ representing the meridians of $\Sigma$. Since $L(s_i^n) \to L(s_i)$ as $n \to \infty$, it follows that for $n \gg 0$, the total normalized length $L(\ss^n)$ must also satisfy the lower bound \refeqn{L2HypothFilling}.

After doubling $W_n^\circ$ twice, as in \refdef{DoubleDouble}, we obtain a finite-volume hyperbolic manifold $DD(W_n^\circ)$. This manifold contains the disjoint union of four isometric copies of $H_i^n$ for each $i$. Then the tuple of slopes $\ss^n$ is also double-doubled to become a tuple of slopes $DD(\ss^n)$ on the cusps of $DD(W_n^\circ)$.
Because each slope $s_i^n$ appears four times in $DD(\ss^n)$,  the total normalized length of the meridians of $DD(H^n)$ satisfies
\[
L^2 >  \max \left\{ \frac{2\pi \cdot 6771 \cosh^5(0.6 \epsilon + 0.1475)}{\epsilon^5} + 11.7, \:
\frac{2\pi \cdot 11.35}{\epsilon^{5/2}\log(J)} + 11.7 \right\}.
\]
Thus, for $n \gg 0$, \refthm{EffectiveFillFiniteVolume} enables us to fill $DD(W_n^\circ)$ along the tuple of meridians of $\ss^n$ and obtain a hyperbolic 3-manifold $U_n$, in which the union of cores of the filled solid tori is a geodesic link $\Upsilon_n$.  By \refthm{FillingApproximation}, we have $U_n = DD(V_n^\circ)$ and $\Upsilon_n = DD(\Sigma_n)$ for a geodesic link $\Sigma_n \subset V_n^\circ$.

Now, \refthm{EffectiveFillFiniteVolume} says that there are $J$--bilipschitz inclusions
\[
\varphi_n \from U_n^{\geq \epsilon} \to DD(W_n^\circ)^{\geq \epsilon/1.2}, \qquad \psi_n \from DD(W_n^\circ)^{\geq \epsilon} \to U_n^{\geq \epsilon/1.2},
\]
which are equivariant with respect to the $\ZZ_2 \times \ZZ_2$ group of symmetries of the pair $(U_n, \Upsilon_n)$. Since $(V_n^\circ, \Sigma_n)$ is a fundamental domain for this group action, $\varphi_n$ and $\psi_n$ restrict to $J$--bilipschitz inclusions 
\[
\varphi_n \from (V_n^\circ)^{\geq \epsilon} \to (W_n^\circ)^{\geq \epsilon/1.2}, \qquad \psi_n \from DD(W_n^\circ)^{\geq \epsilon} \to (V_n^\circ)^{\geq \epsilon/1.2}.
\]
By \refthm{FillingApproximation}, the geodesic links $\Sigma_n \subset V_n^\circ$ converge to a geodesic link $\Sigma_\infty \subset Y$, isotopic to $\Sigma$. After performing this isotopy, we may suppose that $\Sigma \subset Y$ is a geodesic link.

In preparation for \refthm{ConstructBilip}, we choose appropriate basepoints for our geometric limits. 
Recall that we have picked a component $\sigma \subset \Sigma$, and chosen a basepoint $z \in Z$ so that $2 \injrad(z) = \epsilon$, and furthermore $z$ lies on an embedded $\epsilon$--thick horotorus $T^\epsilon(\sigma)$. In a similar fashion, we choose a basepoint $y \in Y$
so that $2 \injrad(y) = \eta \in (\epsilon, 2\epsilon)$, and furthermore $y$ lies on an embedded, $\eta$--thick equidistant torus $T^\eta(\sigma)$ about the same component $\sigma \subset \Sigma$. Such a choice of $y \in Y^{> \epsilon}$ is possible because $\epsilon$ is a Margulis number for $N$; hence the thick part $N^{> \epsilon}$ is non-empty.

\refthm{FillingApproximation} says that for $n \gg 0$, there exist choices of basepoints 
$v_n \in V_n^\circ$ and $w_n \in W_n^\circ$, such that $(V_n^\circ, v_n) \to (Y, y)$ and $(W_n^\circ, w_n) \to (Z, z)$.
The convergence of injectivity radii in a geometric limit implies that $2 \injrad(v_n) \in (\epsilon, 2 \epsilon)$ for large $n$. Similarly,
$2 \injrad(w_n) \to \epsilon$ as $n \to \infty$.

We are now ready to 
construct the $J$--bilipschitz inclusion $\varphi \from Y^{\geq \epsilon} \hookrightarrow Z^{\geq \epsilon/1.2}$, using \refthm{ConstructBilip}. 
We have geometrically convergent sequences $(V_n^\circ, v_n) \to (Y,y)$ and $(W_n, w_n) \to (Z, z)$. We have $y \in Y^{> \epsilon}$ and  $v_n \in (V_n^\circ)^{\geq \epsilon}$ for large $n$, as required. For large $n$, we have a $J$--bilipschitz inclusion $\varphi_n \from (V_n^{\circ})^{\geq \epsilon} \to ( W_n^{\circ})^{\geq \epsilon /1.2}$. Furthermore, $d(\varphi_n(v_n), w_n)$ is uniformly bounded, by exactly the same argument as in the end of the proof of \refthm{EffectiveDrillingTame}. (Essentially, this follows because injectivity radii are well-behaved under both geometric limits and bilipschitz maps.)
Thus  \refthm{ConstructBilip} gives a $J$--bilipschitz inclusion $\varphi \from Y^{\geq \epsilon} \hookrightarrow Z^{\geq \epsilon/1.2}$.

The reverse inclusion $\psi \from Z^{\geq \epsilon} \hookrightarrow Y^{\geq \epsilon/1.2}$ is constructed in exactly the same way, tracing points backwards to check the hypotheses of  \refthm{ConstructBilip}. 
\end{proof}

\section{Short geodesics in infinite-volume manifolds}\label{Sec:ShortGeodesic}

The main results of this section are \refthm{ShortDrillingTame} and \refthm{ShortFillingTame}, which bound 
the change in complex length of a short geodesic under drilling and filling, respectively. 
Corollaries of those results include \refthm{LenBoundDownInfinite} and \refcor{LenBoundUpInfinite}, where the functions that estimate the change in complex length are replaced by constants.

As in \refsec{Drilling}, the proof of \refthm{ShortDrillingTame} combines an approximation result (\refthm{DrillingApproximation}) with a previously proved theorem that works in finite volume (\refthm{ShortDrillingFinVolume}). Similarly, the proof of \refthm{ShortFillingTame} combines an approximation result (\refthm{FillingApproximation}) with a finite-volume theorem (\refthm{ShortFillingFinVolume}).
 
To set up our results, we need to define the functions that will estimate the change in length.

\begin{definition}\label{Def:Haze}
Let $z_0 = \sqrt{\sqrt{5}-2} =  0.5306 \ldots$. For $z \in [ z_0, 1 ]$, define a function
\[ \haze(z) = 3.3957 \, \frac{z(1-z^2)}{1+z^2}. \]
By a derivative computation, 
the function $\haze(z)$ is decreasing and invertible in this domain. Using Cardano's Formula, the inverse function $\haze^{-1}$ can be expressed as follows: 
\[
\haze^{-1}( 3.3957 x) =  \frac{2 \sqrt{ x^2 + 3 }}{3}   \cos \left(   \frac{\pi}{3}  + \frac{1}{3} \tan^{-1} \! \left(\frac{-3  \sqrt{ -3x^4 - 33x^2 + 3 } }{ x^3 + 18x} \right)  \right) - \frac{x}{3}.
\]
Note that $\haze^{-1}$ is defined and monotonically decreasing on $[0, 1.0196]$. Compare \cite[Remark 4.23]{FPS:EffectiveBilipschitz}.
\end{definition}

Here is the geometric meaning of $\haze$ and $\haze^{-1}$. If $(N, \Sigma, g_t)$ is a hyperbolic cone manifold whose singular locus $\Sigma$ has angle $\alpha \in [0, 2\pi]$, the \emph{visual area of $\Sigma$} is defined to be $\calA(\Sigma) = \alpha \len(\Sigma)$. Hodgson and Kerckhoff showed that under appropriate hypotheses, there is an embedded tube about $\Sigma$ of radius $R \geq \arctanh (\haze^{-1}(\calA(\Sigma)))$. See \cite[Theorem 5.6]{hk:univ} and \cite[Corollary 4.25]{FPS:EffectiveBilipschitz} for details. In turn, the radius of this tube is used to control a number of geometric quantities through the cone deformation \cite[Sections 5--7]{FPS:EffectiveBilipschitz}. One of those quantities is the complex length of a non-singular closed geodesic, which we seek to control here.

\begin{definition}\label{Def:ComplexLengthDistance}
Let $\gamma$ be a closed geodesic in a hyperbolic 3-manifold $N$.
Then $\gamma$ corresponds to a loxodromic isometry $\varphi=\varphi(\gamma)\in \Isom^+(\HH^3)$. This loxodromic isometry $\varphi$ has an invariant axis in $\HH^3$, which it translates by distance $\len(\gamma)$ and rotates by angle $\tau(\gamma)$. We define the complex length $\calL_N(\gamma) = \calL(\gamma) = \len(\gamma) + i \tau(\gamma)$.
Observe that $i \calL(\gamma)$ lies in the upper half-plane of $\CC$, which we identify with the hyperbolic plane $\HH^2$. 

Given two complex lengths $\calL_Y(\gamma)$, $\calL_Z(\delta)$, we define the \emph{hyperbolic distance} between them to be
\[
\dhyp(\calL_Y(\gamma), \calL_Z(\delta)) = d_{\HH^2} (i\calL_Y(\gamma), \, i \calL_Z(\delta)).
\]
This is closely related to distance in the Teichm\"uller space of the torus, which is isometric to $\HH^2$. See Minsky~\cite[Section 6.2]{minsky:punctured-tori} for details.
\end{definition}

The hyperbolic distance between lengths can be translated into a bound on the real and imaginary parts of length. The following elementary lemma is  \cite[Lemma~7.14]{FPS:EffectiveBilipschitz}. 

\begin{lemma}\label{Lem:KenPropH2}
Let $\calL_Y(\gamma)$ and $\calL_Z(\delta)$ be complex lengths.
Suppose that 
$ \dhyp(\calL_Y(\gamma), \,  \calL_Z(\delta)) \leq K$ for some $K > 0$.
Then the real and imaginary parts of $\calL_Y(\gamma)$ and $\calL_Z(\delta)$ are bounded as follows:
\[
e^{- K } \: \leq \: \frac{ \len_Z(\delta) }{ \len_Y(\gamma)} \: \leq \: e^K ,
\qquad
|\tau_Z(\delta) - \tau_Y(\gamma)| \: \leq \: \sinh ( K  )  \cdot \min \{ \len_Y(\gamma), \, \len_z(\delta) \}.
\]
\end{lemma}

Finally, given a filled manifold $Y$ and a drilled manifold $Z$, we control the hyperbolic distance $ \dhyp(\calL_Y(\gamma), \,  \calL_Z(\gamma))$ via the following function.

\begin{definition}
\label{Def:FDefine}
For $z \in [z_0,1]$ and $\ell \in (0, 0.5085]$, define a function
\[
F(z, \ell) =  \frac{(1+ z^2)}{  z^3 (3-z^2)} \cdot \frac{\ell}{10.667 - 20.977 \ell} \, .
\]
Note that $F$ is positive everywhere on its domain, decreasing in $z$, and increasing in $\ell$. Compare \cite[Definition 7.2]{FPS:EffectiveBilipschitz}.
\end{definition}

\subsection{Short geodesics under drilling}
The first main result of this section controls the complex length of a short geodesic $\gamma$ under the drilling of a geodesic link $\Sigma$.
In the next theorem, $\ell = \len_Y(\Sigma)$ is the length of the geodesic link that we wish to drill, $m= \len_Y(\gamma)$ is the length of the geodesic that we wish to control, and $z = \zmin = \tanh \Rmin$, where $\Rmin$ is the minimum radius of an embedded tube about $\Sigma^+ = \Sigma \cup \gamma$. Then the function $F(z,\ell)$ controls the change in the complex length $\calL(\gamma)$. We will not compute $\Rmin$ or $\zmin$ directly; we will merely estimate $\zmin$ as a function of $\ell$ and $m$.

\begin{theorem}\label{Thm:ShortDrillingTame}
Let $Y$ be a tame hyperbolic 3-manifold.
Let $\Sigma$ be a geodesic link in $Y$, and $\gamma$ a closed geodesic disjoint from $\Sigma$.
Let $\ell = \len_Y(\Sigma)$ and $m = \len_Y(\gamma)$ be the lengths of $\Sigma$ and $\gamma$ in the complete metric on $Y$. 
Suppose $\ell < 0.018375$ and $m < 0.0996 - 1.408\cdot \ell$.
 Let
\[
\zmin = \haze^{-1}(2\pi(4\ell + m + \ERROR)). 
\]
Then $Y - \Sigma$ also admits a complete hyperbolic metric $Z$, with the same end invariants as those of $Y$. The closed curve $\gamma$ is isotopic to a geodesic in this metric. Furthermore, the complex lengths of $\gamma$ in $Y$ and $Z$ are related as follows:
\[
\dhyp(\calL_Y(\gamma), \,  \calL_{Z}(\gamma))
\: \leq \:  4\pi^2 \,F(\zmin, 4\ell) .
\]
\end{theorem}

The proof of \refthm{ShortDrillingTame} relies on the following finite-volume analogue \cite[Theorem 7.19]{FPS:EffectiveBilipschitz}.

\begin{theorem}[Short geodesics under drilling, \cite{FPS:EffectiveBilipschitz}]\label{Thm:ShortDrillingFinVolume}
Let $V$ be a complete, finite volume hyperbolic 3-manifold.
Let $\Sigma$ be a geodesic link in $V$, and $\gamma$ a closed geodesic disjoint from $\Sigma$.
Let $\ell = \len_V(\Sigma)$ and $m = \len_V(\gamma)$ be the lengths of $\Sigma$ and $\gamma$ in the complete metric on $V$. 
Suppose $\ell \leq 0.0735$ and $m \leq 0.0996 - 0.352\cdot \ell$.
Let
\[
z'_{\min} = \haze^{-1}(2\pi(\ell + m + \ERROR)) > 0.6288.
\]
Then $V - \Sigma$ also admits a complete hyperbolic metric $W$, in which $\gamma$ is again isotopic to a geodesic. Furthermore, the complex lengths of $\gamma$ in $V$ and $W$ are related as follows:
\[
\dhyp(\calL_V(\gamma), \,  \calL_{W}(\gamma))
\: \leq \:  4\pi^2 \,F(z'_{\min}, \ell) .
\]
\end{theorem}

\begin{proof}[Proof of \refthm{ShortDrillingTame}]
If $\vol(Y)<\infty$, the desired result already follows from \refthm{ShortDrillingFinVolume}. (Although the definition of $z'_{\min}$ in \refthm{ShortDrillingFinVolume} differs from the definition of $\zmin$ in \refthm{ShortDrillingTame}, the monotonicity of $\haze^{-1}$ implies that $z'_{\min} > \zmin$. Then, the monotonicity of $F$ ensures 
that the conclusion of \refthm{ShortDrillingFinVolume} still applies with $\zmin$ in place of $z'_{\min}$ and $4\ell$ in place of $\ell$.) 
For the rest of the proof, we assume that $\vol(Y)= \infty$. 

We will apply \refthm{DrillingApproximation}. Let $V_n$ and $W_n$ be the sequences of geometrically finite manifolds constructed in that theorem. Let $Z$ be the hyperbolic manifold homeomorphic to $Y-\Sigma$, with the same end invariants as those of $Y$. Then, by \refthm{DrillingApproximation}, the conformal boundaries of each $V_n$ and each $W_n$ admit the same circle packing $C_n$. That theorem also guarantees a strong limit $\rho_n \to \rho$ (where $\rho_n$ is the representation corresponding to $W_n$ and $\rho$ corresponds to $Y$) and a strong limit $\xi_i \to \xi$ (where $\xi_n$ is the representation corresponding to $W_n$ and $\xi$ corresponds to $Z$).

Now, let $\gamma \subset Y$ be a closed geodesic satisfying the length bound of the theorem. By Meyerhoff's theorem \cite[Section 7]{meyerhoff}, we have $\gamma \cap \Sigma = \emptyset$, hence $\Sigma^+ = \Sigma \cup \gamma$ is a geodesic link where each component is shorter than $0.1$.
Then \refprop{GeodesicsConverge} implies that for $n \gg 0$, the approximating manifold $V_n$ contains a geodesic link $\Sigma_n^+ = \Sigma_n \cup \gamma_n$, where the sequence $\{ \Sigma_n^+ \}$ converges to $\Sigma^+$ as $n \to \infty$.
In particular, setting $\ell_n = \len_{V_n}(\Sigma_n)$ and $m_n = \len_{V_n}(\gamma_n)$, we have $\ell_n \to \ell$ and $m_n \to m$. Consequently, for all $n \gg 0$, we have $\ell_n \leq 0.018375$ and $m_n \leq 0.0996 - 1.408\cdot \ell_n$.

For every $n$ where $\Sigma_n^+$ is defined, we have $\Sigma_n^+ \subset CC(V_n) \subset V_n^\circ$. Thus, for all $n \gg 0$, the double-double $DD(V_n^\circ)$ contains the double-double $DD(\Sigma_n)$, a geodesic link of total length $4 \ell_n \leq 0.0735$. Furthermore, by construction, we have
\[
m_n \leq 0.0996 - 1.408\cdot \ell_n = 0.0996 - 0.352\cdot 4 \ell_n.
\]
Thus $DD(V_n^\circ)$ satisfies the hypotheses of \refthm{ShortDrillingFinVolume}. Combining that result with \refthm{DrillingApproximation}, we may drill the link $DD(\Sigma_n)$ and obtain a cusped hyperbolic 3-manifold $DD(W_n^\circ) = DD(V_n^{\circ}) - DD(\Sigma_n)$ containing a closed geodesic isotopic to $\gamma_n$. Furthermore, \refthm{ShortDrillingFinVolume} gives
\[
\dhyp(\calL_{DD(V_n^\circ)}(\gamma_n), \,  \calL_{DD(V_n^\circ)}(\gamma_n))
 \leq   4\pi^2 \,F(z_{\min}^n, 4\ell_n),
 \]
where
\[
z_{\min}^n = \haze^{-1}(2\pi(4\ell_n + m_n + \ERROR)). 
\]

Observe that the isotopy class of $\gamma_n$ in $V_n^{\circ} - \Sigma_n$ contains a representative disjoint from the scooped boundary (the red and blue faces). 
Thus the closed geodesic $\gamma_n$ in the hyperbolic metric on $DD(W_n^\circ) = DD(V_n^{\circ}) - DD(\Sigma_n)$ must be disjoint from the red and blue totally geodesic surfaces that partition the four copies of the fundamental domains $W_n^\circ$. In short, we may take $\gamma_n$ to be a closed geodesic in $W_n^\circ$.
Thus, by the above displayed equation, we also have
\begin{equation}\label{Eqn:LengthChangeApprox}
\dhyp(\calL_{V_n}(\gamma_n), \,  \calL_{W_n}(\gamma_n))
 \leq   4\pi^2 \,F(z_{\min}^n, 4\ell_n),
\end{equation}
Observe that 
\[
\lim_{n\to \infty} z_{\min}^n = \lim_{n \to \infty} \haze^{-1}(2\pi(4\ell_n + m_n + \ERROR)) = \haze^{-1}(2\pi(4\ell + m + \ERROR)) = \zmin.
\]
Since $z_{\min}^n \geq 0.6288$ by \refthm{ShortDrillingFinVolume}, we can substitute $F(z_{\min}^n, 4\ell_n) \leq 0.0174$ in \refeqn{LengthChangeApprox}, hence \reflem{KenPropH2} implies $\len_{W_n}(\gamma_n) < 2 m_n < 0.2$.

Now, recall the strong limit $\xi_n \to \xi$. The closed geodesics $\gamma_n \subset W_n$ have length universally bounded by a constant less than $\log 3$, hence \refprop{GeodesicsConverge}, says that $\gamma_n \subset W_n$ converge to a closed geodesic $\gamma \subset Z$ in the geometric limit. In particular, $\calL_{W_n}(\gamma_n) \to \calL_Z(\gamma)$. Taking limits of the bound in \refeqn{LengthChangeApprox} as $n \to \infty$
gives
\[
\dhyp(\calL_{Y}(\gamma), \,  \calL_{Z}(\gamma))
 \leq   4\pi^2 \,F(\zmin, 4\ell),
\]
as desired.
\end{proof}

We can now derive \refthm{LenBoundDownInfinite}, which was stated in the introduction.

\begin{proof}[Proof of \refthm{LenBoundDownInfinite}]
Let $\ell = \len_Y(\Sigma)$ and $m = \len_Y(\gamma)$, and assume $\max\{4\ell, m\} < 0.0735$. This hypothesis implies $\ell < 0.018375$ and thus $m + 1.408 \ell < 0.0996$, hence the hypotheses of \refthm{ShortDrillingTame} are satisfied. By \reflem{DrillHyperbolic}, $Y - \Sigma$ admits a hyperbolic metric $Z$, with the same end invariants. In addition, the above hypothesis on $\ell$ and $m$, combined with the monotonicity of $\haze^{-1}$, implies
\[
\zmin = \haze^{-1}(2\pi(4\ell+m + \ERROR)) \geq 0.6299.
\]
Now, \refthm{ShortDrillingTame} gives
\[
\dhyp(\calL_{Y}(\gamma), \,  \calL_{Z}(\gamma))
 \leq  4 \pi^2 F(\zmin, 4\ell) 
 \leq 4 \pi^2 F(0.6299, 0.0735)
\leq 0.6827,
\]
where the second inequality uses the monotonicity of $F$ and the third inequality comes from evaluating \refdef{FDefine}.
Finally, 
\reflem{KenPropH2} gives
\[
1.9793^{-1} \leq \frac{\len_{Z}(\gamma)}{\len_{Y}(\gamma)} \leq 1.9793
\qquad \text{and} \qquad
|\tau_{Z}(\gamma) - \tau_{Y}(\gamma) | \leq 0.05417. \qedhere
\]
\end{proof}

\subsection{Short geodesics under filling}
Next, we turn our attention to bounding the length of a short geodesic under filling rather than drilling. The following result is the
filling analogue of \refthm{ShortDrillingTame}.

\begin{theorem}\label{Thm:ShortFillingTame}
Let $V$ be a tame, hyperbolic 3-manifold and $\Sigma$ a geodesic link in $V$. Suppose that $V - \Sigma$ admits a hyperbolic structure $W$ with the same end invariants as those of $V$, such that the total normalized length of the meridians of $\Sigma$ in $W$ satisfies $L^2 > 512$. Let $\gamma \subset W$ be a closed geodesic of length $m = \len_{W}(\gamma) < 0.056$. Define
\[
\zmin = \haze^{-1}\left(\frac{(2\pi)^2}{(L/2)^2 - 14.7} + 2\pi \cdot 1.656 \, m \right).
\]
Then $\gamma$ is isotopic to a closed geodesic in $V$. Furthermore, the complex lengths of $\gamma$ in $V$ and $W$ are related as follows:
\[
\dhyp(\calL_V(\gamma), \,  \calL_{W}(\gamma))
\: \leq \:  4\pi^2 \,F \big( \zmin, \tfrac{2\pi}{(L/2)^2 - 14.7} \big) .
\]
\end{theorem}

Just as with the drilling argument, the proof of \refthm{ShortFillingTame} relies on the following finite-volume analogue \cite[Theorem 7.21]{FPS:EffectiveBilipschitz}.

\begin{theorem}[Short geodesics under filling, \cite{FPS:EffectiveBilipschitz}]\label{Thm:ShortFillingFinVolume}
Let $Y$ be a complete, finite-volume hyperbolic 3-manifold and $\Sigma$ a geodesic link in $Y$. Suppose that $Y - \Sigma$ admits a hyperbolic structure $Z$, such that the total normalized length of the meridians of $\Sigma$ in $Z$ satisfies $L^2 \geq 128$. Let $\gamma \subset Z$ be a closed geodesic of length $m = \len_{Z}(\gamma) \leq 0.056$. Define
\[
z_{\min}' = \haze^{-1}\left(\frac{(2\pi)^2}{L^2 - 14.7} + 2\pi \cdot 1.656 \, m \right) > 0.624.
\]
Then $\gamma$ is isotopic to a closed geodesic in $Y$. Furthermore, the complex lengths of $\gamma$ in $Y$ and $Z$ are related as follows:
\[
\dhyp(\calL_Y(\gamma), \,  \calL_{Z}(\gamma))
\: \leq \:  4\pi^2 \,F \big( z_{\min}' , \tfrac{2\pi}{L^2 - 14.7} \big) .
\]
\end{theorem}

\begin{proof}[Proof of \refthm{ShortFillingTame}]
If $\vol(Y)<\infty$, the desired result already follows from \refthm{ShortFillingFinVolume}. (Although the definition of $z'_{\min}$ in \refthm{ShortFillingFinVolume} differs from the definition of $\zmin$ in \refthm{ShortFillingTame}, the monotonicity of $\haze^{-1}$ and $F$ ensures 
that the conclusion of \refthm{ShortDrillingFinVolume} still applies with $(L/2)^2$ in place of $L^2$.) For the rest of the proof, we assume that $\vol(Y)= \infty$.

Let $V_n$ and $W_n$ be the sequences of geometrically finite manifolds constructed in \refthm{DrillingApproximation}. By that theorem, the conformal boundaries of each $V_n$ and each $W_n$ admit the same circle packing $C_n$. Furthermore, there is
a strong limit $\rho_n \to \rho$ (where $\rho_n$ is the representation corresponding to $W_n$ and $\rho$ corresponds to $Y$) and a strong limit $\xi_i \to \xi$ (where $\xi_n$ is the representation corresponding to $W_n$ and $\xi$ corresponds to $Z$).

Let $\gamma \subset Z$ be a closed geodesic satisfying the length bound of the theorem. Then \refprop{GeodesicsConverge} implies that for $n \gg 0$, the approximating manifold $W_n$ contains a closed geodesic $\gamma_n$, where the sequence $\{ \gamma_n \}$ converges to $\gamma$ as $n \to \infty$.
Consequently, $\calL_{W_n}(\gamma_n) \to \calL_{Z}(\gamma)$. In particular, for all $n \gg 0$, we have $m_n = \len_{W_n}(\gamma_n) \leq 0.056$.

Let $\ss$ be the tuple of slopes in $Z$ corresponding to the meridians of $\Sigma$, and let $\ss^n$ be the tuple of slopes in $W_n$ corresponding to the meridians of $\Sigma$. Then, as in the proof of \refthm{FillingApproximation}, we have $L(\ss^n) \to L(\ss) > 512$ as $n \to \infty$. In the double-doubled manifold $DD(W_n^\circ)$, we obtain a tuple of slopes $DD(\ss^n)$, 
where each coordinate of $\ss^n$ appears four times, once per copy of $W_n^\circ$. Thus, just as in \refeqn{DoubleDoubledNormLength}, we get
\[
\frac{1}{L(DD(\ss^n))^2} = \frac{4}{L(\ss^n)^2} \leq \frac{4}{512} = \frac{1}{128}, 
\]
where the inequality holds for $n \gg 0$. Thus $DD(W_n^\circ)$ satisfies the hypotheses of \refthm{ShortFillingFinVolume}. 
By \refthm{FillingApproximation}, filling $DD(W_n^\circ)$ along the tuple of slopes $DD(\ss^n)$ produces the finite-volume hyperbolic manifold $DD(V_n^\circ)$.

By \refthm{ShortFillingFinVolume}, the closed geodesic $\gamma_n \subset DD(W_n^\circ)$ is isotopic to a geodesic in the filled manifold $DD(V_n^\circ)$. Furthermore, since $\gamma_n$ can be isotoped to be disjoint from the red and blue totally geodesic surfaces that partition the copies of $V_n^\circ$, the geodesic representative of $\gamma_n$ must be entirely contained in one copy of $V_n^\circ$. Applying \refthm{ShortFillingFinVolume} to $DD(W_n^\circ)$ and $DD(V_n^\circ)$, we obtain
\[
\dhyp(\calL_{V_n^\circ}(\gamma_n), \,  \calL_{W_n^\circ}(\gamma_n))
 \leq   4\pi^2 \,F\big( z_{\min}^n, \tfrac{2\pi}{(L(\ss^n)/2)^2 - 14.7} \big),
\]
where
\[
z_{\min}^n 
= \haze^{-1}\left(\tfrac{(2\pi)^2}{L(DD(\ss^n))^2 - 14.7} + 2\pi \cdot 1.656 \, m_n \right)
= \haze^{-1}\left(\tfrac{(2\pi)^2}{(L(\ss^n)/2)^2 - 14.7} + 2\pi \cdot 1.656 \, m_n \right).
\]
Since we are using \refthm{ShortFillingFinVolume} with  $z_{\min}^n \geq 0.624$ and $(L(\ss^n)/2)^2 - 14.7 \geq 113.3$, we can substitute $F(0.642, \frac{2\pi}{113.3}) \leq 0.128$ in the above bound on complex length. Thus \reflem{KenPropH2} implies $\len_{V_n}(\gamma_n) < 1.66 m_n < 0.1$, enabling us to apply \refprop{GeodesicsConverge}.

Now, recall the strong limit $\rho_n \to \rho$. By \refprop{GeodesicsConverge}, the geodesics $\gamma_n \subset V_n$ converge to a geodesic $\gamma \subset Y$. In particular, $\calL_{V_n}(\gamma_n) \to \calL_Y(\gamma)$. Taking limits of the above bound on $\dhyp(\calL_{V_n^\circ}(\gamma_n), \,  \calL_{W_n^\circ}(\gamma_n))$
\[
\dhyp(\calL_{Y}(\gamma), \,  \calL_{Z}(\gamma))
 \leq   4\pi^2 \,F \big( \zmin, \tfrac{2\pi}{(L(\ss^n)/2)^2 - 14.7} \big),
\]
as desired.
\end{proof}

As a corollary of \refthm{ShortFillingTame}, we obtain

\begin{corollary}\label{Cor:LenBoundUpInfinite}
Let $Y$ be a tame hyperbolic 3-manifold and $\Sigma$ a geodesic link in $Y$. Suppose that $Y-\Sigma$ admits a hyperbolic structure $Z$ with the same end invariants as those of $Y$, and such that the total normalized length of the meridians of $\Sigma$ satisfies $L^2 > 512$. Let $\gamma \subset Z$ be a closed geodesic of complex length $\len_{Z} (\gamma)+i\tau_{Z}(\gamma)$, with $\len_{Z}(\gamma) < 0.056$. Then $\gamma$ is isotopic to a closed geodesic in $Y$, and furthermore
\[
1.657^{-1} \leq \frac{\len_{Z}(\gamma)}{\len_{Y}(\gamma)} \leq 1.657
\qquad \text{and} \qquad
|\tau_{Z}(\gamma) - \tau_{Y}(\gamma) | \leq 0.0295.
\]
\end{corollary}

\begin{proof}
The hypotheses of this corollary match those of \refthm{ShortFillingTame}. The assumption $L^2 > 512$ is equivalent to $(L/2)^2 - 14.7 \geq 113.3$, hence
\[
\zmin = \haze^{-1}\left(\frac{(2\pi)^2}{(L/2)^2 - 14.7} + 2\pi \cdot 1.656 \, m \right) \geq \haze^{-1}\left(\frac{(2\pi)^2}{113.3} + 2\pi \cdot 1.656 \cdot 0.056 \right) \geq 0.624. 
\]
Plugging $z = \zmin \geq 0.624$ and $\ell = \tfrac{2\pi}{(L/2)^2 - 14.7} \leq \frac{2\pi}{113.3}$ into 
\refthm{ShortFillingTame}, we obtain
\[
\dhyp(\calL_Y(\gamma), \,  \calL_{Z}(\gamma))
 \leq   4\pi^2 \,F (z, \ell )  
 \leq 4\pi^2 \,F(0.624, \tfrac{2\pi}{113.3})
 \leq 0.5045,
\]
where the second inequality uses the monotonicity of $F$ and the third inequality comes from evaluating \refdef{FDefine}.
Finally, \reflem{KenPropH2} converts the bound on $\dhyp(\calL_{Y}(\gamma), \,  \calL_{Z}(\gamma))$ into the desired upper bounds on the distance between the real and imaginary parts of $\calL_{Y}(\gamma)$ and $\calL_{Z}(\gamma)$.
\end{proof}

\bibliographystyle{amsplain}
\bibliography{biblio}

\end{document}